\documentclass[reqno,a4paper,11pt]{amsart}

\usepackage[utf8]{inputenc}
\usepackage{graphicx}
\usepackage{subfigure}
\usepackage{amsmath, amssymb, amsthm, mathtools,times,comment}
\usepackage[foot]{amsaddr}
\usepackage{enumitem}
\usepackage{mathrsfs}
\usepackage[english]{babel}
\usepackage[T1]{fontenc}
\usepackage{fourier}
\usepackage{fullpage}
\usepackage{xparse}
\usepackage{etoolbox} 
\usepackage{hyperref}
\usepackage[capitalize,nameinlink]{cleveref}
\usepackage{crossreftools}
\usepackage{aligned-overset}
\usepackage{comment}
\usepackage{bbold}
\usepackage{xcolor}
\usepackage{tikz}
\usetikzlibrary{3d}
\usetikzlibrary{calc}
\usetikzlibrary{shapes.geometric}
\usepackage{csquotes}

\usepackage[backend=biber,giveninits=true,doi=false,url=false,isbn=false,style=trad-plain,hyperref]{biblatex}
\bibliography{references} 

\DeclareMathOperator{\distance}{dist}
\newcommand{\dist}[2]{\distance\left(#1,#2\right)}  
\newcommand{\closure}[1]{\left[{#1}\right]} 

\newcommand{\subcube}[1]{H_{#1}} 
\newcommand{\intspanproj}{\text{internally spanned \projection}}
\newcommand{\intspanprojs}{\text{internally spanned \projections}}
\newcommand{\projection}{projection}
\newcommand{\projections}{projections}
\newcommand{\intspanset}[1]{\mathcal{H}_{#1}}
\newcommand{\initinfec}[1]{A_{#1}} 
\newcommand{\transpose}[1]{\ensuremath{#1^{\scriptscriptstyle T}}} 

\newcommand{\upboundcubeintspan}[1]{\Phi(#1)} 
\newcommand{\witquad}[0]{witnessing quadruple} 
\newcommand{\cand}[2]{\mathcal C(#1,#2)} 
\DeclareMathOperator{\dime}{dim}
\newcommand{\dimension}[1]{\dime \left(#1\right)}
\newcommand{\upboundcubeintspanconstant}[1]{c_{#1}} 
\newcommand{\optimalj}{j^*}

\newcommand{\probcubeintspan}[1]{\varphi({#1})}
\newcommand{\indicator}[1]{\mathbb{1}_{\set{#1}}}
\newcommand{\union}[6]{U_{#1,#2}\left(#3,#4,#5,#6\right)} 
\newcommand{\ffunc}[6]{f_{#1,#2}\left(#3,#4,#5,#6\right)} 
\newcommand{\ffuncred}[4]{f\left(#1,#2,#3,#4\right)}
\newcommand{\ffuncredj}[1]{F\left(#1\right)} 
\newcommand{\ffuncredi}[3]{\hat{F}_{#1,#2}\left(#3\right)}
\newcommand{\ffuncredl}[1]{\Bar{F}\left(#1\right)}

\newcommand{\indextriplesum}[2]{I_{#1,#2}}

\newcommand{\criticaldimension}{t_{*}} 
\newcommand{\triple}[2]{\hat I_{#1,#2}}

\newcommand{\vertexvec}[1]{#1}
\newcommand{\vertexvecn}[1]{\Vec{#1}}
\newcommand{\extset}[5]{\mathcal{E}\left(#1,#2,#3,#4,#5\right)} 
\newcommand{\extnum}[3]{E\left(#1,#2,#3\right)} 
\newcommand{\extnumindex}[5]{E_{#4}^{(#5)}\left(#1,#2,#3\right)} 
\newcommand{\criticalsizesequential}{\ell^*} 
\newcommand{\seqspannum}[1]{X_{#1}} 
\newcommand{\seqspanconst}[1]{C_{#1}} 
\newcommand{\feasible}[0]{sequentially spanning} 
\newcommand{\seq}[2]{\overline{#1}_{#2}} 
\newcommand{\seqsize}[0]{length}
\newcommand{\seqset}[1]{S_{#1}} 
\newcommand{\seqoverlap}[2]{Y_{#1}(#2)} 
\newcommand{\seqoverlaplastindexone}[1]{Y_{#1}^{(1)}} 
\newcommand{\seqoverlaplastindextwo}[1]{Y_{#1}^{(2)}} 
\newcommand{\seqsetwithvertex}[3]{\seqset{#1}(\vertexvec{#2})} 
\newcommand{\auxfunc}[3]{\Psi\left(#1,#2,#3\right)} 

\newcommand{\D}[1]{\Delta_{}} 
\newcommand{\constinlemma}[1]{D_{#1}} 
\newcommand{\specialindex}{j^*} 

\newcommand{\N}{\mathbb{N}}
\newcommand{\expp}[1]{\exp \left(#1\right)} 
\newcommand{\prob}[1]{\mathbb{P}\left[#1\right]} 

\newcommand{\expec}[1]{\mathbb{E}\left[#1\right]} 
\newcommand{\rounddown}[1]{\left\lfloor#1\right\rfloor} 
\newcommand{\roundup}[1]{\left\lceil#1\right\rceil} 

\newcommand{\setbuilder}[2]{\left\{#1 \mid #2\right\}} 
\newcommand{\set}[1]{\left\{#1 \right\}} 
\newcommand{\parens}[1]{\left( #1 \right)} 
\newcommand{\abs}[1]{\left|#1\right|} 
\newcommand{\minp}[2]{\min_{#1} \left\{ #2\right\}} 
\newcommand{\maxp}[2]{\max_{#1} \left\{ #2\right\}} 

\newcommand{\whp}{whp}
\newcommand{\Whp}{Whp}

\newcommand{\smallo}[1]{o\left(#1\right)}

\newcommand{\smallomega}[1]{\omega\left(#1\right)}

\newcommand{\Th}[1]{\Theta\left(#1\right)}

\newcommand{\mycomment}[1]{} 

\newcommand{\crefitemintheorem}[2]{\hyperref[#2]{\namecref{#1}~\labelcref*{#1}~\ref*{#2}}} 

\newtheorem{thm}{Theorem}[section]
\newtheorem{lem}[thm]{Lemma}
\newtheorem{rem}[thm]{Remark}

\newtheorem{definition}[thm]{Definition}
\newtheorem{cor}[thm]{Corollary}
\newtheorem{problem}[thm]{Problem}
\newtheorem{conj}[thm]{Conjecture}
\newtheorem{claim}[thm]{Claim}

\title{Bootstrap percolation on the high-dimensional Hamming graph}
\author{Mihyun Kang, Michael Missethan, Dominik Schmid}
\address{Institute of Discrete Mathematics, Graz University of Technology, Steyrergasse 30, 8010 Graz, Austria}
\email{\{kang,missethan,schmid\}@math.tugraz.at}

\keywords{}

\numberwithin{equation}{section}
\begin{document}
\begin{abstract}
In the random $r$-neighbour bootstrap percolation process on  a graph $G$, a set of initially infected vertices is chosen at random by retaining each vertex of $G$ independently with probability $p\in (0,1)$, and \lq healthy\rq\ vertices get infected in subsequent rounds if they have at least $r$ infected neighbours. A graph $G$ \emph{percolates} if every vertex becomes  eventually infected. A central problem in this process is to determine the critical probability $p_c(G,r)$, at which the probability that $G$ percolates passes through one half. In this paper, we study random $2$-neighbour bootstrap percolation on the $n$-dimensional Hamming graph $\square_{i=1}^n K_k$, which is the graph obtained by taking the Cartesian product of $n$ copies of the complete graph $K_k$ on $k$ vertices. We extend a result of Balogh and Bollob\'{a}s [Bootstrap percolation on the hypercube, Probab. Theory Related Fields. 134 (2006), no. 4, 624–648. MR2214907] about the asymptotic value of the critical probability $p_c(Q^n,2)$ for random $2$-neighbour bootstrap percolation on the $n$-dimensional hypercube $Q^n=\square_{i=1}^n K_2$ to the $n$-dimensional Hamming graph $\square_{i=1}^n K_k$, determining  the asymptotic value of $p_c\left(\square_{i=1}^n K_k,2\right)$, up to multiplicative constants (when $n \rightarrow \infty$), for arbitrary $k \in \mathbb N$ satisfying $2 \leq k\leq 2^{\sqrt{n}}$.
\end{abstract}
\maketitle

\section{Introduction and results}\label{sec:Intro&Main}
\subsection{Motivation}\label{subsec:Motivation}
Bootstrap percolation was introduced in 1979 by Chalupa, Leath, and Reich \cite{ChReLe1979} in the context of magnetic systems. Since then it has found many applications ranging from studying the spread of information through social networks \cite{KeKlTa2015} to modelling collective behaviour \cite{Granovetter1978}. For a survey of applications in physical science we refer the readers to \cite{AdLe2003}. 

Consider the $r$-neighbour bootstrap percolation process described as follows. Given a graph $G$, a set $\initinfec{0} \subseteq V(G)$ of \emph{initially infected} vertices and the so-called \emph{infection parameter} $r \in \mathbb N$, a \emph{healthy} vertex gets infected if it has at least $r$ infected neighbours. Once a vertex has been infected, it remains in this state forever. The process ends if either all vertices have been infected or all healthy vertices have less than $r$ infected neighbours. The set of infected vertices can be thought of growing in rounds: In each round $j\in \mathbb{N}$, we move all healthy vertices to the set $\initinfec{j}$ that have at least $r$ neighbours in $\initinfec{j-1}$, i.e., we let
\begin{align}\label{eq:def:bootstraprounds}
    \initinfec{j}: = \initinfec{j-1} \cup \setbuilder{v \in V(G)}{\left|N(v) \cap \initinfec{j-1}\right| \geq r}.
\end{align}
Note that the percolation process stops (or stabilises) as soon as $\initinfec{j} = \initinfec{j-1}$ for some $j\in \N$. We set 
\begin{align*}
    \closure{A} \coloneqq \bigcup_{j \in \N \cup \{0\}} \initinfec{j}.
\end{align*}
If  $\closure{\initinfec{}} = V(G)$, we say $\initinfec{}$ \emph{spans} $G$; or equivalently $G$ \emph{percolates}; or equivalently   $\initinfec{}$ is a percolating set. 

Choosing the set of initially infected vertices {\em deterministically} leads to interesting extremal questions, e.g., finding a smallest percolating set \cite{BaPe1998,BeKeNo2022,DuNoRo2023,MoNo2018,PrSh2020} or determining the longest time a graph can take to percolate \cite{BePr2015,HaLi2024,Przykucki2012}. Yet, the probably most-studied case is when the set of initially infected vertices is chosen {\em randomly} in such a way that each vertex is initially infected with probability $p \in (0,1)$ independently of the other vertices. A typical question in this context is above which \emph{critical probability} $p_c$ a graph $G$ is likely to percolate. More formally, in so-called \emph{random} $r$-neighbour bootstrap percolation on a graph $G$, the set of initially infected vertices  is a $p$-\emph{random} subset $\boldsymbol{\initinfec{}}_p$ of $V(G)$, which is obtained by retaining each vertex in $V(G)$  independently with probability $p$. We define the \emph{critical probability}  as 
\begin{align*}
    p_c(G, r) \coloneqq \inf \setbuilder{p \in (0,1)}{ \prob{\boldsymbol{\initinfec{}}_p \text{ spans } G }\geq 1/2 }.
\end{align*} 
We aim to determine the asymptotic value of  the critical probability or to bound the width of the \lq critical window\rq\, where the probability that $G$ percolates changes from tending to zero to tending to one as $n$ tends to infinity. 
One of the first results on random $r$-neighbour bootstrap percolation was obtained by Schonmann \cite{Schonmann1992}, who showed that $p_c\parens{\mathbb{Z}^d,r} = 0$ if $r \leq d$ and $p_c\parens{\mathbb{Z}^d,r} = 1$ otherwise. Aizenmann and Lebowitz\cite{AiLe1988} determined $p_c\parens{[n]^d,2}$ up to a constant factor. In subsequent papers, the percolation threshold on $[n]^d$ was extended to different sets of parameters $r,d$ and improved to a sharp threshold  \cite{BaBoDCMo2012,BaBoMo2010,BaBoMo2009,CeCi1999,CeMa2002,Ho2003}. Random $r$-neighbour bootstrap percolation was also studied for other base graphs, e.g., the infinite lattice \cite{GrHo2008,HaTe2024,Ho2003,Schonmann1992},   the binomial random graph $G(n,p)$\cite{AnKo2021,JaLuToVa2012,KaMa2016,Scalia1985},  or various types of trees \cite{BaPePe2006,BiSc2009,BoGuHoJaPr2014,BrSa2015,GuPr2014,Shapira2019}. 

In this paper, we consider \emph{random} $2$-neighbour bootstrap percolation on the high-dimensional \textit{Hamming graph}. For each $n,k\in \mathbb N$ with $k\ge 2$ let $\square_{i=1}^n K_k$ denote the $n$-dimensional Hamming graph which is obtained by taking the $n$-fold Cartesian product of the complete graph $K_k$ -- the formal definition of the Cartesian product of graphs can be found in \Cref{def:Cartesianproduct}. Note that the special case $k=2$ yields the well-studied $n$-dimensional hypercube $Q^n=\{0,1\}^n=\mathbb Z_2^n =[2]^n$, in other words, $Q^n=\square_{i=1}^n K_2$.  When the dimension $n$ is \emph{constant} and the size $k$ of the base graph $K_k$ tends to infinity, random $r$-neighbour bootstrap percolation on $\square_{i=1}^n K_k$ is well studied, e.g., by Gravner, Hoffman, Pfeiffer, and Sivakoff \cite{GrHoPfSi2015}and by Slivken \cite{Sl2017}. However, the case when $n$ tends to infinity and $k\ge 2$ is arbitrary is relatively unknown except for the special case $k=2$. Regarding  random $2$-neighbour bootstrap percolation on $Q^n=\square_{i=1}^n K_2$, Balogh and Bollob\'{a}s\cite{BaBo2006} determined the asymptotic value of the critical probability  up to multiplicative constants:
\begin{align}\label{eq:BaBoHypercuberesult}
    p_c\parens{Q^n,2} = \Th{n^{-2} 2^{-2\sqrt{n}}}. 
\end{align} 
In this paper we generalise the result \eqref{eq:BaBoHypercuberesult} of Balogh and Bollob\'{a}s\cite{BaBo2006} to the $n$-dimensional  Hamming graph $\square_{i=1}^n K_k$ when $n$ tends to infinity \footnote{Throughout the paper, all asymptotics are considered as $n \to \infty$ and we say an event $\mathcal{E}$ holds whp (with high
probability) if the probability of $\mathcal{E}$ tends to one as $n \to \infty$.}. Our results allow the size $k$ of the base graphs $K_k$ to be an arbitrarily large constant and we even allow $k=k(n)$ to be a function in $n$ that grows at most exponentially in $\sqrt{n}$.

\subsection{Main results and key proof techniques}\label{subsec:Main}
For $n,k\in \mathbb N$ with $k\ge 2$, consider  \emph{random} $2$-neighbour bootstrap percolation on the $n$-dimensional Hamming graph $\square_{i=1}^n K_k$, where the set of initially infected vertices is a $p$-random subset $\boldsymbol{\initinfec{}}_p$ of $V\left(\square_{i=1}^n K_k\right)$, meaning $\prob{v \in \boldsymbol{\initinfec{}}_p} = p \in (0,1)$  independently for each vertex $v \in V\left(\square_{i=1}^n K_k\right)$. Set $A_0 := \boldsymbol{\initinfec{}}_p$ and for each $j \in \mathbb N$ let 
\begin{align*}
    \initinfec{j} \coloneqq \initinfec{j-1} \cup \setbuilder{v \in V\left(\square_{i=1}^n K_k\right)}{\left|N(v) \cap \initinfec{j-1}\right| \geq 2} \quad \text{ and } \quad \closure{\boldsymbol{\initinfec{}}_p} := \bigcup_{j \in \N \cup \set{0}}\initinfec{j}.
\end{align*}
Our main result is that for {\em any} $2 \leq k \leq 2^{\sqrt{n}}$ the critical probability satisfies 
$$p_c\left(\square_{i=1}^n K_k, 2\right) = \Th{ n^{-2}k^{-2\sqrt{n}+1}},$$ 
which recovers \eqref{eq:BaBoHypercuberesult} of Balogh and Bollob\'{a}s \cite{BaBo2006} by taking $k=2$. 
More precisely, we determine for which $p$ does $\square_{i=1}^n K_k$ percolate, i.e., $\closure{\boldsymbol{\initinfec{}}_p} = V\left(\square_{i=1}^n K_k\right)$ and for which $p$ it does not, i.e., $V\left(\square_{i=1}^n K_k\right)\setminus \closure{\boldsymbol{\initinfec{}}_p}  \neq \emptyset.$ 
\begin{thm}\label{thm:main:threshold}
Let $n,k\in \mathbb N$ satisfy $2 \leq k \leq 2^{\sqrt{n}}$. Consider random $2$-neighbour bootstrap percolation on the $n$-dimensional Hamming graph $G=\square_{i=1}^n K_k$ in which the set of initially infected vertices  is a $p$-random subset $\boldsymbol{\initinfec{}}_p$ of $V(G)$. Set $p_* := n^{-2} k^{-2\sqrt{n}+1}$ and $ p^* := 200 n^{-2}k^{-2\sqrt{n}+1}$. Then the following hold.
    \begin{enumerate}[label = (\alph*)]
        \item If $p \leq p_*$, then \whp{} $G$ does not percolate.\label{thm:main:threshold:lower}
        \item If $p \geq p^*$, then \whp{} $G$ percolates.\label{thm:main:threshold:upper}
    \end{enumerate}
\end{thm} 

Note that Theorem \ref{thm:main:threshold} implies
\begin{align*}
    n^{-2}k^{-2\sqrt{n}+1} \leq p_c\left(\square_{i=1}^n K_k, 2\right) \leq 200 n^{-2}k^{-2\sqrt{n}+1}
\end{align*}
and determines a threshold for the property that $G$ percolates, up to  multiplicative constant factors. Regarding the case $k=2$, improving upon \eqref{eq:BaBoHypercuberesult} Balogh, Bollob\'{a}s and Morris\cite{BaBoMo2010} determined a sharp threshold (and more):  
\begin{align}\label{eq:BaBoHypercuberesult_sharp}
 \frac{16\lambda}{n^2}\parens{1+\frac{\log n}{\sqrt{n}}}2^{-2\sqrt{n}} \leq p_c(Q^n,2) \leq \frac{16\lambda}{n^2}\parens{1+\frac{ 5 \log^2 n}{\sqrt{n}}}2^{-2\sqrt{n}},
\end{align} where $\lambda \approx 1.166$. 
Although the constants in \Cref{thm:main:threshold} could be improved (especially for $p^*$), different techniques are needed to improve our result to  a sharp threshold. Furthermore, note that the restriction $k \leq 2^{\sqrt{n}}$ on the size of the base graphs $K_k$ is due to technical reasons in the proof of \crefitemintheorem{thm:main:threshold}{thm:main:threshold:lower} and we did not try to optimise this restriction, although the result might still hold for $k$ growing arbitrarily fast in $n$. 
 
The advantages of our proof techniques lie in their universality and applicability to more general classes of high-dimensional product graphs. 
Two key concepts in the standard analysis of bootstrap percolation on high-dimensional graphs are  {\em internally spanned \projections{}} and a {\em critical dimension}. Given a Cartesian product graph $G$,  a \projection{} of $G$ is a subgraph of $G$ that is isomorphic to a {\em lower-dimensional} Cartesian product graph; and a \projection{} is {\em internally spanned} if its induced subgraph percolates. One of the key properties of the Hamming graph is that two \projections{} within distance two together span another \projection{} containing both of them. This enables us to use  {\em \lq hierarchies\rq} to obtain good upper bounds on the probability that a \projection{} is internally spanned and to show that the probability of $G$ percolating is bounded from above by the probability of the appearance of an internally spanned \projection{} of a certain \emph{critical dimension}. To prove \crefitemintheorem{thm:main:threshold}{thm:main:threshold:lower} it then suffices to show that for $p\leq p_*$ there exists no internally spanned \projection{} of this \emph{critical dimension}. We note that in earlier work \cite{BaBo2006,BaBoMo2010} as well as in our work the property that two \projections{} within distance two always span a larger \projection{} was crucial. 

To prove \crefitemintheorem{thm:main:threshold}{thm:main:threshold:upper} we consider a specific type of percolating set, called a {\em \feasible{} set} (see \Cref{def:sequentiallyspanning} for formal definition). The importance of sequentially spanning sets is two-fold: They are {\em minimum} percolating sets; and the majority of all minimum percolating sets comes from sequentially spanning sets. We use the second moment method to show the existence of a sufficiently large \feasible{} set. In comparison to the earlier work  of Balogh and Bollob\'{a}s \cite{BaBo2006} and Balogh, Bollob\'{a}s and Morris\cite{BaBoMo2010}, we incorporate the {\em ordering} of the vertices in  a sequentially spanning set and consider a sequentially spanning {\em sequence} instead (see \Cref{def:seqspanningsequences} for formal definition), which has certain advantages: On the one hand, the correlation between two \feasible{} sequences can be computed quite explicitly, allowing for much more generous applications of the second moment method; and on the other hand, while \feasible{} sets are used in literature in order to derive lower bounds on the probability that a \projection{} is internally spanned, the use of \feasible{} {\em sequences} allows us to avoid this step. 
The concept of \feasible{} sequences transfers quite naturally to other types of Cartesian product graphs, and so we believe that our techniques will be useful for deriving similar bounds on more general classes of product graphs.

\subsection{Outline of the paper}\label{subsec:Outline}
In \Cref{sec:Preliminaries} we introduce notations and concepts needed for  bootstrap percolation on the Hamming graph. 
\Cref{sec:LowerBound} is dedicated to the proof of \crefitemintheorem{thm:main:threshold}{thm:main:threshold:lower}, and  more technical proofs of auxiliary results of \Cref{sec:LowerBound}   are deferred to \Cref{sec:proof:lower}. 
Similarly, we prove \crefitemintheorem{thm:main:threshold}{thm:main:threshold:upper} in \Cref{sec:UpperBound} and other technical results in \Cref{sec:proof:upper}. 
We discuss open problems and future research directions in \Cref{sec:Discussion}. 

\section{Preliminaries}\label{sec:Preliminaries}
In this section we first introduce necessary notations and useful known results. We then provide some intuition on the deterministic $2$-neighbour bootstrap percolation process on the Hamming graph  and characterise percolating or non-percolating sets. In \Cref{subsec:bootstrapprozess}, we take a closer look at the 
percolation process and set up the scene for what is known as the \emph{hierarchy} technique, which was introduced by Holroyd \cite{Ho2003} and since then has become a fundamental tool to derive lower bounds on bootstrap percolation thresholds. 

\subsection{Notation and terminology}\label{subsec:Notation}
We define $\N_0 \coloneqq \N\cup \set{0}$. Given a set $X$ and $\ell \in \N$, we denote by $\binom{X}{\ell}$ the set of all $\ell$-element subsets of $X$ and define $[\ell]:=\{1,2,\ldots, \ell\}$. 

Given a graph $G$, the \emph{distance} between two different vertices in $G$, denoted by $\dist{x}{y}$,  is defined as the length of a shortest path between them. The distance between two sets of vertices $A,A'\subseteq V(G)$ is defined as the minimum distance between all pairs of vertices from the two sets:
\begin{align*}
  \dist{A}{A'} =  \dist{A'}{A} \coloneqq \min_{u \in A} \min_{v \in A'} \dist{u}{v}.  
\end{align*}

Now we define the Cartesian product of graphs and introduce important concepts. 
\begin{definition}\label{def:Cartesianproduct}
Given graphs $G^{(1)}=\left(V^{(1)}, E^{(1)}\right), \ldots, G^{(n)}=\left(V^{(n)}, E^{(n)}\right)$, we define the {\bf Cartesian product} of $G^{(1)}, \ldots, G^{(n)}$, denoted by $G=G^{(1)}\square\cdots\square G^{(n)}$ or $G=\square_{i=1}^{n}G^{(i)}$, as the graph $G=(V,E)$ with vertex and edge set
\begin{align*}
V &=V(G) = V^{(1)}\times \cdots \times V^{(n)}:=\setbuilder{\vertexvec{v} = (v_1, \ldots, v_n)}{\forall i\in [n]: v_i\in V^{(i)}}\\
E &= E(G) := \setbuilder{\{\vertexvec{v},\vertexvec{w}\} \in \binom{V}{2} }{\exists j\in[n]: \Big(\set{v_j,w_j} \in E^{(j)} \Big) \wedge \Big(\forall i\in[n]\setminus \{j\}: v_i=w_i\Big)}.
\end{align*}
\end{definition}
Throughout the paper we will call $G=\square_{i=1}^{n}G^{(i)}$ the {\em product graph} instead of the Cartesian product graph, unless specified otherwise, and the graphs $G^{(1)}, \ldots, G^{(n)}$ the \emph{base graphs} of $G$.  A base graph is called \emph{trivial}, if it consists of a single vertex.

\begin{definition}\label{def:projection}
Given a product graph $G=\square_{i=1}^{n}G^{(i)}$ we call a subgraph $H\subseteq G$ a {\bf \projection{}} of $G$ if $\subcube{}$ can be written as $\subcube{}=\square_{i=1}^{n}\subcube{}^{(i)}$ where for each $i\in[n]$ we have either $\subcube{}^{(i)}=G^{(i)}$ or $\subcube{}^{(i)}=\left\{v_i\right\}$ for some vertex $v_i\in V^{(i)}$. The {\bf dimension} of $H$  is defined as the number of non-trivial base graphs $\subcube{}^{(i)}$ in $H$.
\end{definition}
In order to keep track of the dimension of a \projection{}, we will often write $\subcube{\ell}$ for a projection of dimension $\ell \in \N_0$. With slight abuse of notation, we will sometimes view a \projection{} not as a subgraph of $G$, but just as a subset of $V(G)$. Similarly, when we say a subset  $U\subseteq V(G)$ is a \projection{},  we mean that the subgraph of $G$ induced by $U$ is a \projection{}. 

Next we introduce some necessary notations for the bootstrap percolation process. 
\begin{definition}\label{def:closure}
    Given a graph $G=(V,E)$ and a set $A \subseteq V$, we define the \textbf{closure} of $A$, denoted by $\closure{A}$, as the set of eventually infected vertices in the $2$-neighbour bootstrap percolation process with $A$ being the set of initially infected vertices. More formally, we define
    \begin{align*}
        \closure{A} \coloneqq \bigcup_{j\in \N_0} A_j,
    \end{align*} where $A_0:= A$ and $A_j := A_{j-1} \cup \setbuilder{v \in V}{\left|N(v) \cap A_{j-1}\right| \geq 2}$ for $j \in \N$.
We say 
\begin{enumerate}
    \item[(1)] $A $ is \textbf{closed}, if $ \closure{A} = A $;
    \item[(2)] a closed set  $U\subseteq V$ is called \textbf{internally spanned} by  $\initinfec{}$, if $U = \closure{\initinfec{} \cap U}$.
\end{enumerate}
The collection of all closed subsets of $V$ that are internally spanned by $\initinfec{}$ is denoted by 
    \begin{align*}
        \intspanset{\initinfec{}} \coloneqq \setbuilder{U \subseteq V}{U = \closure{\initinfec{} \cap U}}.
    \end{align*}
Furthermore we say  
\begin{enumerate}
    \item[(3)] two (not-necessarily disjoint) closed sets $U,W\subseteq V$ are \textbf{disjointly internally spanned} by $\initinfec{}$, if there exist two disjoint subsets $\initinfec{U} ,\initinfec{W} \subseteq \initinfec{}$ that internally span $U$ and $W$, respectively,  i.e., $U = \closure{\initinfec{U} \cap U}$ and $W = \closure{\initinfec{W} \cap W}$.
\end{enumerate} 
In this case, we write 
    \begin{align*}
        \{U,W\} \subseteq_{d} \intspanset{\initinfec{}}.
    \end{align*} 
Analogously  for arbitrary numbers of sets $U_1,\ldots, U_{\ell}\subseteq V$ we say they are  disjointly internally spanned by $A$, if  there exist disjoint subsets $\initinfec{U_1} ,\ldots, \initinfec{U_{\ell}} \subseteq \initinfec{}$ that internally span $U_1,\ldots, U_{\ell}$, respectively.
\end{definition}
A powerful tool in our analysis is the van den Berg-Kesten Lemma. Essentially in the context of bootstrap percolation it implies the following. 

\begin{lem}[{\cite[Theorem 1.6.(iii)]{VaKe1985}}]\label{lem:vandenbergkesten}
    Given a graph $G=(V,E)$ and $p \in (0,1)$ let $\boldsymbol{\initinfec{}}_p$ be a $p$-random subset of $V$, i.e., $\prob{v \in \boldsymbol{\initinfec{}}_p} = p$ for each $v \in V$ independently. For any $U,W \subseteq V$, we have
    \begin{align*}
        \prob{\{U,W\} \subseteq_{d} \intspanset{\boldsymbol{\initinfec{}}_p} } \leq \prob{U \in \intspanset{\boldsymbol{\initinfec{}}_p}} \cdot \prob{W \in \intspanset{\boldsymbol{\initinfec{}}_p}}.
    \end{align*}
\end{lem}
We will use the following version of the second moment method, due to Alon and Spencer \cite{AlSp2008}.
\begin{lem}[\cite{AlSp2008}, Corollary 4.3.4]\label{lem:secondmoment}
    Let $X = \sum_{i=1}^n X_i$ be a sum of indicator random variables for the events $(\mathcal E_i)_{i=1}^n$. For indices $i,j \in [n]$ write $i \sim j$, if the events $\mathcal E_i$ and $\mathcal E_j$ are dependent. If 
    \begin{align*}
        \expec{X} \rightarrow \infty \quad \text{ and } \quad \sum_{i \sim j} \prob{\mathcal E_i \cap \mathcal E_j} = \smallo{\expec{X}^2} \quad \text{ as } n\rightarrow \infty, 
    \end{align*}
    then $\prob{X \geq 1} \rightarrow 1$ as $n \rightarrow \infty$.
\end{lem}

\subsection{Deterministic bootstrap percolation on the Hamming graph}\label{subsec:determ.bootstrap}
In this section we collect properties of closed sets in the Hamming graph and how they interact with each other. Note that the results presented in this section are already known and can be found e.g., in \cite{Sl2017}, but for the sake of completeness we include their proofs in  Appendix A. 
Throughout this section and the next section we consider the $n$-dimensional Hamming graph 
$$G = (V,E) = \square_{i=1}^n K_k,$$ 
where $n,k\in \mathbb N$ satisfying $k\ge 2$.
We start with a simple but important observation about projections.

\begin{lem}\label{lem:projclosed}
Any projection of $G$ is closed.   
\end{lem}

Next we state the most important result in this section, which tells us how \projections{} within small distance interact in the $2$-neighbour bootstrap percolation process. 
\begin{lem}\label{lem:projections_span_projection}
    Let $\subcube{}',\subcube{}''$ be two \projections{} of $G$ satisfying $\dist{\subcube{}'}{\subcube{}''} \leq 2$. Then, there exists a \projection{} $\subcube{}$ of $G$ satisfying $\dimension{\subcube{}} \leq \dimension{\subcube{}'}+\dimension{\subcube{}''}+\dist{\subcube{}'}{\subcube{}''}$ and
    \begin{align*}
        \closure{\subcube{}' \cup \subcube{}''} = \subcube{}.
    \end{align*}
\end{lem}
\Cref{lem:projections_span_projection} tells us that two \projections{} within distance two span another larger \projection{} that contains both of them.  
\begin{rem}\label{rem:projections_at_dist_three}
Note that if two \projections{} $\subcube{}'$ and $\subcube{}''$ lie at distance at least three, they have no common neighbours and thus do not interact under $2$-neighbour bootstrap percolation, i.e., 
\begin{align*}
    \closure{\subcube{}' \cup \subcube{}''} = \subcube{}' \cup \subcube{}''.
\end{align*}
\end{rem}

By \Cref{lem:projections_span_projection} we know that we can generate larger and larger \projections{} by adding vertices that lie at distance two of the previous \projection{}. The set of vertices used in this construction is very useful in the $2$-neighbour bootstrap percolation process and we define it more precisely in the following.  

\begin{definition}\label{def:sequentiallyspanning}
    A set $U  \subseteq V$  is  called \textbf{sequentially spanning}, if there is an ordering $(\vertexvec{v}_0,\ldots, \vertexvec{v}_{|U|-1})$ of its vertices such that 
    \begin{align}
        \dist{\closure{ \{\vertexvec{v}_{0},\ldots,\vertexvec{v}_{i-1} \}} }{\vertexvec{v}{_{i}}} =  2  \quad \text{ for all } i \in [|U|-1].\label{Prop:sequentiallyspanning}
    \end{align}   
\end{definition}
In words, Property \eqref{Prop:sequentiallyspanning} guarantees that each vertex in the ordering lies exactly at distance two from the \projection{} spanned by the previous vertices. 
We will see  in \Cref{lem:probintspanupperbound} that a majority of all percolating sets in $G$ are sequentially spanning. The next lemma tells us that sequentially spanning sets always span large \projections{}.  

\begin{figure}
    \centering
    \begin{tikzpicture}
\definecolor{myred}{rgb}{1,0.2,0.3}
\definecolor{mycand}{rgb}{1,1,0}
\tikzstyle{vertex}= [circle,draw, fill = white, inner sep=3pt]
\tikzstyle{infvertex} = [circle,draw, fill = myred, inner sep=3pt]
\tikzstyle{candvertex} = [regular polygon, regular polygon sides=4,draw, fill = black, inner sep=3pt]
\def\lengths{2};
\def\lengthl{4}
\coordinate (anchor1) at (1.5,1.5);
\coordinate (anchor2) at (0,0);
\coordinate (anchor3) at (10,1.5);
\coordinate (anchor4) at ($(anchor3) +(anchor2)-(anchor1) $);

\node[vertex] (v1) at (anchor1) {};

\node[vertex] (v2) at ($(anchor1) +(\lengths,0)$){};
\node[vertex] (v3) at ($(anchor1) +(0,\lengths)$){};
\node[vertex] (v4) at ($(anchor1) +(\lengths,\lengths)$){};

\node[vertex] (v5) at ($(anchor1) +(\lengths,0)+(0.5*\lengths,0.5*\lengths)$){};
\node[vertex] (v6) at ($(anchor1) +(0,\lengths)+(0.5*\lengths,0.5*\lengths)$){};
\node[vertex] (v7) at ($(anchor1) +(\lengths,\lengths)+(0.5*\lengths,0.5*\lengths)$){};
 
\filldraw[fill = gray,opacity = 0.5] (v1.center)--(v2.center)--(v4.center)--(v3.center)--cycle;
\node[vertex] at (v2) {};
\node[vertex] at (v3) {};
\node[infvertex] at(v1) {};
\node[infvertex] at(v4) {};
\draw[-] (v5)--(v7)--(v6);
\draw[-] (v2)--(v5);
\draw[-] (v3)--(v6);
\draw[-] (v4)--(v7);

\node[vertex] (w1) at (anchor2) {};

\node[vertex] (w2) at ($(anchor2) +(\lengthl,0)$){};
\node[vertex] (w3) at ($(anchor2) +(0,\lengthl)$){};
\node[vertex] (w4) at ($(anchor2) +(\lengthl,\lengthl)$){};

\node[vertex] (w5) at ($(anchor2) +(\lengthl,0)+(0.5*\lengthl,0.5*\lengthl)$){};
\node[vertex] (w6) at ($(anchor2) +(0,\lengthl)+(0.5*\lengthl,0.5*\lengthl)$){};
\node[vertex] (w7) at ($(anchor2) +(\lengthl,\lengthl)+(0.5*\lengthl,0.5*\lengthl)$){};
 
\draw[-] (w1)--(w2)--(w4)--(w3)--(w1);
\draw[-] (w5)--(w7)--(w6);
\draw[-] (w2)--(w5);
\draw[-] (w3)--(w6);
\draw[-] (w4)--(w7);
\node[vertex]  at (v7){};

\foreach \i in {1,...,7}
\draw[-] (v\i) --(w\i);

\coordinate (anchor1) at (anchor3);
\coordinate (anchor2) at (anchor4);

\node[vertex] (v1) at (anchor1) {};

\node[vertex] (v2) at ($(anchor1) +(\lengths,0)$){};
\node[vertex] (v3) at ($(anchor1) +(0,\lengths)$){};
\node[vertex] (v4) at ($(anchor1) +(\lengths,\lengths)$){};

\node[vertex] (v5) at ($(anchor1) +(\lengths,0)+(0.5*\lengths,0.5*\lengths)$){};
\node[vertex] (v6) at ($(anchor1) +(0,\lengths)+(0.5*\lengths,0.5*\lengths)$){};
\node[vertex] (v7) at ($(anchor1) +(\lengths,\lengths)+(0.5*\lengths,0.5*\lengths)$){};
 
\filldraw[fill = gray,opacity = 0.5] (v1.center)--(v2.center)--(v4.center)--(v3.center)--cycle;
\node[vertex] at (v2) {};
\node[vertex] at (v3) {};
\node[infvertex] at(v1) {};
\node[infvertex] at(v4) {};
\draw[-] (v5)--(v7)--(v6);
\draw[-] (v2)--(v5);
\draw[-] (v3)--(v6);
\draw[-] (v4)--(v7);

\node[vertex] (w1) at (anchor2) {};

\node[vertex] (w2) at ($(anchor2) +(\lengthl,0)$){};
\node[vertex] (w3) at ($(anchor2) +(0,\lengthl)$){};
\node[vertex] (w4) at ($(anchor2) +(\lengthl,\lengthl)$){};

\node[vertex] (w5) at ($(anchor2) +(\lengthl,0)+(0.5*\lengthl,0.5*\lengthl)$){};
\node[vertex] (w6) at ($(anchor2) +(0,\lengthl)+(0.5*\lengthl,0.5*\lengthl)$){};
\node[vertex] (w7) at ($(anchor2) +(\lengthl,\lengthl)+(0.5*\lengthl,0.5*\lengthl)$){};
 
\draw[-] (w1)--(w2)--(w4)--(w3)--(w1);
\draw[-] (w5)--(w7)--(w6);
\draw[-] (w2)--(w5);
\draw[-] (w3)--(w6);
\draw[-] (w4)--(w7);
\node[vertex]  at (v7){};

\foreach \i in {1,...,7}
\draw[-] (v\i) --(w\i);

\node[candvertex] at (w5) {};
\node[candvertex] at (w6) {};
\node[candvertex] at (w7) {};
\draw[->, very thick] (6.5,3)--(8,3) ;

\end{tikzpicture}
    \caption{A \feasible{} set of size two (discs in red) in $\square_{i=1}^4 K_2$ and all possible extensions to a \feasible{} set of size three (squares in black).}
    \label{fig:seqspanseq}
\end{figure}
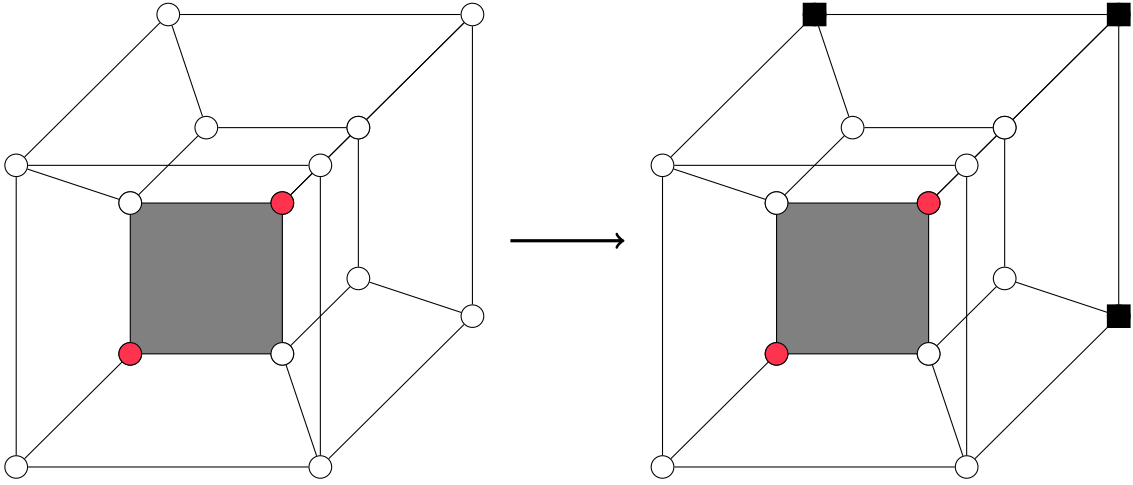

\begin{lem}\label{lem:seqspanningspansprojection}
 Assume that $U\subseteq V$ is a sequentially spanning set. Then there exists a projection $\subcube{}$ of $G$ such that 
    \begin{align*}
\dimension{\subcube{}} = 2(|U|-1) \quad\quad  \text{ and }\quad\quad \closure{U} = \subcube{}.
    \end{align*} 
\end{lem}

Next we discuss an alternative way of looking at the bootstrap percolation process, using projections. 
\subsection{A different view on the bootstrap percolation process}\label{subsec:bootstrapprozess}
Given a set $\initinfec{} \subseteq V$, consider the collection $\intspanset{\initinfec{}}$ of all subsets that are internally spanned by $\initinfec{}$, as in \Cref{def:closure}. For each  $\subcube{} \in \intspanset{\initinfec{}}$  (so $H = \closure{\initinfec{} \cap H}$) we can think of the set $\initinfec{} \cap \subcube{}$ as a set $S_0 = \set{\set{\vertexvec{a}_{0}},\ldots,\set{\vertexvec{a}_{s}}}$ of {\em $0$-dimensional \projections{}}. As each of these \projections{} consists of a single \emph{infected} vertex, it is clear that they are \emph{disjointly internally spanned}. By \Cref{rem:projections_at_dist_three}, \projections{} at distance at least three do not interact at all, so it follows from $\closure{S_0}  = \subcube{}$ that at least two of those \projections{}, w.l.o.g. $\vertexvec{a}_{0}$ and $ \vertexvec{a}_{1}$, are within distance two of each other. By \Cref{lem:projections_span_projection}, $\vertexvec{a}_{0} \cup \vertexvec{a}_{1}$ spans a \projection{} that is by construction disjointly internally spanned from all other \projections{} $\vertexvec{a}_{2}, \ldots, \vertexvec{a}_{s}$. Thus, we obtain a set $S_1 = \set{\closure{\vertexvec{a}_{0} \cup \vertexvec{a}_{1}},\set{\vertexvec{a}_{2}}, \ldots, \set{\vertexvec{a}_{s}}}$ consisting of $|S_1| = |S_0|-1$ disjointly internally spanned \projections{}.
Note that the set $S_1$ is not uniquely determined, as it depends on the choice of \projections{} within distance two. Continuing this process we sequentially obtain sets $S_i$ of disjointly internally spanned \projections{} until at the final step $S_{s} = \{\subcube{}\}$. More formally, we obtain the following. 

\begin{lem}\label{lem:bootstrapprozess}
Let $\subcube{} \in \intspanset{\initinfec{}}$ and set $s \coloneqq \abs{\initinfec{} \cap \subcube{}} -1$. Then, there exist families $S_0,\ldots,S_{s}$ of subsets of $V$ such that the following hold.
    \begin{enumerate}[label = (\arabic*)]
        \item $S_0 = \setbuilder{\{\vertexvec{v}\}}{\vertexvec{v} \in \initinfec{} \cap \subcube{}}$;
        \item for all $ 0 \leq i \leq s-1,$ there exist $H',H'' \in S_i \text{ such that } S_{i+1} = S_i \cup \left\{ \closure{H' \cup H''} \right\} \setminus \{H',H''\}$;
        \item for all $ 0 \leq i \leq s$, the elements of $S_i$ are disjointly internally spanned  \projections{} of $G$; 
        \item $S_{s} = \{\subcube{}\}$.    
\end{enumerate}
\end{lem}

Now given $\subcube{} \in \intspanset{\initinfec{}}$ and any integer $1 \leq t \leq \dimension{\subcube{}}$ we consider the first time a \projection{} of dimension at least $t$ appears in this process. Looking at the two \projections{} that were joined in this step, \Cref{lem:bootstrapprozess} immediately implies the following key result.

\begin{lem}\label{lem:wittnessingquad}
    Let $\subcube{} \in \intspanset{\initinfec{}}$. Then for every $1 \leq t \leq \dimension{\subcube{}}$, there exist indices $(\ell,i,j,d) \in \N_0^4$ with $\ell \geq t, \ j \leq i <t$ and $d \in \{0,1,2\}$ and \projections{} $\subcube{i}, \subcube{j}, \subcube{\ell}$ of respective dimensions $i,j,\ell$ such that 
    \begin{enumerate}[label = (\arabic*)]
        \item\label{Property:witquad:dist} $\dist{\subcube{i}}{\subcube{j}} = d$;
        \item\label{Property:witquad:span} $\closure{\subcube{i} \cup \subcube{j}} = \subcube{\ell}$;
        \item\label{Property:witquad:disjspanned} $\set{\subcube{i}, \subcube{j}} \subseteq_d \intspanset{\initinfec{}}$.
    \end{enumerate}
    We call the quadruple $(\subcube{\ell},\subcube{i}, \subcube{j},d)$ a {\bf \witquad{}} for $\subcube{}$ at threshold $t$. 
\end{lem}

\section{The lower threshold: Proof of \Cref{thm:main:threshold} \ref{thm:main:threshold:lower}}\label{sec:LowerBound}
In the rest of the paper we assume that $n\in \N$ is sufficiently large and let $G= (V,E)  = \square_{i=1}^n K_k$ be the $n$-dimensional Hamming graph with $2 \leq k \leq 2^{\sqrt{n}}$ and let $\boldsymbol{\initinfec{}}_p$ be a $p$-random subset of $V$, as in \Cref{thm:main:threshold}. Let 
\begin{align*}
   \intspanset{} = \intspanset{\boldsymbol{\initinfec{}}_p} \coloneqq \setbuilder{U \subseteq V}{U = \closure{\boldsymbol{\initinfec{}}_p \cap U}}
\end{align*}
be the collection of all subsets of $V$ that are internally spanned by $\boldsymbol{\initinfec{}}_p$. Recall that we write $\subcube{\ell}$ for a \projection{} of dimension $\ell \in \N_0$ in order to keep track of its dimension.

This section is devoted to the proof of \crefitemintheorem{thm:main:threshold}{thm:main:threshold:lower}, along which we will introduce some auxiliary results, but we defer their proofs to \Cref{sec:proof:lower}. So we assume 
\begin{align*}
    p \leq p_* \coloneqq n^{-2}k^{-2\sqrt{n}+1}
\end{align*}
and will show  \begin{align*}
     \prob{G \text{ percolates}} =  \prob{\closure{\boldsymbol{\initinfec{}}_p} = V}=\smallo{1}.
\end{align*}
We first note that in many bootstrap percolation processes the bottleneck for percolation is given by the appearance of a small substructure of the underlying graph, called a \emph{critical droplet}. In our case, this critical droplet is an \intspanproj{} of a specific dimension $\criticaldimension \in [n]$, where $\criticaldimension$ is chosen such that the expected number of \intspanprojs{} is minimal: We call $\criticaldimension$ the {\em critical dimension}.

Assume that $G$ percolates, i.e., $\closure{\boldsymbol{\initinfec{}}_p} = V$. As described in Lemma \ref{lem:bootstrapprozess}, we can view the set of initially infected vertices as a collection of 0-dimensional \projections{}  that sequentially pairwise span larger \projections{} until in the final step two \projections{} together span $G$. In particular, for every threshold dimension $t \in [n]$ we can look at the first time a \projection{} $\subcube{}$ of dimension at least $t$ appears in this process and consider the \projections{} that together span $\subcube{}$. More precisely, it follows from \Cref{lem:wittnessingquad} that there has to exist a \emph{\witquad{}} for $G$ at this threshold and thus we have 
\begin{align}
    \prob{G \text{ percolates}} &\leq \prob{\text{there exists a \witquad{} for $G$ at threshold $t$}}. \label{eq:perc_unionbound_1}
\end{align}%
In order for a quadruple $(\subcube{\ell},\subcube{i}, \subcube{j},d)$ to be a \witquad{}, it has to satisfy Properties \ref{Property:witquad:dist},\ref{Property:witquad:span} and \ref{Property:witquad:disjspanned} of \Cref{lem:wittnessingquad} and the indices $(\ell,i,j,d) \in \N_0^4$ should  fulfill the dimension constraints $\ell \geq t,\ j \leq i  <t$ and $d \in \{0,1,2\}$. However, note that all requirements except Property \ref{Property:witquad:disjspanned} of \Cref{lem:wittnessingquad} are deterministic statements that do not depend on the set of initially infected vertices. 

We call all quadruples $(\subcube{\ell},\subcube{i}, \subcube{j},d)$ that satisfy the dimension constraints as well as Properties \ref{Property:witquad:dist} and \ref{Property:witquad:span} of \Cref{lem:wittnessingquad} {\bf candidate quadruples} and denote by $\cand{G}{t}$ the set of all candidate quadruples. Thus, a candidate quadruple is a \witquad{} if and only if it satisfies Property \ref{Property:witquad:disjspanned} of \Cref{lem:wittnessingquad}. On the other hand, by the van den Berg-Kesten Lemma \ref{lem:vandenbergkesten}, it holds that
\begin{align*}
    \prob{\set{\subcube{i},\subcube{j}}\subseteq_d \intspanset{}} \leq \prob{\subcube{i} \in \intspanset{}} \cdot \prob{\subcube{j} \in \intspanset{}}.
\end{align*} 
Hence, for any candidate quadruple $(\subcube{\ell},\subcube{i}, \subcube{j},d) \in \cand{G}{t}$ we have 
\begin{align}
    \prob{(\subcube{\ell},\subcube{i}, \subcube{j},d) \text{ is a \witquad{}}} \leq \prob{\subcube{i} \in \intspanset{}} \cdot \prob{\subcube{j} \in \intspanset{}}. \label{eq:perc_unionbound_2}
\end{align} 
Putting \eqref{eq:perc_unionbound_1}  and \eqref{eq:perc_unionbound_2} together and taking a union bound over all candidate quadruples  we obtain the following bound for all $t \in [n]$:
\begin{align}
    \prob{G \text{ percolates}} &\leq \prob{\text{there exists a \witquad{} for $G$ at threshold $t$}} \notag \\
    &\leq \sum_{(\subcube{\ell},\subcube{i}, \subcube{j},d) \in \cand{G}{t}} \prob{\subcube{i} \in \intspanset{}} \cdot \prob{\subcube{j} \in \intspanset{}}. \label{eq:perc_unionbound}
\end{align}

Due to the symmetry of $G$, we have that for any \projection{} $\subcube{}$ of $G$, the probability $ \prob{\subcube{} \in \intspanset{}}$ solely depends on $\dimension{\subcube{}}$ (and does not depend on its structure). Thus, for each  $\ell \in [n]\cup\set{0}$ and each \projection{} $\subcube{\ell}$ of dimension $\ell$, we define
\begin{align}
    \probcubeintspan{\ell} \coloneqq \prob{\subcube{\ell} \in \intspanset{}}.\label{eq:def:probcupeintspan}
\end{align}

Given indices $(\ell,i,j,d) \in \N_0^4$ let us denote by
\begin{align}
    \union{G}{t}{\ell}{i}{j}{d} \coloneqq \left| \setbuilder{\parens{\subcube{\ell'},\subcube{i'}, \subcube{j'},d'} \in \cand{G}{t}}{\parens{\ell' =\ell} \land \parens{i'=i}\land \parens{j'=j} \land \parens{d'=d}}\right|\label{eq:def:candidatequad}
\end{align} the number of candidate quadruples in $G$ at threshold $t$ with fixed indices $(\ell,i,j,d)$. Let us determine for which sets of indices this function is actually non-zero. From \Cref{lem:bootstrapprozess} it follows that $\ell \geq t$, $j \leq i <t$ and $d \in \{0,1,2\}$. Also, because the projections $\subcube{i}$ and $\subcube{j}$ have to lie at distance $d$ it is not possible that $i = \ell-1$ and $d =2$, yielding the condition $i+d<\ell$. Furthermore, by \Cref{lem:projections_span_projection}, in order for the projections $\subcube{i}$ and $\subcube{j}$ to span the projection $\subcube{\ell}$ it also has to hold that $i+j+d \geq \ell$. Finally, note that because $\subcube{\ell}$ is a \projection{} of $G$, we have $\ell = \dimension{\subcube{\ell}}\leq \dimension{G}$. These constraints lead to the set of all {\bf admissible indices}, denoted by
\begin{align*}
    \indextriplesum{G}{t} := \setbuilder{(\ell,i,j,d)\in \N_0^4}{\left(t \leq \ell \leq \minp{}{i+j+d,\dimension{G}}\right) \land (j \leq i <t) \land (i+d <\ell) \land ( d \in \{0,1,2\})}.
\end{align*} 

Analogously, we define  $\cand{\subcube{}}{t}$, $U_{\subcube{},t}(\cdot)$  and  $\indextriplesum{\subcube{}}{t}$ for a projection $\subcube{}$  of $G$ and a threshold $t \in \left[\dimension{\subcube{}}\right]$. 

In terms of the functions $ \probcubeintspan{\cdot}$ and $U_{\subcube{},t}(\cdot)$ defined in \eqref{eq:def:probcupeintspan} and \eqref{eq:def:candidatequad}, the percolation probability bound \eqref{eq:perc_unionbound} can be written as follows. 
\begin{lem}\label{lem:perc_unionbound_reformulation}
For every $t \in  [n]$ we have
    \begin{align} \label{eq:perc_unionbound_reformulation}
    \prob{G \text{ percolates}} \leq \sum_{(\ell,i,j,d) \in \indextriplesum{G}{t}} \union{G}{t}{\ell}{i}{j}{d} \probcubeintspan{i} \probcubeintspan{j}.
\end{align}
\end{lem}
The value of $U_{G,t}(\cdot)$ will be calculated quite explicitly in \Cref{lem:U()computation}. Thus, a key step in deriving an upper bound on the percolation probability is to obtain good bounds on the probability of a \projection{} being internally spanned and choosing the optimal threshold value $t$. As we will see in \Cref{sec:proof:lower},  the {\bf critical dimension} in which the right hand side of \eqref{eq:perc_unionbound_reformulation}  yields the best bounds is  $$\criticaldimension:= \rounddown{2\sqrt{n}}-2.$$

In order to bound  the right hand side of \eqref{eq:perc_unionbound_reformulation} from above,  we need the following upper bound on the function $\probcubeintspan{\cdot}$ defined in \eqref{eq:def:probcupeintspan}.
\begin{lem}\label{lem:probintspanupperbound}
For all $m \in [\criticaldimension] \cup \{0\}$ we have
    \begin{align}\label{eq:expectednumberSS}
        \probcubeintspan{m} \leq \upboundcubeintspan{m}
    \end{align} 
    where the function $\Phi:\mathbb N_0 \to \mathbb R$ is defined as 
     \begin{align}
      \upboundcubeintspan{m} \coloneqq \upboundcubeintspanconstant{m} p^{\frac{m}{2}+1} m!\  2^{-\frac{m}{2}}(k-1)^m k^\frac{m^2+2m}{4} \label{def:upboundcubeintspan}
    \end{align}    
     with 
 $$\upboundcubeintspanconstant{m} \coloneqq \indicator{m=0} +\indicator{1 \leq m \leq 14}  \frac{1}{2}\left(\frac{9 \sqrt{2}}{10}\right)^m  + \indicator{m\geq 15}\upboundcubeintspanconstant{m-1} \left(1+ k^{-\frac{m
        -11}{2}}\right)\left(1+ \frac{4}{n}\right).$$    
\end{lem}
Before we proceed with the proof of \crefitemintheorem{thm:main:threshold}{thm:main:threshold:lower} let us give some intuition on the function $\Phi$. The  value  $\upboundcubeintspan{0} = p$ is exactly the probability of a $0$-dimensional \projection{} (i.e., a single vertex) being internally spanned by a $p$-random subset $\boldsymbol{\initinfec{}}_p$. For {\em even}  $m\in [\criticaldimension]$, the value $\upboundcubeintspan{m}$ is equal to the {\em expected number of sequentially spanning sets} in a \projection{} of dimension $m$,  up to a multiplicative factor $\upboundcubeintspanconstant{m}$ and up to permutations of vertices. Thus, \Cref{lem:probintspanupperbound} states that a majority of percolating sets is given by sequentially spanning sets. For odd $m\in [\criticaldimension]$, the value $\upboundcubeintspan{m}$  overestimates the actual probability $ \probcubeintspan{m}$ by quite a large margin, but spares us from a tedious case distinction. Lastly, note that for all $m \in [\criticaldimension] \cup \{0\}$ the functions $\upboundcubeintspanconstant{m}$ are bounded from above by a universal constant (see \Cref{lem:upperboundcubeintspanconst}). 

The key to showing \Cref{lem:probintspanupperbound} lies in the following observation: Any \projection{} $\subcube{}$ of $G$ is isomorphic to a Hamming graph of dimension $\dimension{\subcube{}}$, and thus the event $\{\subcube{} \in \intspanset{}\}=\{\subcube{} \in \intspanset{\boldsymbol{\initinfec{}}_p}\}$ happens with the same probability as the event that the Hamming graph $G' := \square_{i=1}^{\dimension{\subcube{}}} K_k$ percolates. Hence, we can apply \Cref{lem:perc_unionbound_reformulation} to obtain that  for all $m \in [n]$ and for all $t\in [m]$ 
\begin{align*}
    \probcubeintspan{m} \leq \sum_{(\ell,i,j,d) \in \indextriplesum{\subcube{m}}{t}} \union{\subcube{m}}{m}{\ell}{i}{j}{d} \probcubeintspan{i} \probcubeintspan{j}.
\end{align*}
It turns out that we obtain the best bounds by considering the threshold value $t = m$, i.e., by looking at the {\em final} step in the percolation process, where two \projections{} of respective dimensions smaller than $m$ together span $\subcube{m}$.  
Let us assume by induction that the conclusion of \Cref{lem:probintspanupperbound}, i.e., $\probcubeintspan{j} \leq \upboundcubeintspan{j}$, holds for all $j \in [m-1]$ for some $m \in [\criticaldimension] \cup \{0\}$. Because any $(\ell,i,j,d) \in \indextriplesum{\subcube{m}}{m}$ satisfies $j \leq i <m$, we obtain
\begin{align}\label{eq:triplesumbound}
     \probcubeintspan{m} \leq \sum_{(\ell,i,j,d) \in \indextriplesum{\subcube{m}}{m}} \union{\subcube{m}}{m}{\ell}{i}{j}{d} \upboundcubeintspan{i} \upboundcubeintspan{j}.
\end{align}
\Cref{lem:probintspanupperbound} then follows from \eqref{eq:triplesumbound} by a careful comparison of ratios of consecutive terms. 

Returning to the proof of \crefitemintheorem{thm:main:threshold}{thm:main:threshold:lower}, using \Cref{lem:perc_unionbound_reformulation} with $t = \criticaldimension:= \rounddown{2\sqrt{n}}-2$ and \Cref{lem:probintspanupperbound}  we have
\begin{align}\label{eq:perc_unionbound_criticaldimension}
    \prob{G \text{ percolates}} 
    \overset{\eqref{eq:perc_unionbound_reformulation}}&{\leq }\sum_{(\ell,i,j,d) \in \indextriplesum{G}{\criticaldimension}} \union{G}{\criticaldimension}{\ell}{i}{j}{d} \probcubeintspan{i} \probcubeintspan{j}\notag\\
    \overset{\eqref{eq:expectednumberSS}}&{\leq}\sum_{(\ell,i,j,d) \in \indextriplesum{G}{\criticaldimension}} \union{G}{\criticaldimension}{\ell}{i}{j}{d} \upboundcubeintspan{i} \upboundcubeintspan{j}.
\end{align}
Due to technical reasons, we will consider the partition of the set $\indextriplesum{G}{ \criticaldimension}$  into two sets, depending on the size of $\ell$, 
\begin{align}\label{eq:def:admissibleindices_part_one_and_two}
    \indextriplesum{G}{\criticaldimension}^{(1)} \coloneqq \setbuilder{(\ell,i,j,d) \in \indextriplesum{G}{\criticaldimension}}{\ell \leq  3\sqrt{n}} \ \text{ and }\ \indextriplesum{G}{\criticaldimension}^{(2)} \coloneqq \setbuilder{(\ell,i,j,d) \in \indextriplesum{G}{\criticaldimension}}{\ell > 3\sqrt{n}},
\end{align} so that 
\begin{align}\label{eq:def:admissibleindices_parts}
    \indextriplesum{G}{\criticaldimension}  = \indextriplesum{G}{\criticaldimension}^{(1)} \cup \indextriplesum{G}{\criticaldimension}^{(2)}.
\end{align} 
Furthermore, to ease notation on the right hand side of \eqref{eq:perc_unionbound_criticaldimension} we introduce the function 
\begin{align}\label{eq:def:ffunc}
    \ffuncred{\ell}{i}{j}{d} \coloneqq \union{G}{\criticaldimension}{\ell}{i}{j}{d} \upboundcubeintspan{i} \upboundcubeintspan{j}.
\end{align}
Using \eqref{eq:def:admissibleindices_parts} and \eqref{eq:def:ffunc} we rewrite \eqref{eq:perc_unionbound_criticaldimension} as  
\begin{align}
    \prob{G \text{ percolates}} \leq   \sum_{(\ell,i,j,d) \in \indextriplesum{G}{\criticaldimension}^{(1)}}  \ffuncred{\ell}{i}{j}{d} + \sum_{(\ell,i,j,d) \in \indextriplesum{G}{\criticaldimension}^{(2)}}  \ffuncred{\ell}{i}{j}{d}.\label{eq:percbound:splitsum}
\end{align}

We deal with the first and second sums in \eqref{eq:percbound:splitsum} separately. We claim that the second sum over all indices in $\indextriplesum{G}{\criticaldimension}^{(2)}$ converges to zero (see \Cref{claim:expec_to_zero_largeterms})  and  the first sum over all indices in $\indextriplesum{G}{\criticaldimension}^{(1)}$ is bounded from above by a single term  $\ffuncred{\criticaldimension}{\criticaldimension-2}{0}{2}$, up to  a multiplicative factor (see  \Cref{lem:triplesumdomination_part2}). 
\begin{claim}\label{claim:expec_to_zero_largeterms}
    \begin{align*}
    \sum_{(\ell,i,j,d) \in\indextriplesum{G}{\criticaldimension{}}^{(2)}} \ffuncred{\ell}{i}{j}{d} &= o(1). 
    \end{align*}
\end{claim}
In order to deal with  the sum over all indices in $\indextriplesum{G}{\criticaldimension{}}^{(1)} $, we first note that $\ffuncred{\ell}{i}{j}{d}$ essentially bounds the expected number of \projections{} of respective dimensions $i$ and $j$ that lie at distance $d$ and together span a \projection{} of dimension $\ell \geq \criticaldimension$ in $G$. We then note that a major role in $2$-neighbour bootstrap percolation on the hypercube is played by \feasible{} sets (see \Cref{def:sequentiallyspanning}), and thus we expect that the dominating term of the left hand side of  \Cref{claim:expec_to_zero_largeterms}  would also stem from  \feasible{} sets and so in the last step of the percolation process a \projection{} of dimension $\criticaldimension-2$ would be joined with a single vertex at distance two: It corresponds to  $\ffuncred{\criticaldimension}{\criticaldimension-2}{0}{2}$. The next lemma shows that the sum is indeed dominated by this term, up to a small multiplicative factor.

\begin{lem}\label{lem:triplesumdomination_part2}
    \begin{align*}
       \sum_{(\ell,i,j,d) \in\indextriplesum{G}{\criticaldimension{}}^{(1)}} \ffuncred{\ell}{i}{j}{d} \leq  5 k^{\frac{1}{2}}\ffuncred{\criticaldimension}{\criticaldimension-2}{0}{2}.
    \end{align*} 
\end{lem}

Using simple calculations we can relate the value of $\ffuncred{\criticaldimension}{\criticaldimension-2}{0}{2}$ with $\upboundcubeintspan{\criticaldimension}$.
\begin{claim}\label{claim:ffuncSmaller_Upboundcubeintspan}
    \begin{align*}
      \ffuncred{\criticaldimension}{\criticaldimension-2}{0}{2} \leq \binom{n}{\criticaldimension} k^{n-\criticaldimension}\upboundcubeintspan{\criticaldimension}.
    \end{align*}
\end{claim}

Now, combining \Cref{lem:triplesumdomination_part2} and  \Cref{claim:ffuncSmaller_Upboundcubeintspan} we  obtain 
\begin{align}\label{eq:percolation:bottleneck}
    \sum_{(\ell,i,j,d) \in\indextriplesum{G}{\criticaldimension{}}^{(1)}} \ffuncred{\ell}{i}{j}{d}  \overset{Lem.~\ref{lem:triplesumdomination_part2}}{\leq} 5 k^{\frac{1}{2}}\ffuncred{\criticaldimension}{\criticaldimension-2}{0}{2} \overset{C.~ \ref{claim:ffuncSmaller_Upboundcubeintspan}}{\leq} 5\binom{n}{\criticaldimension} k^{n-\criticaldimension+\frac{1}{2}}\upboundcubeintspan{\criticaldimension}.
\end{align} 
Thus, \eqref{eq:percbound:splitsum}, \Cref{claim:expec_to_zero_largeterms} and \eqref{eq:percolation:bottleneck} yield 
\begin{align*}
    \prob{G \text{ percolates}}     \overset{\eqref{eq:percbound:splitsum}}&{ \leq}
      \sum_{(\ell,i,j,d) \in \indextriplesum{G}{\criticaldimension}^{(1)}}  \ffuncred{\ell}{i}{j}{d} + \sum_{(\ell,i,j,d) \in \indextriplesum{G}{\criticaldimension}^{(2)}}  \ffuncred{\ell}{i}{j}{d}\\
\overset{C. \ref{claim:expec_to_zero_largeterms}, \eqref{eq:percolation:bottleneck}}&{ \leq}5 \binom{n}{\criticaldimension} k^{n-\criticaldimension+\frac{1}{2}}\upboundcubeintspan{\criticaldimension} +\smallo{1}. 
\end{align*}
Summing up, we obtain an upper bound on the percolation probability in terms of the value of the function $\upboundcubeintspan{\cdot}$ at the critical dimension $\criticaldimension$.
\begin{lem}\label{lem:percolation:bottleneck}
\begin{align*}
     \prob{G \text{ percolates}} \leq 5 \binom{n}{\criticaldimension} k^{n-\criticaldimension+\frac{1}{2}}\upboundcubeintspan{\criticaldimension} +\smallo{1}.
\end{align*}
\end{lem}

We observe that the right hand side of the inequality in \Cref{lem:percolation:bottleneck} is  an upper bound on the number of {\em internally spanned \projections{} of dimension $\criticaldimension$} in $G$, up to a factor of order $\Th{\sqrt{k}}$. Indeed, there are exactly $\binom{n}{\criticaldimension}k^{n-\criticaldimension}$ ways of choosing a projection of dimension $\criticaldimension$ in $G$ and for each such projection $\subcube{}$, the probability that $\subcube{}$ is internally spanned is bounded from above  by $\upboundcubeintspan{\criticaldimension}$. Thus, \Cref{lem:percolation:bottleneck} shows that the bottleneck for percolation in the Hamming graph lies in finding an internally spanned \projection{} of dimension $\criticaldimension$. On the other hand, we will see in \Cref{lem:expec_to_inf} that as soon as an internally spanned \projection{} of dimension $\criticaldimension$ appears in $G$, the whole graph percolates \whp{}. To complete the proof of \crefitemintheorem{thm:main:threshold}{thm:main:threshold:lower}, we show that for $p \leq p_*$ the expected number of internally spanned \projections{} of dimension $\criticaldimension$ (even with an additional factor $\sqrt{k}$ ) tends to zero. 
\begin{lem}\label{lem:expec_to_zero_smallterm}
    \begin{align*}
         \binom{n}{\criticaldimension} k^{n-\criticaldimension+\frac{1}{2}}\upboundcubeintspan{\criticaldimension} = o(1).
    \end{align*}
\end{lem}
Now we are ready to prove \crefitemintheorem{thm:main:threshold}{thm:main:threshold:lower}.

\begin{proof}[Proof of \crefitemintheorem{thm:main:threshold}{thm:main:threshold:lower}]
From Lemmas \ref{lem:percolation:bottleneck} and \ref{lem:expec_to_zero_smallterm} we have
 \begin{align*}
     \prob{G \text{ percolates}} =\smallo{1},
\end{align*}
as desired. 
\end{proof}

\section{The upper threshold: Proof of \Cref{thm:main:threshold} \ref{thm:main:threshold:upper}}\label{sec:UpperBound}
In this section we will prove \Cref{thm:main:threshold} \ref{thm:main:threshold:upper}. Along the proof we will derive several auxiliary results, whose proofs are deferred to \Cref{sec:proof:upper}. 
We assume 
\begin{align*}
    p \geq p^* =200 n^{-2}k^{-2\sqrt{n}+1}
\end{align*}
and will show
 \begin{align*}
     \prob{G \text{ percolates}} =  \prob{ [\boldsymbol{\initinfec{}}_p]=V }  = 1+\smallo{1}.
\end{align*}

In order to show that \whp{} $G$ percolates it suffices to find a single percolating set in $G$ that is initially infected, contrary to the lower threshold, where we had to consider all possible percolating sets. It turns out the correct type of  sets to consider are \emph{sequentially spanning sets} (see \Cref{def:sequentiallyspanning}). One intuitive reason why these sets are essential to the percolation process is that they form {\em minimum} percolating sets and are thus more likely to be initially infected than larger percolating sets. To keep track of the order on the vertices of a sequentially spanning set (as given by \Cref{def:sequentiallyspanning}), we will consider sequences instead of sets. 

\begin{definition}\label{def:seqspanningsequences}
    Let $\ell \in \N$. 
    
    \begin{enumerate}
        \item 
A sequence $\seq{v}{\ell} = (\vertexvec{v}_{0},\ldots, \vertexvec{v}_{\ell})$ of vertices in $V$ is called \textbf{\feasible{}}, if $(\vertexvec{v}_{0},\ldots, \vertexvec{v}_{\ell})$ satisfies Property \eqref{Prop:sequentiallyspanning}.
The \emph{\seqsize{}} of $\ \seq{v}{\ell} = (\vertexvec{v}_{0},\ldots, \vertexvec{v}_{\ell})$ is defined as $\ell$.

\item A \feasible{} sequence $\seq{v}{\ell} = (\vertexvec{v}_{0},\ldots, \vertexvec{v}_{\ell})$ is called \textbf{infected}, if $\{\seq{v}{\ell}\} \coloneqq \{\vertexvec{v}_{0},\ldots, \vertexvec{v}_{\ell}\} \subseteq \boldsymbol{\initinfec{}}_p$.

\item For two \feasible{} sequences $\seq{v}{\ell},\seq{w}{\ell} $ we define $\seq{v}{\ell} \cap \seq{w}{\ell} \coloneqq \{\seq{v}{\ell}\} \cap \{\seq{w}{\ell}\}$. 
    \end{enumerate}
\end{definition}
Note that the \seqsize{} of $\seq{v}{\ell} = (\vertexvec{v}_{0},\ldots, \vertexvec{v}_{\ell})$ is defined in such a way that a \feasible{} sequence of \seqsize{} $\ell$ spans a \projection{} of dimension $2 \ell$ (see \Cref{lem:seqspanningspansprojection}).

Our aim is to find a sufficiently large \emph{infected} \feasible{} sequence in $G$. To that end, let us denote by $$\seqset{\ell} = \seqset{\ell}(G) \coloneqq \setbuilder{\seq{v}{\ell} = (\vertexvec{v}_{0},\ldots, \vertexvec{v}_{\ell})}{\seq{v}{\ell}  \text{ is \feasible{}}} $$
the set of all \feasible{} sequences of \seqsize{} $\ell$ in $G$ and by $\seqspannum{\ell}$ the number of all \emph{infected} \feasible{} sequences of \seqsize{} $\ell$.  Then, for all $\ell \in \N_0$ we obtain
\begin{align*}
    \seqspannum{\ell} = \sum_{\seq{v}{\ell}  \in \seqset{\ell}} \indicator{\{\seq{v}{\ell}\} \subseteq \boldsymbol{\initinfec{}}_p}.
\end{align*}
A first step towards our goal is to count the number of \feasible{} sequences. Because a \feasible{} sequence of \seqsize{} zero consists of a single vertex, it clearly holds that 
$$|\seqset{0}| = \abs{V}= k^n.$$
In the next lemma we explicitly compute the number of \feasible{} sequences of given length by successively extending such sequences.
\begin{lem}\label{lem:seqspancount}
    For $1 \leq \ell \leq \rounddown{\frac{n}{2} }$, let $\seq{v}{\ell-1} =(\vertexvec{v}_0,\ldots,\vertexvec{v}_{\ell-1}) \in \seqset{\ell-1}$ be a \feasible{} sequence of \seqsize{} $\ell-1$. Denote by
    \begin{align*}
        \seqspanconst{\ell} \coloneqq \left| \setbuilder{\vertexvec{w} \in V}{(\vertexvec{v}_0,\ldots,\vertexvec{v}_{\ell-1},\vertexvec{w}) \in \seqset{\ell} } \right|
    \end{align*}
 the number of vertices in $V$ that extend $\seq{v}{\ell-1}$ to a \feasible{} sequence of \seqsize{} $\ell$. 
    Then, for $1 \leq \ell \leq \rounddown{\frac{n}{2}}$,
    \begin{align*}
 \seqspanconst{\ell} = \binom{n-2\ell+2}{2}(k-1)^2k^{2\ell-2} 
    \end{align*}
       
     and for $0 \leq \ell \leq \rounddown{\frac{n}{2}}$,
         \begin{align}
                 |\seqset{\ell}| =  k^n\prod_{j=1}^{\ell} \seqspanconst{j} = \frac{n!}{(n-2\ell)!}  (k-1)^{2\ell} 2^{-\ell} k^{n+\ell^2-\ell}.\label{eq:Sell}
         \end{align}
\end{lem}
Observe that by the definitions of $\seqset{\ell}$ and $\seqspanconst{\ell}$,   we obtain that for $1 \leq \ell \leq \rounddown{\frac{n}{2}}$, 
\begin{align}\label{eq:SellCell0}
|\seqset{\ell}| = |\seqset{\ell-1}| \cdot \seqspanconst{\ell}.
\end{align}

Using \Cref{lem:seqspancount} we already have a good grip on  the expected number of {\em \feasible{} sequences} of given length, because 
for any sequence $\seq{v}{\ell} \in \seqset{\ell}$ it holds that $ \prob{\indicator{\{\seq{v}{\ell}\} \subseteq \boldsymbol{\initinfec{}}_p}} = p^{\ell+1}$ and thus 
\begin{align}
    \expec{\seqspannum{\ell}} = p^{\ell+1}|\seqset{\ell}|. \label{eq:expec(seqspannum)}
\end{align}

Next we aim to apply the second moment method to show that \whp{} there exists an {\em infected \feasible{} sequence} of length $\rounddown{\frac{n}{2}}$. Note that we cannot hope to find a \feasible{} sequence of \seqsize{} $\ell$ which is strictly larger than  $\rounddown{\frac{n}{2}}$,  because by \Cref{lem:seqspanningspansprojection} such a sequence would span a \projection{} of dimension  $2\ell >n =\dimension{G}$ -- a contradiction. Thus we set  
\begin{align*}
    &\criticalsizesequential \coloneqq \rounddown{\frac{n}{2}}; \\
    &\D{\ell} \coloneqq \sum_{\setbuilder{\{\seq{v}{\criticalsizesequential},\seq{w}{\criticalsizesequential} \} \subseteq \seqset{\criticalsizesequential}}{\seq{v}{\criticalsizesequential} \cap \seq{w}{\criticalsizesequential}\neq \emptyset}} \prob{\indicator{\left(\{\seq{v}{\criticalsizesequential}\} \subseteq \boldsymbol{\initinfec{}}_p \right)\land \left(\{\seq{w}{\criticalsizesequential}\} \subseteq \boldsymbol{\initinfec{}}_p\right)}}
\end{align*}
and note that 
by \Cref{lem:secondmoment} it suffices to show  
    \begin{align}   &\expec{\seqspannum{\criticalsizesequential}}  \xrightarrow{} \infty;\label{eq:secondmoment:expeccond}\\
   &  \frac{\D{\criticalsizesequential}}{ \expec{\seqspannum{\criticalsizesequential}}^2}  \xrightarrow{} 0.\label{eq:secondmoment:variancecond}
\end{align}

Like in most applications of the second moment method, the harder part is to prove \eqref{eq:secondmoment:variancecond}, which requires us to compute the correlation between different infected \feasible{} sequences. For $\{\seq{v}{\ell},\seq{w}{\ell} \} \subseteq \seqset{\ell}$ it holds that 
\begin{align}
\prob{\indicator{\left(\{\seq{v}{\ell}\} \subseteq \boldsymbol{\initinfec{}}_p\right) \land \left(\{\seq{w}{\ell}\} \subseteq \boldsymbol{\initinfec{}}_p\right)}} = p^{2(\ell+1)- \left|\seq{v}{\ell}\cap \seq{w}{\ell}\right|}.\label{eq:secondmoment:prob}
\end{align}
Thus, the proof boils down to bounding the number of pairs of \feasible{} sequences that share a certain amount of vertices. For each $\ell \in [\criticalsizesequential]\cup \set{0}$ and each $i \in [\ell+1]$ we denote by 
\begin{align}
    \seqoverlap{\ell}{i} \coloneqq \setbuilder{\{\seq{v}{\ell},\seq{w}{\ell} \} \subseteq \seqset{\ell}}{\left|\seq{v}{\ell}\cap \seq{w}{\ell}\right| = i}\label{eq:def:seqoverlap}
\end{align}
the number of pairs of \feasible{} sequences of \seqsize{} $\ell$ that intersect exactly in $i$ vertices.  
To give some intuition on what we expect the size of  $\seqoverlap{\ell}{i}$ to be, consider the following way of constructing intersecting \feasible{} sequences: Start both sequences with the same initial sequence of length $i-1$ (with the same $i$ vertices) and then use arbitrary vertices to extend them. Clearly, there are $|\seqset{i-1}|$ ways for choosing an initial sequence of length $i-1$. Furthermore, by \eqref{eq:SellCell0}, we have $\seqspanconst{j}=\frac{|\seqset{j}|}{|\seqset{j-1}|}$ and thus there are 
$ \prod_{j=i+1}^{\ell} \seqspanconst{j}^2  = \left(\frac{|\seqset{\ell}|}{|\seqset{i-1}|}\right)^2$ ways to extend two sequences of length $i-1$ to two sequences of length $\ell$. Note that it is possible that the resulting sequences overlap in even more than $i$ vertices, which happens if one of the extensions also appeared previously in the other sequence, however this only reduces the number of possible extensions by an insignificant amount. Thus, we expect that $|\seqoverlap{\ell}{i}| \geq \frac{|\seqset{\ell}|^2}{|\seqset{i-1}|}$ and it turns out that this naive bound is almost tight, as shown in the following lemma. 

\begin{lem}\label{lem:bound:seqoverlap}
  For $\ell \in [\criticalsizesequential]\cup\{0\}$ and $i \in [\ell+1]$ we have
    \begin{align}
       \left|\seqoverlap{\ell}{i} \right| \leq \Th{1} 8^i(\ell+1)^3 \frac{|\seqset{\ell}|^2}{|\seqset{i-1}|}.\label{eq:bound:seqoverlap}
    \end{align}
\end{lem}
Combining \eqref{eq:secondmoment:prob} and \eqref{eq:def:seqoverlap} we have 
\begin{align}\label{eq:Deltabound}
    \D{\criticalsizesequential} \overset{\eqref{eq:secondmoment:prob},\eqref{eq:def:seqoverlap}}{=} \sum_{i=1}^{\criticalsizesequential+1} \abs{\seqoverlap{\criticalsizesequential}{i}} p^{2\criticalsizesequential+2-i} \leq (\criticalsizesequential+1) \max_{i \in [\criticalsizesequential+1]}  \left(\abs{\seqoverlap{\criticalsizesequential}{i}} p^{2\criticalsizesequential+2-i}\right). 
\end{align}
Using \eqref{eq:Deltabound} and \Cref{lem:bound:seqoverlap} (by taking $\ell=\criticalsizesequential$) we obtain 
\begin{align}
     \D{\criticalsizesequential} \overset{\eqref{eq:Deltabound},\eqref{eq:bound:seqoverlap}}&{\leq}\Theta(1)(\criticalsizesequential+1)^4 \max_{i \in [\criticalsizesequential+1]} \left(8^i \frac{|\seqset{\criticalsizesequential}|^2}{|\seqset{i-1}|}p^{2\criticalsizesequential+2-i}\right) \notag \\
     &\leq \Theta(n^4)\max_{i \in [\criticalsizesequential+1]} \left(8^i \frac{|\seqset{\criticalsizesequential}|^2}{|\seqset{i-1}|}p^{2\criticalsizesequential+2-i}\right) ,\label{eq:Deltabound2}
     \end{align}
     where the last inequality is due to $\criticalsizesequential=\Theta(n)$.
From \eqref{eq:Deltabound2} and  \eqref{eq:expec(seqspannum)} (with $\ell=\criticalsizesequential$ and $\ell=i-1$, obtaining $\expec{\seqspannum{\criticalsizesequential}} = p^{\criticalsizesequential+1}|\seqset{\criticalsizesequential}|$ and $\expec{\seqspannum{i-1}} = p^{i}|\seqset{i-1}|$, respectively) it follows
\begin{align*} 
    \D{\criticalsizesequential}  
   \overset{\eqref{eq:Deltabound2}}&{\leq}  \Theta(n^4)\max_{i \in [\criticalsizesequential+1]} \left(8^i \frac{|\seqset{\criticalsizesequential}|^2 p^{2\criticalsizesequential+2}}{|\seqset{i-1}| p^{i}}\right)
    \notag\\
    \overset{\eqref{eq:expec(seqspannum)}}&{\leq} \Theta(n^4) \max_{i \in [\criticalsizesequential+1]} \left(8^i \frac{\expec{\seqspannum{\criticalsizesequential}}^2}{\expec{\seqspannum{i-1}}}  \right) 
     = \frac{\Theta(n^4)\expec{\seqspannum{\criticalsizesequential}}^2}{\min_{i \in [\criticalsizesequential+1]} \left(8^{-i} \expec{\seqspannum{i-1}}\right)},
\end{align*}
which implies
\begin{align}\label{eq:secondmoment_reformulation}
   \frac{\D{\criticalsizesequential}}{ \expec{\seqspannum{\criticalsizesequential}}^2} \leq \frac{\Theta(n^4) }{\min_{i \in [\criticalsizesequential+1]} \left(8^{-i} \expec{\seqspannum{i-1}}\right)}.
\end{align}
Thus, in order to prove \eqref{eq:secondmoment:variancecond}, i.e., $ \frac{\D{\criticalsizesequential}}{ \expec{\seqspannum{\criticalsizesequential}}^2} \to 0$, it suffices to verify that $\expec{\seqspannum{i}}$ grows faster than $8^{i+1} n^4$ for all $i \in [\criticalsizesequential]\cup\{0\}$. 
\begin{lem}\label{lem:expec_to_inf}
    For every $i \in [\criticalsizesequential]\cup\{0\}$,
    \begin{align}
        \frac{\expec{\seqspannum{i}} }{8^{i+1} \Theta(n^4) } = \omega(1).\label{eq:expec_to_inf}
    \end{align}
\end{lem}
As we will see in the proof of \Cref{lem:expec_to_inf}, the function $x\mapsto \frac{\expec{\seqspannum{x}} }{8^{x+1} \Theta(n^4) }$  attains its minimum over the reals when 
\begin{align*}
    x = \sqrt{n}-1.
\end{align*} However, in \Cref{lem:expec_to_inf} the index $i \in [\criticalsizesequential]\cup\{0\}$ is an {\em integer}, and thus the actual minimiser (or \lq bottleneck\rq), i.e., the index $i\in [\criticalsizesequential]\cup\{0\}$ at which $\frac{\expec{\seqspannum{i}} }{8^{i+1} \Theta(n^4) }$ is smallest, happens either for
$i = \rounddown{\sqrt{n}}-1$ or for $i = \roundup{\sqrt{n}}-1$. Because in both cases the error terms are insignificant, we only focus on the former. We observe that  each \feasible{} sequence of \seqsize{} $i_{*}:=\rounddown{\sqrt{n}}-1$ spans a \projection{} of dimension $2i_{*}$. Thus, the {\bf critical droplet} for percolation is indeed an internally spanned \projection{} of {\em critical dimension}
$$\criticaldimension :=  2 i_{*} = \rounddown{2\sqrt{n}}-2.$$
(See \Cref{lem:percolation:bottleneck} for the \lq negative\rq\ statement, meaning that for $p \leq p_*$ \whp{} there does not exist an internally spanned \projection{} of dimension $\criticaldimension$). 

As discussed before, by the second moment method (see \Cref{lem:secondmoment}), in order to show that \whp{} $\seqspannum{\criticalsizesequential} \geq 1$ it suffices to verify \eqref{eq:secondmoment:expeccond}, i.e., $\expec{\seqspannum{\criticalsizesequential}} \xrightarrow{} \infty$, and \eqref{eq:secondmoment:variancecond}, i.e., $ \frac{\D{\criticalsizesequential}}{ \expec{\seqspannum{\criticalsizesequential}}^2} \to 0$: We observe that  the former  follows directly from \Cref{lem:expec_to_inf} (with $i = \criticalsizesequential$) and the latter from  \eqref{eq:secondmoment_reformulation} and \Cref{lem:expec_to_inf}  (with $i = \criticalsizesequential$).
\begin{cor}\label{cor:Xl_not_0}
\Whp{} $\seqspannum{\criticalsizesequential} \geq 1$. 
\end{cor}

Note that by \Cref{cor:Xl_not_0}, \whp{} $G$   contains an infected \feasible{} sequence of \seqsize{} $\criticalsizesequential$ and by \Cref{lem:seqspanningspansprojection}, this sequence spans a \projection{} $\subcube{}$  in $G$ of dimension $2\criticalsizesequential$: In other words,  \whp{} $G$   contains an internally spanned \projection{} $\subcube{}$ of dimension $2\criticalsizesequential$. Recalling $\criticalsizesequential:= \rounddown{\frac{n}{2}}$ we obtain the following.
\begin{cor}\label{cor:spanningprojection}
\Whp{} there exists an internally spanned \projection{} $\subcube{}$ of dimension $2 \rounddown{\frac{n}{2}}$  in $G$. 
\end{cor}

In order to prove  \crefitemintheorem{thm:main:threshold}{thm:main:threshold:upper} we need one last auxiliary result when $n$ is odd. 
\begin{lem}\label{lem:odd_case_distinction}
    Let $n = 2m+1$ for  $m\in \mathbb N_0$. Then \whp{} the following holds: For every \projection{} $\subcube{}$ of dimension $2m$ there exists at least one infected vertex in $G$ that is not contained in $\subcube{}$.
\end{lem}

\begin{proof}[Proof of \crefitemintheorem{thm:main:threshold}{thm:main:threshold:upper}]
By \Cref{cor:spanningprojection}, 
\whp{} there exists an internally spanned \projection{} $\subcube{}$ of dimension $2 \rounddown{\frac{n}{2}}$  in $G$.

If $n$ is even, then $\dimension{\subcube{}} =2 \rounddown{\frac{n}{2}}= n$ and thus $\subcube{} = G$, showing that \whp{}  $G$ percolates.

If $n$ is odd, say $n = 2m+1$ for  $m\in \mathbb N_0$, then  $\dimension{\subcube{}} =2 \rounddown{\frac{n}{2}}=2m$. But, it follows from \Cref{lem:odd_case_distinction} that \whp{} there exists an infected vertex $\vertexvec{v}$ that is not contained in $\subcube{}$. Clearly, $\{\vertexvec{v}\} \cup \subcube{}$ can have no trivial coordinates and in addition it has to hold that  $\dist{\vertexvec{v}}{\subcube{}} = 1 $. Thus, by \Cref{lem:projections_span_projection} we have $\closure{\{\vertexvec{v}\}\cup \subcube{}} = G,$  and thus \whp{} $G$ percolates.
\end{proof}

\section{Proofs of auxiliary results for the lower threshold}\label{sec:proof:lower}
In this section we will prove  auxiliary  results that were necessary to show the lower threshold of \crefitemintheorem{thm:main:threshold}{thm:main:threshold:lower}. Recall that for a \projection{} $\subcube{}$ of $G$, threshold dimension $t \in [n]$ and indices $(\ell,i,j,d) \in \N_0^4$ we denote by
\begin{align*}
    \union{\subcube{}}{t}{\ell}{i}{j}{d} \coloneqq \left| \setbuilder{\parens{\subcube{\ell'},\subcube{i'}, \subcube{j'},d'} \in \cand{\subcube{}}{t}}{\parens{\ell' =\ell} \land \parens{i'=i}\land \parens{j'=j} \land \parens{d'=d}}\right|
\end{align*} the number of candidate quadruples with fixed indices $(\ell,i,j,d)$. In addition we denote by 
\begin{align*}
    \indextriplesum{\subcube{}}{t} \coloneqq\setbuilder{(\ell,i,j,d)\in \N_0^4}{\left(t \leq \ell \leq \minp{}{i+j+d,\dimension{\subcube{}}}\right) \land (j \leq i <t) \land (i+d <\ell) \land ( d \in \{0,1,2\})}
\end{align*} the set of all \emph{admissible indices} and by 
$$ \triple{\subcube{}}{t} \coloneqq \setbuilder{(\ell,i,d) \subseteq \N_0^3}{\exists j\colon (\ell,i,j,d) \in \indextriplesum{\subcube{}}{t}}$$ 
 the set of \emph{admissible triples}.

We begin with explicitly computing the function $\union{\subcube{}}{t}{\ell}{i}{j}{d}$. 
\begin{lem}\label{lem:U()computation}
 For $m \in [n]$, let $\subcube{}$ be a \projection{} of dimension $m$. Then for $t \in [m]$ and $ (\ell,i,j,d) \in \N_0^4$, 
    \begin{align}\label{eq:U(l,i,j,d)count}
    \union{\subcube{}}{t}{\ell}{i}{j}{d} =\binom{m}{\ell}  \binom{\ell}{i} \binom{\ell-i}{d} \binom{i}{i+j+d-\ell} k^{m+\ell-i-j-d} (k-1)^d \left(\frac{1}{2}\right)^{\indicator{i = j}} \indicator{(\ell,i,j,d) \in \indextriplesum{\subcube{}}{t}}.
\end{align}
\end{lem}
\begin{proof}
Note first that by \Cref{lem:wittnessingquad}, the relations $\subcube{\ell} \subseteq \subcube{m}$ and $\subcube{i} \cup \subcube{j} \subseteq \subcube{\ell}$ have to hold. 
 We start by choosing a \projection{} $\subcube{\ell}$ in $\subcube{m}$. There are $ \binom{m}{\ell}$ ways to choose the coordinates in which $\subcube{\ell}$ will consist of a complete $K_k$. In all other coordinates we can choose any of $k$ vertices, resulting in $ \binom{m}{\ell} k^{m-\ell}$ choices for $\subcube{\ell}$.
Similarly, there are $\binom{\ell}{i} k^{\ell-i}$ ways to fix a \projection{} $\subcube{i}$ in $\subcube{\ell}$. Next, we aim to fix the $d$ coordinates in which $\subcube{i}$ and $\subcube{j}$ differ. Clearly, $\subcube{i}$ has to consist of single vertices in these coordinates, implying that there are $\binom{\ell-i}{d}$ possibilities for choosing those coordinates and further $(k-1)^d$ possibilities for specifying the value of  $\subcube{j}$ in those coordinates. To recap shortly, we already have specified $\subcube{\ell}$ and $\subcube{i}$ as well as the $d$ coordinates of $\subcube{j}$ in which it differs from $\subcube{i}$. 

Continuing with fixing $\subcube{j}$, let us next assign the $j$ coordinates in which $\subcube{j}$ consists of complete graphs. There are $\ell-i-d$ remaining coordinates, in which $\subcube{i}$ consists of a single vertex. In each of those $\subcube{j}$ has to be a complete $K_k$, as otherwise the two \projections{} would either lie at distance at least three, or not span $\subcube{\ell}$. Thus, there only remain $j-(\ell-i-d) = i+j+d-\ell$ coordinates in which $\subcube{j}$ is complete and these coordinates can be chosen from any coordinate in which $\subcube{i}$ is  also complete, resulting in further $\binom{i}{i+j+d-\ell}$ choices. Finally, it remains to specify the value of $\subcube{j}$ in the remaining $\ell-j-d$ coordinates in which $\subcube{j}$ consists of a single vertex, leading to $k^{\ell-j-d}$ further possibilities. After simplifying the expression, we are left with a total of $\binom{m}{\ell}  \binom{\ell}{i} \binom{\ell-i}{d} \binom{i}{i+j+d-\ell} k^{m+\ell-i-j-d} (k-1)^d$ possibilities.  Lastly, note that in the case $i=j$ a factor of $2$ is lost due to symmetry and we have to require $(\ell,i,j,d) \in \indextriplesum{\subcube{m}}{t}$ in order to make sure that the indices $(\ell,i,j,d)$ fulfill the dimensional constraints posed in \Cref{lem:wittnessingquad}. 
\end{proof}

Our next aim is to prove \Cref{lem:probintspanupperbound}. Similarly as in \eqref{eq:def:ffunc}, for every \projection{} $\subcube{}$, threshold dimension $t \in [n]$ and indices $(\ell,i,j,d) \in \N_0^4$  we define 
\begin{align}
    \ffunc{\subcube{}}{t}{\ell}{i}{j}{d} \coloneqq \union{\subcube{}}{t}{\ell}{i}{j}{d} \upboundcubeintspan{i} \upboundcubeintspan{j}, \label{def:ffunc2}
\end{align}
and set $\criticaldimension \coloneqq \rounddown{2\sqrt{n}}-2$ as before. Observe  that the functions $\upboundcubeintspanconstant{\ell}$ in \Cref{lem:probintspanupperbound} are chosen in such a way that for all $\ell \in [\criticaldimension]$ we have
\begin{align}\label{eq:cubeintspanconst_bounds}
    \frac{\upboundcubeintspanconstant{\ell}}{\upboundcubeintspanconstant{\ell-1}} \leq \frac{9 \sqrt{2}}{10} \leq \frac{9 \sqrt{k}}{10}, 
\end{align} which can be verified by simple calculations.
To show \Cref{lem:probintspanupperbound} we need the following auxiliary result.
\begin{lem}\label{lem:triplesumdomination}
    Let $15 \leq m \leq \criticaldimension$ and let $\subcube{}$ be a \projection{} of dimension $m$. Then, 
    \begin{align*}
         \sum_{(\ell,i,j,d) \in \indextriplesum{\subcube{}}{m}}  \ffunc{\subcube{}}{m}{\ell}{i}{j}{d} \leq \left(1+ k^{-\frac{m
        -11}{2}}\right)\left(1+ \frac{4}{n}\right) \ffunc{\subcube{}}{m}{m}{m-2}{0}{2}. 
    \end{align*}
\end{lem}
Observe that \Cref{lem:triplesumdomination} is quite similar to \Cref{lem:triplesumdomination_part2} and we will prove \Cref{lem:triplesumdomination} using slightly more general results than necessary in order for them to be also applicable to \Cref{lem:triplesumdomination_part2}. Let us start by comparing consecutive terms of $\upboundcubeintspan{\cdot}$. 

\begin{lem}\label{lem:ratio:P(j):estimate}
For $j \in [n]\cup \{0\}$,
    \begin{align*}
     \frac{\upboundcubeintspan{j+1}}{\upboundcubeintspan{j}} \leq \frac{j+1}{n}k^{\frac{j+5-2\sqrt{n}}{2}}. 
    \end{align*}
\end{lem}
\begin{proof}
    By the definition \eqref{def:upboundcubeintspan} of  $\upboundcubeintspan{\cdot}$ it holds that
    \begin{align}
    \frac{\upboundcubeintspan{j+1}}{\upboundcubeintspan{j}} = \frac{\upboundcubeintspanconstant{j+1} p^{\frac{j+1}{2}+1} (j+1)! 2^{\frac{-j-1}{2}}(k-1)^{j+1} k^\frac{(j+1)^2+2(j+1)}{4}}{\upboundcubeintspanconstant{j} p^{\frac{j}{2}+1} j! 2^{\frac{-j}{2}}(k-1)^j k^\frac{j^2+2j}{4}} =  \frac{\upboundcubeintspanconstant{j+1}\sqrt{p}(j+1)k^{\frac{2j+3}{4}}(k-1)}{\upboundcubeintspanconstant{j}\sqrt{2}} \label{eq:ratio:P(j):equality}. 
    \end{align}
    Furthermore, using \eqref{eq:cubeintspanconst_bounds} and $p \leq n^{-2}k^{-2\sqrt{n}+1}$ we get
    \begin{align*}
        \frac{\upboundcubeintspanconstant{j+1}\sqrt{p}(j+1)k^{\frac{2j+3}{4}}(k-1)}{\upboundcubeintspanconstant{j}\sqrt{2}} \leq \sqrt{p}(j+1)k^{\frac{2j+3}{4}}(k-1) \leq \sqrt{p}(j+1)k^{\frac{j+4}{2}} \leq \frac{j+1}{n}k^{\frac{j+5-2\sqrt{n}}{2}},
    \end{align*} yielding together with \eqref{eq:ratio:P(j):equality} that
    \begin{align*}
        \frac{\upboundcubeintspan{j+1}}{\upboundcubeintspan{j}}\leq \frac{j+1}{n}k^{\frac{j+5-2\sqrt{n}}{2}}, 
    \end{align*} as claimed.
\end{proof}

In order to prove \Cref{lem:triplesumdomination}, we will now consider the effect it has on $\ffunc{\subcube{}}{t}{\ell}{i}{j}{d}$ to change each of the respective indices $i,j$ and $d$. We will first deal with different values of $j$, then, in \Cref{lem:sumdom:i}, consider different values of $i$ and lastly compare different values of $d$  in \Cref{lem:sumdom:d} . 
\begin{lem}\label{lem:sumdom:j}
For $m\in [n]$, let $\subcube{}$ be a \projection{} of dimension $m$. Let  $15 \leq t \leq \minp{}{m,\criticaldimension}$ and $(\ell,i,d) \in \triple{\subcube{}}{t}$ such that $15 \leq \ell \leq 3\sqrt{n}$. Then,   
    \begin{align}
        \sum_{j=0}^i \ffunc{\subcube{}}{t}{\ell}{i}{j}{d} \leq \ffunc{\subcube{}}{t}{\ell}{i}{\ell-i-d}{d}\left(1+k^{-\frac{\sqrt{n}}{6}}\right).\label{eq:varyingj}
    \end{align}
\end{lem}
\begin{proof} Observe first that $j$ is the only index of $\ffunc{\subcube{}}{t}{\ell}{i}{j}{d}$ that is not fixed on the left and right hand sides of the inequality in \eqref{lem:sumdom:j}. To ease notation, let us define
\begin{align}
    \ffuncredj{j} \coloneqq \ffunc{\subcube{}}{t}{\ell}{i}{j}{d}  \overset{\eqref{def:ffunc2}}{=}\union{\subcube{}}{t}{\ell}{i}{j}{d} \upboundcubeintspan{i} \upboundcubeintspan{j}\label{eq:ffuncred:j}
\end{align} 
and $$\optimalj \coloneqq \ell -i -d.$$
   By \Cref{lem:U()computation},  for all $j < \optimalj$ or  $j \geq \criticaldimension$ it holds that $ \union{\subcube{m}}{t}{\ell}{i}{j}{d} = 0$. For $\optimalj \leq j \leq j+1\leq \criticaldimension$ we have 
    \begin{align*}
        \frac{ \union{\subcube{}}{t}{\ell}{i}{j+1}{d}}{ \union{\subcube{}}{t}{\ell}{i}{j}{d}} =\frac{\binom{i}{i+j+1+d-\ell}\left(\frac{1}{2}\right)^{\indicator{i = j+1}}}{ \binom{i}{i+j+d-\ell}k\left(\frac{1}{2}\right)^{\indicator{i = j}}} \leq \frac{\ell-j-d}{k(i+j+d+1-\ell)},
    \end{align*} implying together with \Cref{lem:ratio:P(j):estimate} that  for all $j \geq \optimalj$ 
    
    \begin{align}
  \frac{\ffuncredj{j+1}}{\ffuncredj{j}} \overset{\eqref{eq:ffuncred:j}}&{=}  \frac{\upboundcubeintspan{i}\upboundcubeintspan{j+1} \union{\subcube{}}{t}{\ell}{i}{j+1}{d}}{\upboundcubeintspan{i }\upboundcubeintspan{j} \union{\subcube{}}{t}{\ell}{i}{j}{d}} \leq k^{\frac{j+3-2\sqrt{n}}{2}} \frac{(j+1)(\ell-d-j)}{ n(j+1+i+d-\ell)} \label{eq:sumdom:j:eq1}\\
  &\leq 6k^{\frac{j+3-2\sqrt{n}}{2}},\label{eq:sumdom:j:eq3}
    \end{align} where the last inequality uses $j +1\leq \criticaldimension:= \rounddown{2\sqrt{n}}-2 \leq 2\sqrt{n}$  and $\ell \leq 3\sqrt{n}$.
    
    Next we aim to show that $ \frac{\ffuncredj{j+1}}{\ffuncredj{j}} \leq 1$. Indeed, if $j \leq \frac{5 \sqrt{n}}{3}$ we obtain from \eqref{eq:sumdom:j:eq3} that 
    \begin{align}
       \frac{\ffuncredj{j+1}}{\ffuncredj{j}} \leq 6 k^{\frac{j+3-2\sqrt{n}}{2}} \leq 6 k^{-\frac{\sqrt{n}}{7}} \leq 1.\label{eq:sumdom:j:eqx}
    \end{align} If $j >\frac{5 \sqrt{n}}{3}$, because  $j +1 \leq \criticaldimension-1\leq 2\sqrt{n}-3$ and $j+1+i+d-\ell \geq \frac{10 \sqrt{n}}{3} - 3\sqrt{n} \geq \frac{\sqrt{n}}{3}$, we obtain from   \eqref{eq:sumdom:j:eq1}
     \begin{align}
        \frac{\ffuncredj{j+1}}{\ffuncredj{j}} \leq k^{\frac{j+3-2\sqrt{n}}{2}} \frac{(j+1)(\ell-d-j)}{ n(j+1+i+d-\ell)} \leq  \frac{\parens{2\sqrt{n}}\parens{4\sqrt{n}}}{n \frac{\sqrt{n}}{3}} \leq \frac{24}{\sqrt{n}} \leq1.\label{eq:sumdom:j:eqy}
    \end{align}
    It thus follows from \eqref{eq:sumdom:j:eqx}  and \eqref{eq:sumdom:j:eqy} that   $$\frac{\ffuncredj{j+1}}{\ffuncredj{j}} \leq 1 \quad \text{for all} \quad  j \geq \optimalj$$ and hence 
    \begin{align}\label{eq:sumdom:j:eq2}
        \sum_{j=0}^i \ffuncredj{j} \leq \ffuncredj{\optimalj} + i \ffuncredj{\optimalj+1}.
    \end{align}
    Furthermore, as $\optimalj\leq i$, it follows from $\ell = i+\optimalj+d$ that $\optimalj \leq \frac{\ell}{2} \leq \frac{3\sqrt{n}}{2}$, implying together with \eqref{eq:sumdom:j:eq3}
    \begin{align*}
       \frac{\ffuncredj{\optimalj+1}}{\ffuncredj{\optimalj}} \leq 6 k^{\frac{\optimalj+3-2\sqrt{n}}{2}} \leq 6k^\frac{3-\frac{1}{2}\sqrt{n}}{2} \leq k^{-\frac{\sqrt{n}}{5}}, 
    \end{align*} and thus together with \eqref{eq:sumdom:j:eq2} we have
    \begin{align*}
        \sum_{j=0}^i \ffuncredj{j} \leq\ffuncredj{\optimalj} \left(1+ik^{-\frac{\sqrt{n}}{5}}\right) \leq\ffuncredj{\optimalj}\left(1+k^{-\frac{\sqrt{n}}{6}}\right),
    \end{align*} as claimed.
\end{proof}

We continue by considering different values of $i$. Here we need to make a case distinction depending on whether $\ell > \criticaldimension$ or not.
\begin{lem}\label{lem:sumdom:i}
(a)    For $\criticaldimension < \ell \leq 3\sqrt{n} $ and $d\in \{0,1,2\}$ it holds that
    \begin{align}\label{eq:lem:sumdom:i:secondpart}
        \sum_{i=0}^{\criticaldimension-1} \ffunc{G}{\criticaldimension}{\ell}{i}{\ell-i-d}{d} \leq \left(1+k^{-\frac{\sqrt{n}}{6}}\right) \ffunc{G}{\criticaldimension}{\ell}{\criticaldimension-1}{\ell-(\criticaldimension-1)-d}{d}.
    \end{align}
(b) For $15 \leq m \leq \criticaldimension$ let $\subcube{}$ be a \projection{} of dimension $m$. Then for   $d\in \{0,1,2\}$ we have
    \begin{align}
        &\sum_{i=0}^{m-1} \ffunc{\subcube{}}{m}{m}{i}{m-i-d}{d}   \leq \left(1+k^{-\frac{m-11}{2}}\right) \ffunc{\subcube{}}{m}{m}{m-1-\indicator{d=2}}{1-d+\indicator{d=2}}{d}.\label{eq:lem:sumdom:i:secondpartb}
    \end{align} 
   
\end{lem}
\begin{proof}
We first derive some properties that are relevant for both (a) and (b). Assume $\subcube{}$ is a \projection{} of  $G$ and let $t \in [n]$ (in fact, we will take $t=\criticaldimension$ for part (a) and $t=m$ for  part (b)). To ease notation, for  $(\ell,i,d) \in \triple{\subcube{}}{t}$ we set $\optimalj := \ell -i-d$ and define
    \begin{align}
        \ffuncredi{\subcube{}}{t}{i} \coloneqq \ffunc{\subcube{}}{t}{\ell}{i}{\optimalj}{d} \overset{\eqref{def:ffunc2}}{=}\union{\subcube{}}{t}{\ell}{i}{\optimalj}{d} \upboundcubeintspan{i} \upboundcubeintspan{\optimalj}.\label{eq:ffuncred:i}
    \end{align}
    We aim to bound the ratio $\frac{\ffuncredi{\subcube{}}{t}{i-1}}{\ffuncredi{\subcube{}}{t}{i}}$ from above, so let us assume for now that we are on the domain where $\ffuncredi{\subcube{}}{t}{\cdot}$ is non-zero.  First observe that by  \eqref{eq:ratio:P(j):equality} and \eqref{eq:cubeintspanconst_bounds} it holds that 
    \begin{align}\label{eq:sumdom:i:eq1}
        \frac{\upboundcubeintspan{i-1}\upboundcubeintspan{\optimalj+1}}{\upboundcubeintspan{i}\upboundcubeintspan{\optimalj}} \overset{\eqref{eq:ratio:P(j):equality}}{=} 
        k^{-\frac{i-\optimalj-1}{2}} \frac{\upboundcubeintspanconstant{i-1}\upboundcubeintspanconstant{\optimalj+1}}{\upboundcubeintspanconstant{i}\upboundcubeintspanconstant{\optimalj}}\frac{(\optimalj+1)}{i} \overset{\eqref{eq:cubeintspanconst_bounds}}{\leq} k^{-\frac{i-\optimalj-2}{2}} \frac{ 9(\optimalj+1)}{10i} .
    \end{align}
Furthermore, it follows from \Cref{lem:U()computation} that
\begin{align}\label{eq:sumdom:i:eq2}
     \frac{\union{\subcube{}}{m}{\ell}{i-1}{\optimalj+1}{d}}{\union{\subcube{}}{m}{\ell}{i}{\optimalj}{d}} = \frac{\binom{\ell}{i-1}\binom{\ell-i+1}{d}}{\binom{\ell}{i}\binom{\ell-i}{d}} \frac{\left(\frac{1}{2}\right)^{\indicator{i-1 = \optimalj+1}}}{\left(\frac{1}{2}\right)^{\indicator{i = \optimalj}}} \leq \frac{i(\ell-i+1)}{(\ell-i+1)(\ell-i-d+1)} = \frac{i}{\optimalj+1}.
\end{align}
From the definition \eqref{eq:ffuncred:i} of  $\ffuncredi{\subcube{}}{t}{\cdot}$ as well as from \eqref{eq:sumdom:i:eq1} and \eqref{eq:sumdom:i:eq2} we have
\begin{align}
    \frac{\ffuncredi{\subcube{}}{t}{i-1}}{\ffuncredi{\subcube{}}{t}{i}} \overset{\eqref{eq:ffuncred:i}}&{=}  \frac{\upboundcubeintspan{i-1}\upboundcubeintspan{\optimalj+1}\union{\subcube{}}{m}{\ell}{i-1}{\optimalj+1}{d}}{\upboundcubeintspan{i}\upboundcubeintspan{\optimalj}\union{\subcube{}}{m}{\ell}{i}{\optimalj}{d}}  \notag\\ \overset{\eqref{eq:sumdom:i:eq1}, \eqref{eq:sumdom:i:eq2}}&{\leq} k^{-\frac{i-\optimalj-2}{2}} \frac{ 9(\optimalj+1)}{10i} \frac{i}{\optimalj+1} \notag\\ &=   k^{-\frac{i-\optimalj-2}{2}} \frac{9}{10}\label{eq:ffunc:diagonalbound1}\\
   &<  k^{-\frac{i-\optimalj-2}{2}}.\label{eq:ffunc:diagonalbound}
\end{align}
From \eqref{eq:ffunc:diagonalbound1} we obtain
\begin{align}\label{eq:ffunc:diagonalbound2}
    \frac{\ffuncredi{\subcube{}}{t}{i-1}}{\ffuncredi{\subcube{}}{t}{i}}  &\leq \frac{9}{10}, 
\end{align}
because $i-1 \geq \optimalj+1$ on the domain where $\ffuncredi{\subcube{}}{t}{i-1}$ is non-zero. Thus the function $i\mapsto \ffuncredi{\subcube{}}{t}{i}$ is increasing in $i$ and it remains to determine the largest value of $i$ at which $\ffuncredi{\subcube{}}{t}{i}$ is non-zero. We proceed now with a case distinction, starting with the proof of \eqref{eq:lem:sumdom:i:secondpart}. 

To prove part (a)  take $ \subcube{} =G$ and  $t=\criticaldimension$. In this case, the largest value $i$ may take is at $i = \criticaldimension-1$ at which $\optimalj := \ell -i-d = \ell-\criticaldimension-d+1$. Checking all conditions from the definition of the set $\indextriplesum{G}{\criticaldimension}$ of admissible indices it indeed holds that $(\ell,\criticaldimension-1,\optimalj,d) \in \indextriplesum{G}{\criticaldimension}$ and thus it follows from \eqref{eq:ffunc:diagonalbound2} that 
\begin{align}\label{eq:sumdom:i:eq3}
     \sum_{i=0}^{\criticaldimension-1} \ffuncredi{G}{\criticaldimension}{i} \overset{\eqref{eq:ffunc:diagonalbound2}}{\leq} \ffuncredi{G}{\criticaldimension}{\criticaldimension-1} + \ffuncredi{G}{\criticaldimension}{\criticaldimension-2} \sum_{i=0}^{\criticaldimension-2} \left(\frac{9}{10}\right)^i \leq \ffuncredi{G}{\criticaldimension}{\criticaldimension-1} + 10\ffuncredi{G}{\criticaldimension}{\criticaldimension-2},
\end{align}
where the last inequality follows from a standard bound on geometric series. 
Furthermore, it follows from \eqref{eq:ffunc:diagonalbound} (by taking $H=G$, $t = \criticaldimension$, $i=\criticaldimension-1$ and $\optimalj := \ell -i-d$) that
\begin{align}
    \frac{\ffuncredi{G}{\criticaldimension}{\criticaldimension-2}}{\ffuncredi{G}{\criticaldimension}{\criticaldimension-1}} \overset{\eqref{eq:ffunc:diagonalbound}}{\leq}  k^{-\frac{(\criticaldimension-1)-(\ell-\criticaldimension+1-d)-2}{2}} \leq k^{-\frac{\sqrt{n}}{4}},
    \label{eq:ffuncredi6}
\end{align} 
where the last inequality is because $(\criticaldimension-1)-(\ell-\criticaldimension+1-d) -2\geq 2 \criticaldimension-\ell-4 \geq \frac{\sqrt{n}}{2}$.
Combining \eqref{eq:sumdom:i:eq3} and \eqref{eq:ffuncredi6} implies
\begin{align*}
     \sum_{i=0}^{\criticaldimension-1} \ffuncredi{G}{\criticaldimension}{i} 
     \overset{\eqref{eq:sumdom:i:eq3}}{\leq}\ffuncredi{G}{\criticaldimension}{\criticaldimension-1}\left(1+ \frac{10\ffuncredi{G}{\criticaldimension}{\criticaldimension-2}}{\ffuncredi{G}{\criticaldimension}{\criticaldimension-1}} \right) 
      \overset{\eqref{eq:ffuncredi6}}&{\leq}\ffuncredi{G}{\criticaldimension}{\criticaldimension-1} \left(1+10 k^{-\frac{\sqrt{n}}{4}}\right)  \notag\\ 
      &\leq \ffuncredi{G}{\criticaldimension}{\criticaldimension-1} \left(1+k^{-\frac{\sqrt{n}}{6}}\right),
\end{align*} and thus obtaining \eqref{eq:lem:sumdom:i:secondpart}. 

To prove part (b) let  $\subcube{}$ be a projection of dimension $m$ and assume $15 \leq m \leq \criticaldimension$. Consider the function $\ffuncredi{\subcube{}}{m}{i}$ as defined in \eqref{eq:ffuncred:i} by taking $\ell=m$. Recalling that $i$ now may take values from $0$ to $m -1$, it follows from the definition of the set $\indextriplesum{G}{\criticaldimension}$ of admissible indices that $\ffuncredi{\subcube{}}{m}{i}$ is only zero in the case $i = m-1$ and $d=2$. Thus, $\ffuncredi{\subcube{}}{m}{i}$ takes its largest value at $i = m -1$ if $d \in \set{0,1}$, and  at $i = m-2$ if $d=2$. Hence, it follows (analogously to \eqref{eq:sumdom:i:eq3}) from \eqref{eq:ffunc:diagonalbound2}  that
\begin{align}\label{eq:ffuncredi7}
     \sum_{i=0}^{m-1} \ffuncredi{\subcube{}}{m}{i} \overset{\eqref{eq:ffunc:diagonalbound2}}&{\leq} \ffuncredi{\subcube{}}{m}{m-1-\indicator{d=2}} +\ffuncredi{\subcube{}}{m}{m-2-\indicator{d=2}} \sum_{i=0}^{m-2} \left(\frac{9}{10}\right)^i \notag\\
     &\leq \ffuncredi{\subcube{}}{m}{m-1-\indicator{d=2}} + 10\ffuncredi{\subcube{}}{m}{m-2-\indicator{d=2}}.
\end{align}
Furthermore, if  $d \in \set{0,1}$, we obtain from \eqref{eq:ffunc:diagonalbound}  (by taking $t = m$, $i = m-1$ and $\optimalj := \ell -i-d = m-i-d = 1-d$)  that  \begin{align}
   \frac{\ffuncredi{\subcube{}}{m}{m-2}}{\ffuncredi{\subcube{}}{m}{m-1}} \overset{\eqref{eq:ffunc:diagonalbound}}{\leq}  k^{-\frac{m-1-(1-d)-2}{2}} \leq k^{-\frac{m-4}{2}},\label{eq:ffuncredi1}
\end{align} and, if $d = 2$,  it also follows from \eqref{eq:ffunc:diagonalbound} (by taking $t = m$, $i = m-2$ and $\optimalj := \ell -i-d = m-i-2 = 0$)  that 
\begin{align}
   \frac{\ffuncredi{\subcube{}}{m}{m-3}}{\ffuncredi{\subcube{}}{m}{m-2}} \overset{\eqref{eq:ffunc:diagonalbound}}{\leq}  k^{-\frac{m-4}{2}}.\label{eq:ffuncredi2}
\end{align} 
Combining \eqref{eq:ffuncredi7}, \eqref{eq:ffuncredi1} and  \eqref{eq:ffuncredi2} implies
\begin{align*}
    \sum_{i=0}^{m-1} \ffuncredi{\subcube{}}{m}{i} \overset{\eqref{eq:ffuncredi7}}&{\leq} \ffuncredi{\subcube{}}{m}{m-1-\indicator{d=2}}\left(1+ \frac{10\ffuncredi{\subcube{}}{m}{m-2-\indicator{d=2}}}{\ffuncredi{\subcube{}}{m}{m-1-\indicator{d=2}}} \right)  \\
    \overset{\eqref{eq:ffuncredi1},\eqref{eq:ffuncredi2}}&{\leq} \ffuncredi{\subcube{}}{m}{m-1-\indicator{d=2}}\left(1+ 10 k^{-\frac{m-4}{2}} \right) \\
    &\leq \ffuncredi{\subcube{}}{m}{m-1-\indicator{d=2}}\left(1+ k^{-\frac{m-11}{2}} \right),
\end{align*}where we used the fact that $10 \leq 2^\frac{7}{2} \leq k^\frac{7}{2}$, completing the proof.
\end{proof}
Next, we consider all the possible values of $d\in \{0,1,2\}$. 

\begin{lem}\label{lem:sumdom:d}
(a)    For $\criticaldimension < \ell \leq 3\sqrt{n} $ it holds that
   \begin{align}\label{eq:lem:sumdom:d:secondpart}
        \sum_{d=0}^2 \ffunc{G}{\criticaldimension}{\ell}{\criticaldimension-1}{\ell-(\criticaldimension-1)-d}{d} \leq \left(1+  k^{-\frac{\sqrt{n}}{6}}\right)\ffunc{G}{\criticaldimension}{\ell}{\criticaldimension-1}{\ell-\criticaldimension-1}{2}.
    \end{align}
(b) For $15 \leq m \leq \criticaldimension$ let $\subcube{}$ be a \projection{} of dimension $m$. Then we have
    \begin{align}
        \sum_{d=0}^2 \ffunc{\subcube{}}{m}{m}{m-1-\indicator{d=2}}{1-d+\indicator{d=2}}{d} \leq \left(1+\frac{3}{n} \right) \ffunc{\subcube{}}{m}{m}{m-2}{0}{2}.\label{eq:lem:sumdom:d:secondpartb}
    \end{align}   
\end{lem}
\begin{proof}
Let us start with \eqref{eq:lem:sumdom:d:secondpart}. Note first that the function $d\mapsto \union{G}{\criticaldimension}{\ell}{\criticaldimension-1}{\ell-(\criticaldimension-1)-d}{d}$ is clearly increasing in $d$.
   Using \Cref{lem:ratio:P(j):estimate} it thus follows that for $d \in \{0,1\}$,
   \begin{align*}
       \frac{\ffunc{G}{\criticaldimension}{\ell}{\criticaldimension-1}{\ell-(\criticaldimension-1)-d}{d}}{\ffunc{G}{\criticaldimension}{\ell}{\criticaldimension-1}{\ell-(\criticaldimension-1)-d-1}{d+1}} \leq \frac{\upboundcubeintspan{\ell-\criticaldimension-d+1}}{\upboundcubeintspan{\ell-\criticaldimension-d}} \leq \frac{\ell-\criticaldimension-d+1}{n} k^{\frac{\ell-\criticaldimension-d+5-2\sqrt{n}}{2}} \leq \frac{1}{2}k^{-\frac{\sqrt{n}}{6}}, 
   \end{align*}where the last inequality uses $\ell-\criticaldimension-d+1 \leq \sqrt{n}+4$,  immediately implying \eqref{eq:lem:sumdom:d:secondpart}.

   Moving on to the case $15 \leq m \leq \criticaldimension$, an analogous argument shows that 
   \begin{align}\label{eq:sumdom:d:eq1}
       \frac{\ffunc{\subcube{}}{m}{m}{m-1}{1}{0}}{\ffunc{\subcube{}}{m}{m}{m-1}{0}{1}} \leq \frac{1}{2}.
   \end{align} For the ratio $ \frac{\ffunc{\subcube{}}{m}{m}{m-1}{0}{1}}{\ffunc{\subcube{}}{m}{m}{m-2}{0}{2}}$ we need to be a bit more precise. A direct computation shows that 
   \begin{align}\label{eq:sumdom:d:eq2}
      \frac{\union{\subcube{}}{m}{m}{m-1}{0}{1}}{\union{\subcube{}}{m}{m}{m-2}{0}{2}} =  \frac{2}{(k-1)(m-1)} 
   \end{align}and, being a bit more careful in the estimates in \Cref{lem:ratio:P(j):estimate}, it can be shown that 
    \begin{align}\label{eq:sumdom:d:eq3}
        \frac{\upboundcubeintspan{m-1}}{\upboundcubeintspan{m-2}} \leq \frac{(k-1)(m-1)}{n}k^{\frac{m+2-2\sqrt{n}}{2}}.
    \end{align}
Putting these together, we obtain
    \begin{align}\label{eq:sumdom:d:eq4}
        \frac{\ffunc{\subcube{}}{m}{m}{m-1}{0}{1}}{\ffunc{\subcube{}}{m}{m}{m-2}{0}{2}} \overset{\eqref{def:ffunc2}}&{=} \frac{\union{\subcube{}}{m}{m}{m-1}{0}{1}}{\union{\subcube{}}{m}{m}{m-2}{0}{2}}\frac{\upboundcubeintspan{m-1}}{\upboundcubeintspan{m-2}}\notag \\ \overset{\eqref{eq:sumdom:d:eq2}}&{\leq} \frac{2}{(k-1)(m-1)} \frac{\upboundcubeintspan{m-1}}{\upboundcubeintspan{m-2}} \notag \\
        \overset{\eqref{eq:sumdom:d:eq3}}&{\leq}  \frac{2}{(k-1)(m-1)} \frac{(k-1)(m-1)}{n}k^{\frac{m+2-2\sqrt{n}}{2}} \leq \frac{2}{n},
    \end{align} where the last inequality uses $m \leq \criticaldimension:=\rounddown{2\sqrt{n}}-2 \leq 2\sqrt{n}-2$. Combining  \eqref{eq:sumdom:d:eq1} and \eqref{eq:sumdom:d:eq4} now yields
    \begin{align*}
         \sum_{d=0}^2 \ffunc{\subcube{}}{m}{m}{m-1-\indicator{d=2}}{1-d+\indicator{d=2}}{d} \overset{\eqref{eq:sumdom:d:eq1}}&{\leq} \frac{3}{2}\ffunc{\subcube{}}{m}{m}{m-1}{0}{1} + \ffunc{\subcube{}}{m}{m}{m-2}{0}{2} 
         \\\overset{\eqref{eq:sumdom:d:eq4}}&{\leq} \left(1+\frac{3}{n}\right)\ffunc{\subcube{}}{m}{m}{m-2}{0}{2},
    \end{align*} completing the proof. 
\end{proof}

Now we have all the tools to show \Cref{lem:triplesumdomination}.
\begin{proof}[Proof of \Cref{lem:triplesumdomination}]
Let $15\leq m \leq \criticaldimension$ and let $\subcube{}$ be a \projection{} of dimension $m$. Observe first that by definition of $\indextriplesum{\subcube{}}{m}$, any indices $(\ell,i,j,d) \in \indextriplesum{\subcube{}}{m}$ have to satisfy $\ell = m,\ d \in \set{0,1,2}$ and $ 0\leq j \leq i \leq \ell-1$. 
It thus follows from \Cref{lem:sumdom:j} that 
\begin{align}\label{eq:proofof_triplesumdom:eq1}
    \sum_{(\ell,i,j,d) \in \indextriplesum{\subcube{}}{m}} \ffunc{\subcube{}}{m}{\ell}{i}{j}{d} 
    &=  \sum_{d=0}^2 \sum_{i=0}^{m-1} \sum_{j=0}^{i} \ffunc{\subcube{}}{m}{m}{i}{j}{d} \nonumber
    \\\overset{\eqref{eq:varyingj}}&{\leq}  \left(1+k^{-\frac{\sqrt{n}}{6}}\right) \sum_{d=0}^2 \sum_{i=0}^{m-1} \ffunc{\subcube{}}{m}{m}{i}{m-i-d}{d}.
\end{align}
Furthermore, by Lemmas \ref{lem:sumdom:i} (b) and \ref{lem:sumdom:d} (b) we obtain
\begin{align}
    \sum_{d=0}^2 \sum_{i=0}^{m-1} \ffunc{\subcube{}}{m}{m}{i}{m-i-d}{d} \nonumber\overset{\eqref{eq:lem:sumdom:i:secondpartb}}&{\leq} \left(1+k^{-\frac{m-11}{2}}\right) \sum_{d=0}^2 \ffunc{\subcube{}}{m}{m}{m-1-\indicator{d=2}}{1-d+\indicator{d=2}}{d} \notag
    \\ \overset{\eqref{eq:lem:sumdom:d:secondpartb}}&{\leq}  \left(1+ k^{-\frac{m
        -11}{2}}\right)\left(1+ \frac{3}{n}\right) \ffunc{\subcube{}}{m}{m}{m-2}{0}{2}. \label{eq:proofof_triplesumdom:eq2}
\end{align}
Combining \eqref{eq:proofof_triplesumdom:eq1} and \eqref{eq:proofof_triplesumdom:eq2} yields 
\begin{align*}
\sum_{(\ell,i,j,d) \in \indextriplesum{\subcube{}}{m}} \ffunc{\subcube{}}{m}{\ell}{i}{j}{d}  &\leq  \left(1+k^{-\frac{\sqrt{n}}{6}}\right) \left(1+ k^{-\frac{m-11}{2}}\right)\left(1+ \frac{3}{n}\right) \ffunc{\subcube{}}{m}{m}{m-2}{0}{2} \\
&\leq \left(1+ k^{-\frac{m-11}{2}}\right)\left(1+ \frac{4}{n}\right) \ffunc{\subcube{}}{m}{m}{m-2}{0}{2},
\end{align*} where the last inequality is because  $\left(1+k^{-\frac{\sqrt{n}}{6}}\right) \left(1+ \frac{3}{n}\right) \leq \left(1+ \frac{4}{n}\right)$ holds for large enough $n$. 
\end{proof}

\Cref{lem:triplesumdomination} now enables us to prove \Cref{lem:probintspanupperbound}. 
\begin{proof}[Proof of \Cref{lem:probintspanupperbound}]
We will show the statement by induction on $m$. The case $0\leq m \leq 14$ can easily be proved by a tedious enumeration and will be omitted. Below we assume  $15 \leq m \leq \criticaldimension$. Let $\subcube{}$ be a projection of dimension $m$ and assume that $\probcubeintspan{m'} \leq \upboundcubeintspan{m'}$ for all $m' < m$.  By \Cref{lem:perc_unionbound_reformulation} and using the induction hypothesis we obtain analogously as in \eqref{eq:triplesumbound} that 
    \begin{align*}
        \probcubeintspan{m} &\leq \sum_{(\ell,i,j,d) \in \indextriplesum{\subcube{}}{m}} \union{\subcube{m}}{m}{\ell}{i}{j}{d}\probcubeintspan{i} \probcubeintspan{j} \leq \sum_{(\ell,i,j,d) \in \indextriplesum{\subcube{}}{m}} \union{\subcube{}}{m}{\ell}{i}{j}{d}\upboundcubeintspan{i} \upboundcubeintspan{j} \\
        &= \sum_{(\ell,i,j,d) \in \indextriplesum{\subcube{}}{m}} \ffunc{\subcube{}}{m}{\ell}{i}{j}{d}.
    \end{align*}
    Furthermore, \Cref{lem:triplesumdomination} implies
    \begin{align*}
       \sum_{(\ell,i,j,d) \in \indextriplesum{\subcube{}}{m}}  \ffunc{\subcube{}}{m}{\ell}{i}{j}{d} \leq \left(1+ k^{-\frac{m
        -11}{2}}\right)\left(1+ \frac{4}{n}\right) \ffunc{\subcube{}}{m}{m}{m-2}{0}{2}, 
    \end{align*} yielding
    \begin{align}\label{eq:proof:upboundcubeintspan:eq2}
        \probcubeintspan{m}  \leq \left(1+ k^{-\frac{m
        -11}{2}}\right)\left(1+ \frac{4}{n}\right) \ffunc{\subcube{}}{m}{m}{m-2}{0}{2}.
    \end{align}
    By the definition of $\upboundcubeintspan{\cdot}$ and by \Cref{lem:U()computation} it holds that
    \begin{align}
        \ffunc{\subcube{}}{m}{m}{m-2}{0}{2} &= \upboundcubeintspan{m-2} \upboundcubeintspan{0} \union{\subcube{}}{m}{m}{m-2}{0}{2} \notag\\
        &= \upboundcubeintspanconstant{m-2}p^{\frac{m-2}{2}+1} (m-2)!\  2^{-\frac{m-2}{2}}(k-1)^{m-2} k^\frac{(m-2)^2+2(m-2)}{4} p \binom{m}{2} (k-1)^2 k^m \notag\\
        &=  \upboundcubeintspanconstant{m-2 }p^{\frac{m}{2}+1} m!\  2^{-\frac{m}{2}}(k-1)^{m} k^\frac{m^2+2m}{4} \notag\\
        &= \frac{\upboundcubeintspanconstant{m-2}}{\upboundcubeintspanconstant{m}} \upboundcubeintspan{m} \label{eq:proof:upboundcubeintspan:eq1} . 
    \end{align}
    Using  $\left(1+k^{-\frac{m-11}{2}}\right)\left(1+\frac{4}{n} \right) = \frac{\upboundcubeintspanconstant{m}}{\upboundcubeintspanconstant{m-1}} \leq \frac{\upboundcubeintspanconstant{m}}{\upboundcubeintspanconstant{m-2}} $, it follows together with \eqref{eq:proof:upboundcubeintspan:eq1}  that
   \begin{align}\label{eq:proof:upboundcubeintspan:eq3}
    &\left(1+k^{-\frac{m-11}{2}}\right)\left(1+\frac{4}{n} \right) \ffunc{\subcube{}}{m}{m}{m-2}{0}{2} 
    \leq \frac{\upboundcubeintspanconstant{m}}{\upboundcubeintspanconstant{m-2}} \frac{\upboundcubeintspanconstant{m-2}}{\upboundcubeintspanconstant{m}} \upboundcubeintspan{m} \leq \upboundcubeintspan{m}.
\end{align}
Combining \eqref{eq:proof:upboundcubeintspan:eq2} and \eqref{eq:proof:upboundcubeintspan:eq3} we obtain
\begin{align*}
    \probcubeintspan{m} \overset{\eqref{eq:proof:upboundcubeintspan:eq2}}{\leq} \left(1+ k^{-\frac{m
        -11}{2}}\right)\left(1+ \frac{4}{n}\right) \ffunc{\subcube{}}{m}{m}{m-2}{0}{2} \overset{\eqref{eq:proof:upboundcubeintspan:eq3}}{\leq}\upboundcubeintspan{m},
\end{align*} completing the induction and thus the proof. 
\end{proof}

Next, we show that the functions $\upboundcubeintspanconstant{m}$ are bounded from above by a universal constant.
\begin{lem}\label{lem:upperboundcubeintspanconst}
   For all $m \in [\criticaldimension]$ we have
    \begin{align*}
        \upboundcubeintspanconstant{m} = \Th{1}.
    \end{align*}
\end{lem}
\begin{proof}
It is easy to see that the function $m\mapsto \upboundcubeintspanconstant{m}$ is increasing in $m$. Thus, it suffices to bound $\upboundcubeintspanconstant{\criticaldimension}$ from above and we obtain  
    \begin{align*}
    \upboundcubeintspanconstant{\criticaldimension} &= \frac{1}{2}\left(\frac{9 \sqrt{2}}{10}\right)^{14} \prod_{m=15}^{\criticaldimension} \left(1+ k^{-\frac{m-11}{2}}\right)\left(1+ \frac{4}{n}\right) \notag \\
        &\leq \Theta(1) \left(1+ \frac{4}{n}\right)^{\criticaldimension} \prod_{m=15}^{\criticaldimension} \expp{k^{-\frac{m-11}{2}}} \notag \\
        &\leq \Theta(1) \expp{\frac{4 \criticaldimension}{n}}\expp{\sum_{m = 0}^{\infty} k^{-\frac{m}{2}}} = \Theta(1),
\end{align*} where the last equality follows because $\criticaldimension:=\rounddown{2\sqrt{n}}-2 \leq n$. 
\end{proof}
We continue by proving \Cref{claim:expec_to_zero_largeterms}. 

\begin{proof}[Proof of \Cref{claim:expec_to_zero_largeterms}]
We aim to bound the function $\ffunc{G}{\criticaldimension}{\ell}{i}{j}{d}$ for all indices $(\ell,i,j,d) \in \indextriplesum{G}{\criticaldimension}^{(2)}$. 
Using  crude bounds we get
\begin{align}\label{eq:prf_of_expec_to_zero_largeterms:1}
    \union{G}{\criticaldimension}{\ell}{i}{j}{d} &\leq   \binom{n}{\ell} \binom{\ell}{i} \binom{\ell-i}{d} \binom{i}{j+i+d-\ell} k^{n+\ell-i-j-d} (k-1)^d \notag \\
        &\leq n^\ell \ell^{i+d} i^{j+i+d-\ell} k^{n+\ell} \notag \\
        &\leq n^{3\ell}k^{n+\ell}.
\end{align}
On the other hand, by the definition \eqref{def:upboundcubeintspan} of $\upboundcubeintspan{\cdot}$ and using $j! \leq i! \leq n^{\ell}$ and \Cref{lem:upperboundcubeintspanconst} we obtain 
\begin{align*}
    \upboundcubeintspan{i} \upboundcubeintspan{j} \overset{\eqref{def:upboundcubeintspan}}&{=} \upboundcubeintspanconstant{i}\upboundcubeintspanconstant{j} p^{\frac{i+j}{2}+2} (i!)(j!) 2^{-\frac{i+j}{2}}(k-1)^{2(i+j)} k^{\frac{1}{4}\left(i^2+j^2+2i+2j\right)} \\
    &\leq \Th{1} p^{\frac{i+j}{2}+2} n^{2\ell} k^{\frac{1}{4}\left(i^2+j^2+10i+10j\right)}.
\end{align*}
Furthermore, using $p \leq p_* \leq k^{-2\sqrt{n}+1}$ it follows that
\begin{align}\label{eq:prf_of_expec_to_zero_largeterms:2}
    \upboundcubeintspan{i} \upboundcubeintspan{j} \leq \Th{1} n^{2\ell} k^{\frac{1}{4}\left(i^2+j^2-(i+j)(4\sqrt{n}-12)+16\sqrt{n}-8\right)}.
\end{align}
In order to obtain a universal upper bound of \eqref{eq:prf_of_expec_to_zero_largeterms:2} for all indices $(i,j)$ under consideration, we focus on  the exponent of \eqref{eq:prf_of_expec_to_zero_largeterms:2} and consider the function 
$$h: \mathbb R^2\to \mathbb R, (x,y)\mapsto x^2+y^2-(x+y)(4\sqrt{n}-12)+16\sqrt{n}-8.$$
Note that the function $h(x,y)$ defines a circular paraboloid whose center is shifted to $(2\sqrt{n}-6,2\sqrt{n}-6)$.
Recalling $(\ell,i,j,d) \in \indextriplesum{G}{\criticaldimension}^{(2)}$, where 
\begin{align*}
    \indextriplesum{G}{\criticaldimension}^{(2)} \coloneqq \setbuilder{(\ell,i,j,d)\in \N_0^4}{\left(3 \sqrt{n} < \ell \leq \minp{}{i+j+d,\dimension{G}}\right) \land (j \leq i <\criticaldimension) \land (i+d <\ell) \land ( d \in \{0,1,2\})},
\end{align*} 
we note that the function $h(x,y)$ should however be considered only for $(i,j)\in \tilde I$, where 
$$\tilde I:=\setbuilder{(i,j)\in \N_0^2}{(j \leq i < \criticaldimension )\wedge (i+j \geq \ell -d \geq 3\sqrt{n}-1) \text{ for }d\in \{0,1,2\}}.$$ 

Hence, the (shifted circular paraboloid) function $h:\tilde I\to \mathbb R$ takes its maximum when $i$ is maximal and $j$ is minimal possible  on the domain $\tilde I$, in other words, the maximiser of the function $h$ restricted on the domain $\tilde I$ is given by $i = \criticaldimension-1 =\rounddown{2\sqrt{n}}-3$ and $j= \roundup{3\sqrt{n}}-1- i = \roundup{3\sqrt{n}}-\rounddown{2\sqrt{n}}+2$. Relaxing the domain  $\tilde I$ of $h$ to the reals, it follows that at the function $h:\tilde I\to \mathbb R$ is upper bounded by 
$h(2\sqrt{n}-3,\sqrt{n}+2) = -7n +48\sqrt{n}-7 \leq -6n$, implying together with \eqref{eq:prf_of_expec_to_zero_largeterms:2}  that 
\begin{align}\label{eq:prf_of_expec_to_zero_largeterms:3}
    \upboundcubeintspan{i} \upboundcubeintspan{j} \leq \Th{1} n^{2\ell} k^{-\frac{3n}{2}}.
\end{align}
Combining \eqref{eq:prf_of_expec_to_zero_largeterms:1} and \eqref{eq:prf_of_expec_to_zero_largeterms:3} and plugging into the definition \eqref{def:ffunc2} of $f_{G,\criticaldimension}$, we obtain that for large enough $n$ and for all $(\ell,i,j,d) \in \indextriplesum{G}{\criticaldimension}^{(2)}$,  
\begin{align}\label{eq:prf_of_expec_to_zero_largeterms:4}
     \ffunc{G}{\criticaldimension}{\ell}{i}{j}{d} \overset{\eqref{def:ffunc2}}&{=} \upboundcubeintspan{i} \upboundcubeintspan{j} \union{G}{\criticaldimension}{\ell}{i}{j}{d} \notag\\ 
     \overset{\eqref{eq:prf_of_expec_to_zero_largeterms:1},\eqref{eq:prf_of_expec_to_zero_largeterms:3}}&{\leq} \Th{1} n^{5\ell} k^{-\frac{n}{2}+\ell}  \leq  n^{20 \sqrt{n}} k ^{-\frac{n}{3}} = o(n^{-3}),
\end{align} where the penultimate inequality uses that $\ell \leq i+j+d \leq 2\criticaldimension \leq 4\sqrt{n}$.
Hence, by the definition of $\indextriplesum{G}{\criticaldimension}^{(2)}$ and \eqref{eq:prf_of_expec_to_zero_largeterms:4} it holds that 
\begin{align*}
\sum_{(\ell,i,j,d) \in \indextriplesum{G}{\criticaldimension}^{(2)}} \ffunc{G}{\criticaldimension}{\ell}{i}{j}{d} &= 
    \sum_{\ell = 3\sqrt{n}+1}^{4\sqrt{n}} \sum_{d = 0}^2 \sum_{i= 0}^{\criticaldimension-1} \sum_{j= 0}^{i} \ffunc{G}{\criticaldimension}{\ell}{i}{j}{d}\\
    &\leq 3 n^3 \max_{(\ell,i,j,d) \in \indextriplesum{G}{\criticaldimension}^{(2)}} \ffunc{G}{\criticaldimension}{\ell}{i}{j}{d} \overset{\eqref{eq:prf_of_expec_to_zero_largeterms:4}}{=} o(1).
\end{align*}
\end{proof}

Moving on, we now aim to show \Cref{lem:triplesumdomination_part2}. As mentioned before, the proof of \Cref{lem:triplesumdomination_part2} goes along similar lines as the proof of \Cref{lem:triplesumdomination}, however in addition to dealing with different values of $i,j$ and $d$, we now also have to consider different values of $\ell$.
\begin{lem}\label{lem:sumdom:l}
It holds that 
\begin{align}\label{eq:sumdom:l}
    \sum_{\ell=\criticaldimension+1}^{3\sqrt{n}} \ffunc{G}{\criticaldimension}{\ell}{\criticaldimension-1}{\ell-\criticaldimension-1}{2} \leq 2 k^{\frac{1}{2}}\ffunc{G}{\criticaldimension}{\criticaldimension}{\criticaldimension-2}{0}{2}.
\end{align}    
\end{lem}
\begin{proof}
To ease notation, let us define for $\criticaldimension+1\leq \ell \leq 3\sqrt{n}$
\begin{align*}
    \ffuncredl{\ell} \coloneqq \ffunc{G}{\criticaldimension}{\ell}{\criticaldimension-1}{\ell-\criticaldimension-1}{2}
\end{align*} and consider the ratio
    \begin{align*}
       \frac{ \ffuncredl{\ell+1}}{ \ffuncredl{\ell}} =  \frac{\upboundcubeintspan{\ell-\criticaldimension}}{\upboundcubeintspan{\ell-\criticaldimension-1}} \frac{\union{G}{\criticaldimension}{\ell+1}{\criticaldimension-1}{\ell-\criticaldimension}{2}}{\union{G}{\criticaldimension}{\ell}{\criticaldimension-1}{\ell-\criticaldimension-1}{2}}.
    \end{align*} 

Using \Cref{lem:ratio:P(j):estimate} and $\ell-\criticaldimension \leq  \sqrt{n}+2$ we obtain
\begin{align}\label{eq:sumdom:l:eq1}
    \frac{\upboundcubeintspan{\ell-\criticaldimension}}{\upboundcubeintspan{\ell-\criticaldimension-1}} \leq \frac{(\ell-\criticaldimension)}{n}k^{\frac{1}{2}\left(\ell-\criticaldimension+4-2\sqrt{n}\right)} \leq k^{-\frac{\sqrt{n}}{3}}. 
\end{align}
Furthermore, we get  
\begin{align*}
    \frac{\union{G}{\criticaldimension}{\ell+1}{\criticaldimension-1}{\ell-\criticaldimension}{2}}{\union{G}{\criticaldimension}{\ell}{\criticaldimension-1}{\ell-\criticaldimension-1}{2}} = \frac{\binom{n}{\ell+1}\binom{\ell+1}{\criticaldimension-1}\binom{\ell-\criticaldimension+2}{2}}{\binom{n}{\ell}\binom{\ell}{\criticaldimension-1}\binom{\ell-\criticaldimension+1}{2}} = \frac{n-\ell}{\ell-\criticaldimension} \leq n,
\end{align*} implying together with \eqref{eq:sumdom:l:eq1} 
\begin{align*}
      \frac{ \ffuncredl{\ell+1}}{ \ffuncredl{\ell}} \leq n k^{-\frac{\sqrt{n}}{3}} \leq \frac{1}{2}.
\end{align*}
It thus follows that 
\begin{align}\label{eq:sumdom:l:eq2}
    \sum_{\ell=\criticaldimension+1}^{3\sqrt{n}} \ffuncredl{\ell} \leq 2 \ffuncredl{\criticaldimension+1}.
\end{align} 
On the other hand, using Lemmas \ref{lem:U()computation} and \ref{lem:ratio:P(j):estimate} we obtain
\begin{align*}
    \frac{\ffuncredl{\criticaldimension+1}}{\ffunc{G}{\criticaldimension}{\criticaldimension}{\criticaldimension-2}{0}{2}} &= \frac{\upboundcubeintspan{\criticaldimension-1}}{\upboundcubeintspan{\criticaldimension-2}} \frac{\union{G}{\criticaldimension}{\criticaldimension+1}{\criticaldimension-1}{0}{2}}{\union{G}{\criticaldimension}{\criticaldimension}{\criticaldimension-2}{0}{2}}\\
    &\leq \frac{\criticaldimension-1}{n}k^{\frac{1}{2}\left(\criticaldimension+3-2\sqrt{n}\right)} \frac{\binom{n}{\criticaldimension+1} \binom{\criticaldimension+1}{2}}{\binom{n}{\criticaldimension} \binom{\criticaldimension}{2}} \leq \frac{(n-\criticaldimension)k^{\frac{1}{2}}}{n} \leq k^{\frac{1}{2}},
\end{align*} where the penultimate inequality uses $\criticaldimension := \rounddown{2\sqrt{n}}-2 \leq 2\sqrt{n}-2$.

Combining with  \eqref{eq:sumdom:l:eq2} we thus obtain
\begin{align*}
    \sum_{\ell=\criticaldimension+1}^{3\sqrt{n}} \ffuncredl{\ell} \leq 2 \ffuncredl{\criticaldimension+1} \leq 2 k^{\frac{1}{2}}\ffunc{G}{\criticaldimension}{\criticaldimension}{\criticaldimension-2}{0}{2},
\end{align*} completing the proof.
\end{proof}

\Cref{lem:sumdom:l} provides us with the final ingredient to prove \Cref{lem:triplesumdomination_part2}. 

\begin{proof}[Proof of \Cref{lem:triplesumdomination_part2}]
Let us partition $\indextriplesum{G}{\criticaldimension}^{(1)} $ into  $\indextriplesum{G}{\criticaldimension}^{(1)} =\indextriplesum{G}{\criticaldimension}^{(a)} \cup \indextriplesum{G}{\criticaldimension}^{(b)}$, where 
\begin{align*}
 \indextriplesum{G}{\criticaldimension}^{(a)} \coloneqq \setbuilder{(\ell,i,j,d) \in \indextriplesum{G}{\criticaldimension}^{(1)}}{\ell = \criticaldimension} \ \text{ and }\ \indextriplesum{G}{\criticaldimension}^{(b)} \coloneqq \setbuilder{(\ell,i,j,d) \in \indextriplesum{G}{\criticaldimension}^{(1)}}{\ell \geq \criticaldimension+1}.
\end{align*}

Starting with the sum over indices in $\indextriplesum{G}{\criticaldimension}^{(b)}$, we observe that 
\begin{align}
    \sum_{(\ell,i,j,d) \in\indextriplesum{G}{\criticaldimension}^{(b)}}\ffunc{G}{\criticaldimension}{\ell}{i}{j}{d} = \sum_{\ell = \criticaldimension+1}^{3\sqrt{n}} \sum_{d=0}^2 \sum_{i=0}^{\criticaldimension-1} \sum_{j=0}^{i} \ffunc{G}{\criticaldimension}{\ell}{i}{j}{d}\label{eq:proofof:lem:triplesumdom}.
\end{align} 
Using \Cref{lem:sumdom:j} and \Cref{lem:sumdom:i} (a) we obtain  
\begin{align*}
    \sum_{\ell = \criticaldimension+1}^{3\sqrt{n}} \sum_{d=0}^2 \sum_{i=0}^{\criticaldimension-1} \sum_{j=0}^{i} \ffunc{G}{\criticaldimension}{\ell}{i}{j}{d} \overset{\eqref{eq:varyingj}}&{\leq} \left(1+k^{-\frac{\sqrt{n}}{6}}\right) \sum_{\ell = \criticaldimension+1}^{3\sqrt{n}} \sum_{d=0}^2 \sum_{i=0}^{\ell-1} \ffunc{G}{\criticaldimension}{\ell}{i}{\ell-i-d}{d} \\
     \overset{\eqref{eq:lem:sumdom:i:secondpart}}&{\leq} \left(1+k^{-\frac{\sqrt{n}}{6}}\right)^2 \sum_{\ell = \criticaldimension+1}^{3\sqrt{n}} \sum_{d=0}^2 \ffunc{G}{\criticaldimension}{\ell}{\criticaldimension-1}{\ell-(\criticaldimension-1)-d}{d},
\end{align*} 
and furthermore, by \Cref{lem:sumdom:d} (a) and \Cref{lem:sumdom:l}
\begin{align*}
    &\left(1+k^{-\frac{\sqrt{n}}{6}}\right)^2 \sum_{\ell = \criticaldimension+1}^{3\sqrt{n}} \sum_{d=0}^2 \ffunc{G}{\criticaldimension}{\ell}{\criticaldimension-1}{\ell-(\criticaldimension-1)-d}{d} \notag \\
        \overset{\eqref{eq:lem:sumdom:d:secondpart}}&{\leq} \left(1+k^{-\frac{\sqrt{n}}{6}}\right)^3 \sum_{\ell = \criticaldimension+1}^{3\sqrt{n}} \ffunc{G}{\criticaldimension}{\ell}{\criticaldimension-1}{\ell-\criticaldimension-1}{2} \\
     \overset{\eqref{eq:sumdom:l}}&{\leq}  \left(1+k^{-\frac{\sqrt{n}}{6}}\right)^3 2 k^{\frac{1}{2}}\ffunc{G}{\criticaldimension}{\criticaldimension}{\criticaldimension-2}{0}{2} \notag \\
        &\leq 3k^{\frac{1}{2}}\ffunc{G}{\criticaldimension}{\criticaldimension}{\criticaldimension-2}{0}{2},
\end{align*} which yields together with  \eqref{eq:proofof:lem:triplesumdom}
\begin{align}
     \sum_{(\ell,i,j,d) \in\indextriplesum{G}{\criticaldimension}^{(b)}}\ffunc{G}{\criticaldimension}{\ell}{i}{j}{d} \leq 3 k^{\frac{1}{2}}\ffunc{G}{\criticaldimension}{\criticaldimension}{\criticaldimension-2}{0}{2}. \label{eq:eq:proofof:lem:triplesumdom:eq2}
\end{align} 

To bound the sum over indices in $\indextriplesum{G}{\criticaldimension}^{(a)}$, observe first that by \Cref{lem:U()computation} it holds that for any indices $(\ell,i,j,d) \in \indextriplesum{G}{\criticaldimension}^{(a)}$ and any projection $\subcube{}$ of dimension $\criticaldimension$,
\begin{align*}
    \union{G}{\criticaldimension}{\ell}{i}{j}{d} = \binom{n}{\criticaldimension} k^{n-\criticaldimension} \union{\subcube{}}{\criticaldimension}{\ell}{i}{j}{d},
\end{align*} implying
\begin{align}\label{eq:relation:ffuncG_and_ffuncH}
    \ffunc{G}{\criticaldimension}{\ell}{i}{j}{d} = \binom{n}{\criticaldimension} k^{n-\criticaldimension} \ffunc{\subcube{}}{\criticaldimension}{\ell}{i}{j}{d}.
\end{align}
Furthermore, it follows by \Cref{lem:triplesumdomination} that 
\begin{align}
     \sum_{(\ell,i,j,d) \in\indextriplesum{G}{\criticaldimension}^{(a)}}\ffunc{\subcube{\criticaldimension}}{\criticaldimension}{\ell}{i}{j}{d} 
     &\leq \left(1+k^{-\frac{\criticaldimension-11}{2}} \right) \left(1+ \frac{4}{n}\right)\ffunc{\subcube{\criticaldimension}}{\criticaldimension}{\criticaldimension}{\criticaldimension-2}{0}{2} \notag \\ 
     &\leq 2 \ffunc{\subcube{\criticaldimension}}{\criticaldimension}{\criticaldimension}{\criticaldimension-2}{0}{2},\label{eq:relation:ffuncG1}
\end{align} yielding together with \eqref{eq:relation:ffuncG_and_ffuncH} that for large enough $n$
\begin{align}
     \sum_{(\ell,i,j,d) \in\indextriplesum{G}{\criticaldimension}^{(a)}}\ffunc{G}{\criticaldimension}{\ell}{i}{j}{d} 
      \overset{\eqref{eq:relation:ffuncG_and_ffuncH}}&{=} \binom{n}{\criticaldimension} k^{n-\criticaldimension}  \sum_{(\ell,i,j,d) \in\indextriplesum{G}{\criticaldimension}^{(a)}}\ffunc{\subcube{\criticaldimension}}{\criticaldimension}{\ell}{i}{j}{d} \notag \\
        \overset{\eqref{eq:relation:ffuncG1}}&{\leq} 2 \binom{n}{\criticaldimension} k^{n-\criticaldimension}  \ffunc{\subcube{\criticaldimension}}{\criticaldimension}{\criticaldimension}{\criticaldimension-2}{0}{2} \notag \\
     \overset{\eqref{eq:relation:ffuncG_and_ffuncH}}&{=} 2  \ffunc{G}{\criticaldimension}{\criticaldimension}{\criticaldimension-2}{0}{2}.\label{eq:relation:ffuncG2}
\end{align}

Finally, using $\indextriplesum{G}{\criticaldimension}^{(1)} =\indextriplesum{G}{\criticaldimension}^{(a)} \cup \indextriplesum{G}{\criticaldimension}^{(b)}$ and combining \eqref{eq:relation:ffuncG2} and \eqref{eq:eq:proofof:lem:triplesumdom:eq2} we have
\begin{align*}
    \sum_{(\ell,i,j,d) \in\indextriplesum{G}{\criticaldimension}^{(1)}}\ffunc{G}{\criticaldimension}{\ell}{i}{j}{d}  &= \sum_{(\ell,i,j,d) \in\indextriplesum{G}{\criticaldimension}^{(a)}}\ffunc{G}{\criticaldimension}{\ell}{i}{j}{d} +\sum_{(\ell,i,j,d) \in\indextriplesum{G}{\criticaldimension}^{(b)}}\ffunc{G}{\criticaldimension}{\ell}{i}{j}{d} \\
    &\leq 5 k^{\frac{1}{2}}\ffunc{G}{\criticaldimension}{\criticaldimension}{\criticaldimension-2}{0}{2},
\end{align*}as desired.
\end{proof}

We continue by showing \Cref{claim:ffuncSmaller_Upboundcubeintspan}.
\begin{proof}[Proof of \Cref{claim:ffuncSmaller_Upboundcubeintspan}]
 As a special case of  \eqref{eq:relation:ffuncG_and_ffuncH} it holds that  
    \begin{align*}
        \ffunc{G}{\criticaldimension}{\criticaldimension}{\criticaldimension-2}{0}{2} =   \binom{n}{\criticaldimension} k^{n-\criticaldimension}\ffunc{\subcube{\criticaldimension}}{\criticaldimension}{\criticaldimension}{\criticaldimension-2}{0}{2}.
    \end{align*} 
In addition, as a special case of \eqref{eq:proof:upboundcubeintspan:eq1} we have $\ffunc{\subcube{}}{\criticaldimension}{\criticaldimension}{\criticaldimension-2}{0}{2} \leq \frac{\upboundcubeintspanconstant{\criticaldimension-2}}{\upboundcubeintspanconstant{\criticaldimension}} \upboundcubeintspan{\criticaldimension}$ and thus
    \begin{align*}
        \ffunc{\subcube{\criticaldimension}}{\criticaldimension}{\criticaldimension}{\criticaldimension-2}{0}{2} \leq \frac{\upboundcubeintspanconstant{\criticaldimension-2}}{\upboundcubeintspanconstant{\criticaldimension}}\upboundcubeintspan{\criticaldimension} \leq \upboundcubeintspan{\criticaldimension},
    \end{align*} 
    where the last inequality follows as the functions $\upboundcubeintspanconstant{m}$  are clearly increasing in $m$. 
    This implies
    \begin{align*}
         \ffunc{G}{\criticaldimension}{\criticaldimension}{\criticaldimension-2}{0}{2} \leq \binom{n}{\criticaldimension} k^{n-\criticaldimension}\upboundcubeintspan{\criticaldimension},
    \end{align*}
    as claimed. 
\end{proof}

We complete the proof of  \crefitemintheorem{thm:main:threshold}{thm:main:threshold:lower} by showing that indeed the expected number of internally spanned \projections{} of dimension $\criticaldimension$ turns to $0$ as $n$ tends to infinity.  
\begin{proof}[Proof of \Cref{lem:expec_to_zero_smallterm}]
Recalling the definition \eqref{def:upboundcubeintspan} of the function $\upboundcubeintspan{\cdot}$ and using \Cref{lem:upperboundcubeintspanconst}  we obtain
    \begin{align} 
       \binom{n}{\criticaldimension} n^{k-\criticaldimension+\frac{1}{2}} \upboundcubeintspan{\criticaldimension} \overset{\eqref{def:upboundcubeintspan}}&{=}\upboundcubeintspanconstant{\criticaldimension} \binom{n}{\criticaldimension} (\criticaldimension!) k^{n-\criticaldimension+\frac{1}{2}} p^{\frac{\criticaldimension}{2}+1}  2^{-\frac{\criticaldimension}{2}-1}(k-1)^{\criticaldimension} k^{\frac{\criticaldimension^2+2\criticaldimension}{4}} \notag \\
        &\leq   \upboundcubeintspanconstant{\criticaldimension} n^{\criticaldimension} \left(\frac{1}{2n^2}\right)^{\frac{\criticaldimension}{2}+1}k^{n+\frac{1}{4}\left(\criticaldimension^2+2\criticaldimension+2\right)-(2\sqrt{n}-1)(\frac{\criticaldimension}{2}+1)} \quad (\text{because } p \leq p_* =n^{-2}k^{-2\sqrt{n}+1} ) \notag\\
        \overset{Lem.~\ref{lem:upperboundcubeintspanconst}}&{\leq} \Theta(1)\left(\frac{1}{2}\right)^{\frac{\criticaldimension}{2}+1}k^{n+\frac{1}{4}\left(\criticaldimension^2+2\criticaldimension+2\right)-(2\sqrt{n}-1)(\frac{\criticaldimension}{2}+1)}. \label{eq:expec_to_zero_smallterm:1}
    \end{align}
    Now observe that the function $g:\mathbb R \to \mathbb R, x\mapsto  n+\frac{1}{4}\left(x^2+2x+2\right)-(2\sqrt{n}-1)(\frac{x}{2}+1)$ defines a convex parabola whose center is shifted to $2\sqrt{n}-2$ and thus it takes its minimum at $x = 2\sqrt{n}-2$ and is decreasing up to this value. In particular, as $\criticaldimension \coloneqq \rounddown{2\sqrt{n}}-2 \geq 2\sqrt{n}-3$ it holds that $g(\criticaldimension) \leq g(2\sqrt{n}-3) = \frac{3}{4}$.
    Thus, we obtain together with \eqref{eq:expec_to_zero_smallterm:1} that 
    \begin{align*}
        \binom{n}{\criticaldimension} n^{k-\criticaldimension+\frac{1}{2}} \upboundcubeintspan{\criticaldimension} \leq \Th{1} \left(\frac{1}{2}\right)^{\sqrt{n}-\frac{1}{2}} k^{\frac{3}{4}} = o(1),
    \end{align*}
    where the last equality uses that $k \leq 2^{\sqrt{n}}$. 
\end{proof}

\section{Proofs of auxiliary results for the upper threshold}\label{sec:proof:upper}
In this section we prove auxiliary results from \Cref{sec:UpperBound}. Recall $\criticalsizesequential \coloneqq \rounddown{\frac{n}{2}}$. For each $1 \leq \ell  \leq \criticalsizesequential$ we set
\begin{align*}
  \tilde{\seqspanconst{\ell}} \coloneqq \binom{n-2\ell+2}{2}(k-1)^2 k^{2\ell -2}. 
\end{align*}

Recall that $\seqset{\ell}$ denotes the set of all \feasible{} sequences of \seqsize{} $\ell$ in $G=(V,E)$ and we start by computing $|\seqset{\ell}|$.

\begin{proof}[Proof of \Cref{lem:seqspancount}]
Let $1 \leq \ell \leq \criticalsizesequential$ and  let $\seq{v}{\ell-1} =(\vertexvec{v}_0,\ldots,\vertexvec{v}_{\ell-1}) \in \seqset{\ell-1}$ be a \feasible{} sequence of \seqsize{} $\ell-1$. Looking at Property \eqref{Prop:sequentiallyspanning} it follows immediately that
\begin{align*}
    \seqspanconst{\ell} :&=    \left| \setbuilder{\vertexvec{w} \in V}{\seq{v}{\ell} =(\vertexvec{v}_0,\ldots,\vertexvec{v}_{\ell-1},\vertexvec{w}) \in \seqset{\ell} } \right| \\
    &= \left| \setbuilder{\vertexvec{w} \in V}{\dist{\closure{\set{\seq{v}{\ell-1}}}}{\vertexvec{w}} = 2} \right|.
\end{align*}
By \Cref{lem:seqspanningspansprojection}, $ \closure{\set{\seq{v}{\ell-1}}}$ spans a \projection{} $\subcube{}$ of dimension $2 \ell -2$. Thus, it suffices to count the number of vertices that lie at distance two from $\subcube{}$. To that end, note that any vertex $\vertexvec{v}$ satisfying $\dist{\vertexvec{v}}{\subcube{2 \ell -2}} =2$ has to differ from $\subcube{}$ in exactly two coordinates. As there are  $2 \ell -2$ coordinates, in which $\subcube{}$ consists of a complete graph $K_k$, there are $\binom{n-2\ell+2}{2}$ possibilities for choosing the coordinates in which   $\vertexvec{v}$ differs from $\subcube{}$ and $(k-1)^2$ possibilities for specifying a value of $\vertexvec{v}$ in those coordinates. In all coordinates, in which $\subcube{}$ consists of a complete graph, $\vertexvec{v}$ can take an arbitrary value $\in [k]$, leading to  $k^{2\ell-2}$ additional choices. All remaining coordinates of $\vertexvec{v}$ have to coincide with $\subcube{}$, implying
\begin{align}\label{eq:seqspanextension}
    \seqspanconst{\ell} = \binom{n-2\ell+2}{2} (k-1)^2 k^{2\ell-2} =: \tilde{\seqspanconst{\ell}} ,
\end{align} as claimed. 

To show \eqref{eq:Sell} and thus to explicitly compute $ |\seqset{\ell}| $, recall that  $|\seqset{0}| = k^n $ and note that by the definitions of $\seqset{\ell}$ and $\seqspanconst{\ell}$ we have
\begin{align}\label{eq:SellCell}
|\seqset{\ell}| = |\seqset{\ell-1}| \cdot \seqspanconst{\ell}   \quad\quad \text{for all } \quad 1 \leq \ell \leq \criticalsizesequential.
\end{align}
Thus, \eqref{eq:seqspanextension} and \eqref{eq:SellCell} imply
\begin{align*}
    |\seqset{\ell}| &= k^n \prod_{j=1}^\ell \seqspanconst{j} = k^n \prod_{j=1}^\ell \binom{n-2j+2}{2}(k-1)^2k^{2j-2} =   k^n\prod_{j=1}^{\ell} \seqspanconst{j} = \frac{n!}{(n-2\ell)!}  (k-1)^{2\ell} 2^{-\ell} k^{n+\ell^2-\ell},
\end{align*} completing the proof. 

\end{proof}
We continue with the proof of \Cref{lem:bound:seqoverlap} on the upper bound on the number of pairs of \feasible{} sequences of \seqsize{} $\ell$ that intersect exactly in $i$ vertices, which was denoted by
$$\seqoverlap{\ell}{i} \coloneqq \setbuilder{\{\seq{v}{\ell},\seq{w}{\ell} \} \subseteq \seqset{\ell}}{\left|\seq{v}{\ell}\cap \seq{w}{\ell}\right| = i}.$$

In general, our proof strategy is to start with pairs of very short sequences and then sequentially extend them to length $\ell$. To that end, we first take a look at pairs of sequences that do not intersect until their last index. Note that the resulting sequences can then either intersect in one or two vertices, depending on whether only one of the respective final vertices is contained in the other sequence, or both of them are. 
In particular, for $m \in [n]\cup\{0\}$ and $j \in [2]$ let us denote by
   \begin{align}\label{eq:def:seqoverlaplastindex}
      Y_m^{(j)} \coloneqq \setbuilder{ \set{\seq{v}{m} = (\vertexvec{v}_0,\ldots,\vertexvec{v}_m), \seq{w}{m} = (\vertexvec{w}_0,\ldots,\vertexvec{w}_m)} \in \seqoverlap{m}{j}}{\{\vertexvec{v}_0,\ldots,\vertexvec{v}_{m-1}\} \cap \{\vertexvec{w}_0,\ldots,\vertexvec{w}_{m-1}\}  = \emptyset}, 
   \end{align} the set of pairs of \feasible{} sequences of length $m$ with intersection of size $j$ such that the cut-off sequences (i.e., the sequences without the last vertex) of length $m-1$  have empty intersection. 
In the next lemma we bound the sizes of these sets from above.
\begin{lem}\label{lem:bound:seqoverlaplastindex}
Let $m \in [n]\cup\set{0}$ and let $ \seqoverlaplastindexone{m}$ and $\seqoverlaplastindextwo{m}$ be defined as in \eqref{eq:def:seqoverlaplastindex}. Then, 
\begin{enumerate}[label = (\alph*)]

\item\label{eq:lem:bound:seqoverlapindex:firstpoint} $|\seqoverlaplastindexone{m}| \leq \frac{2(m+1)}{k^n}|\seqset{m}|^2$;
\item\label{eq:lem:bound:seqoverlapindex:secondpoint} 
$|\seqoverlaplastindextwo{m}| \leq \indicator{m \geq 1}\frac{m(m+1) }{\seqspanconst{m} k^n} |\seqset{m}|^2 $.

\end{enumerate}
\end{lem}
\begin{proof}
We start by showing \ref{eq:lem:bound:seqoverlapindex:firstpoint}. Due to symmetry we can assume w.lo.g. that all pairs of sequences $\{\seq{v}{m} = (\vertexvec{v}_0,\ldots,\vertexvec{v}_m), \seq{w}{m} = (\vertexvec{w}_0,\ldots,\vertexvec{w}_m)\} \in \seqoverlaplastindexone{m}$ satisfy  $\vertexvec{w}_m \in \set{\seq{v}{m}}$ and $\vertexvec{v}_m \notin \set{\seq{w}{m}}$. 
For an arbitrary vertex $\vertexvec{x} \in V$ let us denote by 
\begin{align*}
    \seqsetwithvertex{m}{x}{i} \coloneqq \setbuilder{\seq{w}{m} =(\vertexvec{w}_0,\ldots\vertexvec{w}_m) \in \seqset{m}}{\vertexvec{w}_m = \vertexvec{x}},
\end{align*} the set of \feasible{} sequences of length $m$  that end with vertex $\vertexvec{x}$.  
Now observe that we can obtain every pair of sequences in $\seqoverlaplastindexone{m}$ by first choosing a sequence $\seq{v}{m} \in \seqset{m}$, then fixing a vertex $\vertexvec{x} \in \set{\seq{v}{m}}$ and finally choosing a sequence $\seq{w}{m} \in \seqsetwithvertex{m}{x}{i}$ that ends with vertex $\vertexvec{x}$. In particular, we obtain
\begin{align}\label{eq:lem:bound:seqoverlaplastindex}
    |\seqoverlaplastindexone{m}| \leq 2 \sum_{\seq{v}{m} \in \seqset{m}} \sum_{\vertexvec{x} \in \set{\seq{v}{m}}} |\seqsetwithvertex{m}{x}{i}|.   
\end{align}
To bound $|\seqsetwithvertex{m}{x}{i}|$ from above, note that
\begin{align*}
     \bigcup_{\vertexvec{x} \in G} \seqsetwithvertex{m}{x}{i} = \seqset{m}.
\end{align*}
Indeed, because any sequence $\seq{v}{m} \in \seqsetwithvertex{m}{x}{i} $ is  a \feasible{} sequence of \seqsize{} $m$, it clearly holds that $  \bigcup_{\vertexvec{x} \in G} \seqsetwithvertex{m}{x}{i} \subseteq \seqset{m}$.
For the other inclusion, observe that  for any sequence $\seq{v}{m}=(\vertexvec{v}_0,\ldots,\vertexvec{v}_m) \in \seqset{m}$ it holds that $\seq{v}{m} \in \seqsetwithvertex{m}{v_m}{i}$.
Now note that for different values of $\vertexvec{x}$, the sets $\seqsetwithvertex{m}{x}{i}$ are disjoint and furthermore, due to symmetry, the value $|\seqsetwithvertex{m}{x}{i}|$ does not depend on $\vertexvec{x}$ , implying together that  for all $\vertexvec{x} \in V$, 
\begin{align*}
   k^n \abs{\seqsetwithvertex{m}{x}{i}}=  \sum_{\vertexvec{x} \in V} |\seqsetwithvertex{m}{x}{i}| = |\seqset{m}|,
\end{align*}
and in particular,
\begin{align}
    \abs{\seqsetwithvertex{m}{x}{i}} = \frac{\abs{\seqset{m}}}{k^n}.\label{eq:univeraleq}
\end{align}
Plugging it into \eqref{eq:lem:bound:seqoverlaplastindex} we obtain
\begin{align*}
    |\seqoverlaplastindexone{m}| \overset{\eqref{eq:lem:bound:seqoverlaplastindex},\eqref{eq:univeraleq}}{\leq} 2 \sum_{\seq{v}{m} \in \seqset{m}} \sum_{\vertexvec{x} \in \set{\seq{v}{m}}}  \frac{|\seqset{m}|}{k^n} 
    =   2  | \seqset{m} |  \cdot  |\set{\seq{v}{m}}| \cdot  \frac{|\seqset{m}|}{k^n} 
    =   \frac{2(m+1)}{k^n}|\seqset{m}|^2,
\end{align*} completing the proof of \ref{eq:lem:bound:seqoverlapindex:firstpoint}. 

To show \ref{eq:lem:bound:seqoverlapindex:secondpoint}, note that we can count $\seqoverlaplastindextwo{m}$ similarly as $\seqoverlaplastindexone{m}$: We fix first a sequence $\seq{v}{m-1}$ and then a vertex $\vertexvec{x}$ of $\seq{v}{m-1}$ with which $\seq{w}{m}$ should end. Next we fix a sequence $\seq{w}{m} \in \seqsetwithvertex{m}{x}{i}$ and finally choose a vertex $\vertexvec{y}$ from all the vertices in $\seq{w}{m}$ to extend $\seq{v}{m-1}$. Independently of all these choice we have $\left|\set{\seq{w}{m}} \right|=m+1$. Thus, from \eqref{eq:univeraleq} and  \eqref{eq:SellCell} with $\ell=m$, i.e., $|\seqset{m-1}| = |\seqset{m}|/C_m$ we obtain
\begin{align*}
    |\seqoverlaplastindextwo{m}| &\leq  \sum_{\seq{v}{m-1} \in \seqset{m-1}} \sum_{\vertexvec{x} \in \set{\seq{v}{m-1}}} \sum_{\seq{w}{m} \in \seqsetwithvertex{m}{x}{i}} \left|\set{\seq{w}{m}} \right| \\
    &=  |\seqset{m-1}| \cdot   |\set{\seq{v}{m-1}}|   \cdot |\seqsetwithvertex{m}{x}{i}| \cdot (m+1) \\
   \overset{\eqref{eq:univeraleq}}&{=} |\seqset{m-1}| \cdot m (m+1) \cdot    \frac{|\seqset{m}|}{k^n}\\
   \overset{\eqref{eq:SellCell}}&{=} \frac{m(m+1) }{\seqspanconst{m} k^n}|\seqset{m}|^2,
\end{align*} as required.
\end{proof}

\Cref{lem:bound:seqoverlaplastindex} essentially tells us how many ways there are to start intersecting pairs of sequences.  
We continue by counting the number of ways we can extend such pairs of sequences to sequences of length $\ell$.  
For $m\in [n]$ and $i \in[m+1]\cup\set{0}$, let $\{\seq{v}{m} = (\vertexvec{v}_0,\ldots,\vertexvec{v}_m), \seq{w}{m} = (\vertexvec{w}_0,\ldots,\vertexvec{w}_m)\} \subseteq \seqset{m}$ be two \feasible{} sequences satisfying $|\seq{v}{m} \cap \seq{w}{m} | = i$. For $\ell \geq m+1$ and $s \in [\ell]\cup \set{0} $ let us denote by 
\begin{align}\label{eq:def:extension}
    &{\extset{\seq{v}{m}}{\seq{w}{m}}{\ell}{i}{s}} \coloneqq \notag\\
    &\setbuilder{\set{\seq{v}{\ell}' = \left(\vertexvec{v}_0',\ldots,\vertexvec{v}_\ell'\right), \seq{w}{\ell}' = \left(\vertexvec{w}_0',\ldots,\vertexvec{w}_\ell'\right)} \in \seqoverlap{\ell}{s+i}}{\forall 0 \leq j \leq m \colon \left( \vertexvec{v}_j = \vertexvec{v}_j' \land \vertexvec{w}_j = \vertexvec{w}_j'\right)} 
\end{align} the set of all extensions of the pair $\{\seq{v}{m},\seq{w}{m}\} \subseteq \seqset{m}$ to a pair $\{\seq{v}{\ell}', \seq{w}{\ell}'\} \subseteq \seqset{\ell}$ of \feasible{} sequences of  \seqsize{} $\ell$ satisfying $\left|\seq{v}{\ell}' \cap \seq{w}{\ell}' \right| = i+s$. 
As $\ell$  remains fixed throughout this section, we will drop the subscript $\ell$ to ease notation and write
\begin{align*}
    \extnum{m}{i}{s} \coloneqq \max_{\set{\seq{v}{m},\seq{w}{m}} \in \seqoverlap{m}{i}} \abs{\extset{\seq{v}{m}}{\seq{w}{m}}{\ell}{i}{s}}.
\end{align*}

In the next lemma we provide bounds on $\extnum{m}{i}{s}$.

\begin{lem}\label{lem:bound:seqoverlap:extensions}
Let $m\in [n]$, $i \in[m+1]\cup\set{0}$ and $s \in [\ell]\cup \set{0} $. For $1 \leq j \leq  \criticalsizesequential$ set  

\begin{align*}
    \constinlemma{j} \coloneqq \indicator{j=1}\seqspanconst{1} + \indicator{j>1}\maxp{}{\seqspanconst{j},3\constinlemma{j-1}}.
\end{align*}
Then, 
\begin{enumerate}[label = (\alph*)]
    \item $\seqspanconst{j} \leq \constinlemma{j}$;\label{eq:seqspanconst_smaller_constinlemma}
    \item$ \extnum{m}{i}{s} \leq 8^s \left(\prod_{j=m+1}^{\ell}\constinlemma{j}^2\right) \auxfunc{m}{i}{s}$, \label{eq:bound:seqoverlap:auxiliary}
\end{enumerate}
where 
\begin{align}\label{eq:def:auxfunc}
     \auxfunc{m}{i}{s} &\coloneqq 
     \begin{cases}
         \indicator{s=0}+\indicator{s>0}\prod_{j = m+1}^{2m-i+1} \left(\frac{2m-i+2-j}{\constinlemma{j}} \right)^2\prod_{j = 2m-i+2}^{s+i-1} \constinlemma{j}^{-1} & \text{ if }\ m-i+1 < s/2 \\
        \indicator{s=0}+\indicator{s>0}\prod_{j = m+1}^{m+\rounddown{\frac{s}{2}}} \left(\frac{2m-i+2-j}{\constinlemma{j}} \right)^2 \left(\indicator{s \text{ even }}+ \indicator{s \text{ odd }} \frac{2(m-\rounddown{\frac{s}{2}}-i+1)+1}{\constinlemma{m +\rounddown{\frac{s}{2}}+1}}\right) & \text{ else.}
     \end{cases}
\end{align}
\end{lem}
Let us first remark that the functions $\constinlemma{j}$ are only introduced to ensure that the ratio $\frac{\constinlemma{j}}{\constinlemma{j-1}}$ is at least $3$ for all values of $1 \leq j \leq  \criticalsizesequential$ (which does not hold for the functions $\seqspanconst{j}$).
To give some further intuition on the bounds stated in \Cref{lem:bound:seqoverlap:extensions}, observe that the term $\prod_{j=m+1}^{\ell}\constinlemma{j}^2$ is an upper bound on the number of extensions of a pair $\{\seq{v}{m},\seq{w}{m}\}$ of sequences of length $m$ to a pair $\{\seq{v}{\ell}', \seq{w}{\ell}'\}$ of sequences of lenght $\ell$, in which the extended sequences do not have to overlap any further. In a sense, the function $\auxfunc{\cdot}{\cdot}{\cdot}$ quantifies how much having overlap restricts a pair of extensions. In particular, in order to extend a sequence $\seq{v}{m}$ by a vertex that is also appearing in $\seq{w}{m}$, the number of choices for this vertex are reduced from $\seqspanconst{m+1}$ to the number of vertices appearing in $\seq{w}{m}$. 

In order to prove \Cref{lem:bound:seqoverlap:extensions}, we will first show some general properties of the  function $\auxfunc{\cdot}{\cdot}{\cdot}$. 
\begin{lem}\label{lem:bound:auxfunc}
    For any  $(i,s) \in \N^2$ the following hold. 
    \begin{enumerate}[label = (\alph*)]
        \item\label{item:lem:bound:auxfunc:a} The function $\auxfunc{\cdot}{i}{s}$, as defined in \eqref{eq:def:auxfunc}, is decreasing in the first coordinate.
        \item\label{item:lem:bound:auxfunc:b} For all $m \geq i-1 $ and $s\geq 1$ 
        \begin{align*}
        \sum_{j= m+1}^{\ell}   \left(\frac{2(j+1-i) \auxfunc{j}{i+1}{s-1}}{\constinlemma{j}} + \frac{(j-i)^2  \auxfunc{j}{i+2}{s-2}}{8{\constinlemma{j}}^2}\indicator{s \geq 2}  \right) \leq 8\auxfunc{m}{i}{s}.
    \end{align*}
    \end{enumerate}
\end{lem}
The proof of \Cref{lem:bound:auxfunc} is  simple but tedious and is therefore deferred to the \Cref{sec:appendix_bounds}. We proceed by showing \Cref{lem:bound:seqoverlap:extensions}.

\begin{proof}[Proof of \Cref{lem:bound:seqoverlap:extensions}]
Note first that \crefitemintheorem{lem:bound:seqoverlap:extensions}{eq:seqspanconst_smaller_constinlemma} is a trivial observation.

Continuing with \crefitemintheorem{lem:bound:seqoverlap:extensions}{eq:bound:seqoverlap:auxiliary} we will prove the statement using induction on $s$, starting with the base case $s = 0$. By the definition of $$\seqspanconst{j+1}:=    \left| \setbuilder{\vertexvec{w} \in V}{\seq{v}{j+1} =(\vertexvec{v}_0,\ldots,\vertexvec{v}_{j},\vertexvec{w}) \in \seqset{j+1} } \right|,$$ there are at most $\seqspanconst{j+1}$ ways to extend a \feasible{} sequence of \seqsize{} $j$. Thus, sequentially extending both sequences (and ignoring any potential overlaps) yields together with \ref{eq:seqspanconst_smaller_constinlemma} the upper bound  
    \begin{align*}
     \extnum{m}{i}{0} \leq  \prod_{j = m+1}^{\ell} \seqspanconst{j}^2 \leq   \prod_{j = m+1}^{\ell} \constinlemma{j}^2 = \left(\prod_{j = m+1}^{\ell} \constinlemma{j}^2 \right) \auxfunc{m}{i}{0},
    \end{align*} proving the induction base. 
  
We proceed with the induction hypothesis. Let $\set{\seq{v}{m}, \seq{w}{m}} \in \seqoverlap{m}{i}$, let $s \in [\ell]$ and assume that \crefitemintheorem{lem:bound:seqoverlap:extensions}{eq:bound:seqoverlap:auxiliary} holds for all $j \in [s-1]$.
   For a fixed index $m+1 \leq \specialindex \leq \ell$ and $\tau \in \{1,2\}$ let us consider the number $\extnumindex{m}{i}{s}{\specialindex}{\tau}$ of extensions $\{\seq{v}{\ell} = (\vertexvec{v}_0,\ldots,\vertexvec{v}_\ell), \seq{w}{\ell} = (\vertexvec{w}_0,\ldots,\vertexvec{w}_\ell)\} \subseteq \seqset{\ell}$ satisfying $|\seq{v}{\specialindex-1} \cap \seq{w}{\specialindex-1} | = i$ and $|\seq{v}{\specialindex} \cap \seq{w}{\specialindex} | = i+\tau$, i.e., $\specialindex$ is the first index larger than $m$ at which the extended sequences contain an additional vertex that is also contained in the other sequence up to this point. 
   
   We will compute  $\extnumindex{m}{i}{s}{\specialindex}{\tau}$ in three steps: First, we will compute the number of possible extensions of $\seq{v}{m}$ and $\seq{w}{m}$ to sequences $\seq{v}{\specialindex-1}$ and $\seq{w}{\specialindex-1}$, then count the number of ways to extend those to sequences $\seq{v}{\specialindex}$ and $\seq{w}{\specialindex}$ and finally extend those to sequences $\seq{v}{\ell}$ and $\seq{w}{\ell}$. 
   Analogously to the induction base, there  are at most $\prod_{j = m+1}^{\specialindex-1} \seqspanconst{j}^2$ ways to extend the sequences $\seq{v}{m}$ and $\seq{w}{m}$ to sequences $\seq{v}{\specialindex-1}$ and $\seq{w}{\specialindex-1}$. To compute the number of ways how to  further extend those sequences, we distinguish two cases. 
   \begin{itemize}
       \item[\textbf{Case 1:}] $|\seq{v}{\specialindex} \cap \seq{w}{\specialindex} | = i+1$. \\
       By definition of $\specialindex$, one of the following has to hold in this case.  
       \begin{enumerate}
        \item $\vertexvec{v}_{\specialindex} \in \{\seq{w}{\specialindex-1}\}$ and $ \vertexvec{w}_{\specialindex} \notin \{\seq{v}{\specialindex}\}$; 
        \item $ \vertexvec{w}_{\specialindex} \in \{\seq{v}{\specialindex-1}\}$ and $\vertexvec{v}_{\specialindex} \notin \{\seq{w}{\specialindex}\}$;
        \item $\vertexvec{v}_{\specialindex} = \vertexvec{w}_{\specialindex}$.
    \end{enumerate}
    In the first subcase, there are at most $\seqspanconst{\specialindex}$ ways for picking  $ \vertexvec{w}_{\specialindex}$. Furthermore, $\vertexvec{v}_{\specialindex}$ has to be one of the $\specialindex-i$ vertices contained in $\set{\seq{w}{\specialindex-1}}\setminus \set{\seq{v}{\specialindex-1}}$. Thus, the number of extensions in this subcase can be bounded from above by  $\seqspanconst{\specialindex} (\specialindex-i)$. 
    The second subcase is symmetric to the first, so let us continue with the last subcase. 
    Clearly, if $\vertexvec{v}_{\specialindex} = \vertexvec{w}_{\specialindex}$ then it follows  that there are at most $\seqspanconst{\specialindex}$ ways to extend the sequences. 
    In total, we obtain that the number of ways to extend $\seq{v}{\specialindex-1}$ and $\seq{w}{\specialindex-1}$ to $\seq{v}{\specialindex}$ and $\seq{w}{\specialindex}$ is bounded from above by $\seqspanconst{\specialindex}(2(\specialindex-i)+1) \leq 2\seqspanconst{\specialindex}(\specialindex-i+1)$. 
    \item[\textbf{Case 2:}] $|\{\seq{v}{\specialindex}\} \cap \{\seq{w}{\specialindex}\} | = i+2$. \\
    In this case it follows that $\vertexvec{v}_{\specialindex} \in \{\seq{w}{\specialindex-1}\}$ and $\vertexvec{w}_{\specialindex} \in \{\seq{v}{\specialindex-1}\}$. Similar to the previous case, $\vertexvec{v}_{\specialindex}$ has to be one of the $\specialindex-i$ vertices contained in $\set{\seq{w}{\specialindex-1}}\setminus \set{\seq{v}{\specialindex-1}}$  and analogously $\vertexvec{w}_{\specialindex}$ has to be one of the $\specialindex-i$ vertices contained in $\set{\seq{v}{\specialindex-1}}\setminus \set{\seq{w}{\specialindex-1}}$. Thus, there are at most $(\specialindex-i)^2$ possibilities to extend. 
   \end{itemize}
   It remains to bound the number of ways to extend the sequences $\seq{v}{\specialindex}$ and $\seq{w}{\specialindex}$ to sequences  $\seq{v}{\ell}$ and $\seq{w}{\ell}$. In the first case, this number is clearly given as  $\extnum{\specialindex}{i+1}{s-1}$ and in the second case as  $\extnum{\specialindex}{i+2}{s-2}$. Putting everything together and using the induction hypothesis \crefitemintheorem{lem:bound:seqoverlap:extensions}{eq:bound:seqoverlap:auxiliary} we obtain
   \begin{align}
       \extnumindex{m}{i}{s}{\specialindex}{1} &\leq \left( \prod_{j = m+1}^{\specialindex-1} \seqspanconst{j}^2\right)  2\seqspanconst{\specialindex}(\specialindex-i+1)\extnum{\specialindex}{i+1}{s-1} \notag\\
       \overset{Lem.~\ref{lem:bound:seqoverlap:extensions} (b)}&{\leq} \left( \prod_{j = m+1}^{\specialindex-1} \seqspanconst{j}^2\right)  2\seqspanconst{\specialindex}(\specialindex-i+1)8^{s-1} \left( \prod_{j = \specialindex+1}^{\ell} \constinlemma{j}^2\right)\auxfunc{\specialindex}{i+1}{s-1}. \label{eq:prf:extlem:1}
   \end{align}
   Using  \crefitemintheorem{lem:bound:seqoverlap:extensions}{eq:seqspanconst_smaller_constinlemma} in \eqref{eq:prf:extlem:1} we thus obtain
   \begin{align}
       \extnumindex{m}{i}{s}{\specialindex}{1} \overset{Lem.~\ref{lem:bound:seqoverlap:extensions} (a)}&{\leq}\left( \prod_{j = m+1}^{\specialindex-1} \constinlemma{j}^2\right)  2\constinlemma{\specialindex}(\specialindex-i+1)8^{s-1} \left( \prod_{j = \specialindex+1}^{\ell} \constinlemma{j}^2\right)\auxfunc{\specialindex}{i+1}{s-1}  \notag\\
       &= 8^{s-1}  \left( \prod_{j = m+1}^{\ell} \constinlemma{j}^2\right) \frac{2(\specialindex-i+1)\auxfunc{\specialindex}{i+1}{s-1}}{\constinlemma{\specialindex}}.\label{eq:prf:extlem:2}
   \end{align} 
   Analogously to \eqref{eq:prf:extlem:1} and \eqref{eq:prf:extlem:2} we obtain for $\extnumindex{m}{i}{s}{\specialindex}{2}$ that 
   \begin{align}\label{eq:prf:extlem:3}
       \extnumindex{m}{i}{s}{\specialindex}{2} \leq 8^{s-2}  \left( \prod_{j = m+1}^{\ell} \constinlemma{j}^2\right) \frac{(\specialindex-i ) \auxfunc{\specialindex}{i+2}{s-2}}{\constinlemma{\specialindex}^2}.
   \end{align}
   Summing over all possible values of $\specialindex$, it follows from  \eqref{eq:prf:extlem:2} and \eqref{eq:prf:extlem:3} that

\begin{align}
       & \extnum{m}{i}{s} \leq \sum_{\specialindex = m+1}^{\ell} \extnumindex{m}{i}{s}{\specialindex}{1}+ \extnumindex{m}{i}{s}{\specialindex}{2} \notag\\ 
       &\leq 8^{s-1} \left(\prod_{j=m+1}^{\ell} \constinlemma{j}^2\right)  \sum_{\specialindex= m+1}^{\ell}   \left(\frac{2(\specialindex+1-i) \auxfunc{\specialindex}{i+1}{s-1}}{\constinlemma{\specialindex}} + \frac{(\specialindex-i)^2  \auxfunc{\specialindex}{i+2}{s-2}}{8{\constinlemma{\specialindex}}^2}  \right).  \label{eq:prf:extlem:4}
   \end{align}
   Finally, plugging the bounds from \crefitemintheorem{lem:bound:auxfunc}{item:lem:bound:auxfunc:b} into \eqref{eq:prf:extlem:4} we obtain
   \begin{align*}
        \extnum{m}{i}{s} \leq 8^s  \left(\prod_{j=m+1}^{\ell} \constinlemma{j}^2\right) \auxfunc{m}{i}{s}, 
   \end{align*} completing the proof. 
\end{proof}

We are now ready to proof \Cref{lem:bound:seqoverlap}. 
\begin{proof}[Proof of \Cref{lem:bound:seqoverlap}]
Recall $\criticalsizesequential:=\rounddown{\frac{n}{2}}$ and let $\ell \in [\criticalsizesequential]\cup\{0\}$ and $i \in [\ell+1]$. Splitting the set $\seqoverlap{\ell}{i} \coloneqq \setbuilder{\{\seq{v}{\ell},\seq{w}{\ell} \} \subseteq \seqset{\ell}}{\left|\seq{v}{\ell}\cap \seq{w}{\ell}\right| = i}$ into parts, depending on the first index $m$ at which a pair $\set{\seq{v}{\ell},\seq{w}{\ell}}$ of \feasible{} sequences has non-empty overlap, we obtain
   \begin{align*}
       |\seqoverlap{\ell}{i}| &\leq \sum_{m = 0}^{\ell} \sum_{\{\seq{v}{m},\seq{w}{m}\} \in \seqoverlaplastindexone{m}} \extnum{m}{1}{i-1} + \sum_{m = 1}^{\ell} \sum_{\{\seq{v}{m},\seq{w}{m}\} \in \seqoverlaplastindextwo{m}} \extnum{m}{2}{i-2} \\
       &\leq \sum_{m = 0}^{\ell}\abs{\seqoverlaplastindexone{m}} \extnum{m}{1}{i-1} + \sum_{m = 1}^{\ell} \abs{\seqoverlaplastindextwo{m}} \extnum{m}{2}{i-2} .
   \end{align*}
Together with the bounds on  from \Cref{lem:bound:seqoverlaplastindex} and  \crefitemintheorem{lem:bound:seqoverlap:extensions}{eq:bound:seqoverlap:auxiliary} it follows that
   \begin{align}\label{eq:prf:boundseqoverlap:1}
       |\seqoverlap{\ell}{i}| \overset{Lem.~\ref{lem:bound:seqoverlaplastindex}}&{\leq} \sum_{m = 0}^{\ell}  \frac{2(m+1)|\seqset{m}|^2}{k^n} \extnum{m}{1}{i-1} + \sum_{m = 1}^{\ell}  \frac{m(m+1) |\seqset{m}|^2}{\seqspanconst{m} k^n} \extnum{m}{2}{i-2}. \notag\\
      \overset{Lem.~\ref{lem:bound:seqoverlap:extensions} (b)}&{\leq} \sum_{m = 0}^{\ell}  \frac{2(m+1)|\seqset{m}|^2}{k^n} 8^{i-1} \left(\prod_{j=m+1}^\ell \constinlemma{j}^2 \right) \auxfunc{m}{1}{i-1} \notag \\ &+\sum_{m = 1}^{\ell}  \frac{m(m+1) |\seqset{m}|^2}{\seqspanconst{m} k^n} 8^{i-2} \left(\prod_{j=m+1}^\ell \constinlemma{j}^2 \right) \auxfunc{m}{2}{i-2}.
    \end{align}  
    Now recall that by \Cref{lem:seqspancount} it holds that $\abs{\seqset{m}} = k^n \prod_{j =1}^m \seqspanconst{j}$. In particular, it follows from \crefitemintheorem{lem:bound:seqoverlap:extensions}{eq:seqspanconst_smaller_constinlemma} that 
    \begin{align}\label{eq:prf:boundseqoverlap:2}
        \frac{\abs{\seqset{m}}^2}{k^n} \leq k^n \left(\prod_{j=1}^{m} \constinlemma{j}^2\right)  \quad \text{and} \quad \frac{\abs{\seqset{m}}^2}{k^n\seqspanconst{m}} 
        \leq \frac{k^n \left(\prod_{j=1}^{m} \constinlemma{j}^2\right) }{\constinlemma{m}}.
    \end{align}
    Plugging \eqref{eq:prf:boundseqoverlap:2} into \eqref{eq:prf:boundseqoverlap:1} and extracting common factors from both sums we get 
    \begin{align}\label{eq:prf:boundseqoverlap:3}
        |\seqoverlap{\ell}{i}| &\leq 8^{i-1} k^n\left(\prod_{j=1}^\ell \constinlemma{j}^2 \right)  \left(\sum_{m = 0}^{\ell} 2(m+1) \auxfunc{m}{1}{i-1} + \sum_{m = 1}^{\ell} \frac{m(m+1)}{8 \constinlemma{m}} \auxfunc{m}{2}{i-2}\right) \notag \\
        &\leq 8^{i-1} k^n\left(\prod_{j=1}^\ell \constinlemma{j}^2 \right) \left(\ell+1\right)^2 \left(\sum_{m = 0}^{\ell} 2\auxfunc{m}{1}{i-1} + \sum_{m = 1}^{\ell} \frac{\auxfunc{m}{2}{i-2}}{8 \constinlemma{m}} \right).
    \end{align}
    By \crefitemintheorem{lem:bound:auxfunc}{item:lem:bound:auxfunc:a}, the function $\auxfunc{\cdot}{\cdot}{\cdot}$ is decreasing in the first coordinate and furthermore, the function $m\mapsto 1/\constinlemma{m}$ is decreasing in $m\in [\ell]$, and thus we have
    \begin{align}
        &\left(\sum_{m = 0}^{\ell} 2\auxfunc{m}{1}{i-1} + \sum_{m = 1}^{\ell} \frac{\auxfunc{m}{2}{i-2}}{8 \constinlemma{m}} \right) \leq \left(\sum_{m = 0}^{\ell} 2\auxfunc{0}{1}{i-1} + \sum_{m = 1}^{\ell} \frac{\auxfunc{1}{2}{i-2}}{8 \constinlemma{1}} \right) \notag\\
        &\leq 2(\ell+1) \left( \auxfunc{0}{1}{i-1} + \frac{\auxfunc{1}{2}{i-2}}{ \constinlemma{1}}\right)= 2(\ell+1) \left( \prod_{j=1}^{i-1} \constinlemma{j}^{-1} +\constinlemma{1}^{-1} \prod_{j=2}^{i-1} \constinlemma{j}^{-1} \right) \notag \\
        &= 4(\ell+1) \left( \prod_{j=1}^{i-1} \constinlemma{j}^{-1} \right),\label{eq:prf:boundseqoverlap:4}
    \end{align} where the penultimate equality is due to the definition  \eqref{eq:def:auxfunc} of $\auxfunc{\cdot}{\cdot}{\cdot}$.
    Combining \eqref{eq:prf:boundseqoverlap:3} and \eqref{eq:prf:boundseqoverlap:4} we obtain
       \begin{align}
|\seqoverlap{\ell}{i}| \overset{\eqref{eq:prf:boundseqoverlap:3}}&{\leq}  8^{i-1} k^n\left(\prod_{j=1}^\ell \constinlemma{j}^2 \right) \left(\ell+1\right)^2 \left(\sum_{m = 0}^{\ell} \auxfunc{m}{1}{i-1} + \sum_{m = 1}^{\ell} \frac{\auxfunc{m}{2}{i-2}}{8 \constinlemma{m}} \right) \notag\\
\overset{\eqref{eq:prf:boundseqoverlap:4}}&{\leq} 8^{i-1} k^n\left(\prod_{j=1}^\ell \constinlemma{j}^2 \right) \left(\ell+1\right)^2 4(\ell+1) \left( \prod_{j=1}^{i-1} \constinlemma{j}^{-1} \right) \leq  \frac{8^i k^n (\ell+1)^3 \left(\prod_{j=1}^\ell \constinlemma{j}^2 \right)}{\left(\prod_{j=1}^{i-1} \constinlemma{j} \right)}. \label{eq:prf:boundseqoverlap:5}
   \end{align}
   To complete the proof, we need to replace the auxiliary functions $\constinlemma{j}$ with the functions $\seqspanconst{j}$. To that end, observe that by \Cref{lem:seqspancount} it holds for all $1 \leq j \leq \criticalsizesequential-7$ that
   \begin{align*}
       \frac{\seqspanconst{j+1}}{\seqspanconst{j}} = \frac{\binom{n-2j}{2}k^2}{\binom{n-2j+2}{2}} \geq 3,
   \end{align*} and thus in particular that 
   $\seqspanconst{j} = \constinlemma{j}$ for $1 \leq j \leq \criticalsizesequential-6$. 
   By the construction of $\constinlemma{j}$ it now follows that 
   \begin{align*}
     \constinlemma{\criticalsizesequential-5} \leq 3 \constinlemma{\criticalsizesequential-6} = 3 \seqspanconst{\criticalsizesequential-6} \leq 3 \seqspanconst{\criticalsizesequential-5}.
   \end{align*} Repeating this step we obtain for $1 \leq j \leq 6$ that 
   \begin{align*}
       \constinlemma{\criticalsizesequential-6+j} \leq 3^j \seqspanconst{\criticalsizesequential-6+j},
   \end{align*} and thus
   \begin{align*}
       \left(\prod_{j=1}^{\criticalsizesequential} \constinlemma{j}^2 \right) =  \left(\prod_{j=1}^{\criticalsizesequential-6} \seqspanconst{j}^2 \right) \left(\prod_{j=1}^{6} \constinlemma{\criticalsizesequential-6+j}^2 \right)\leq  \left(\prod_{j=1}^{\criticalsizesequential} \seqspanconst{j}^2 \right) \left(\prod_{j=1}^{6} 3^{2j} \right) = 3^{42} \left(\prod_{j=1}^{\criticalsizesequential} \seqspanconst{j}^2 \right).
   \end{align*}
Clearly, this extends to arbitrary values of $1 \leq \ell \leq \criticalsizesequential$, i.e., 
   \begin{align}
       \left(\prod_{j=1}^\ell \constinlemma{j}^2 \right) \leq 3^{42} \left(\prod_{j=1}^\ell \seqspanconst{j}^2 \right).\label{eq:Dj}
   \end{align} 
  Combining \eqref{eq:prf:boundseqoverlap:5} and \eqref{eq:Dj}  we obtain
   \begin{align*}
       |\seqoverlap{\ell}{i}| \overset{\eqref{eq:prf:boundseqoverlap:5}}{\leq} \frac{8^ik^n(\ell+1)^3\left(\prod_{j=1}^\ell \constinlemma{j}^2 \right)}{\left(\prod_{j=1}^{i-1} \constinlemma{j} \right)} \overset{\eqref{eq:Dj}}{\leq} \Theta(1) \frac{8^i k^n (\ell+1)^3\left(\prod_{j=1}^\ell \seqspanconst{j}^2 \right)}{\left(\prod_{j=1}^{i-1} \seqspanconst{j} \right)} = \Theta(1) \frac{8^i(\ell+1)^3|\seqset{\ell}|^2}{|\seqset{i-1}|},
   \end{align*} where the last equality uses \eqref{eq:Sell}. 
\end{proof}

We continue by showing that the expectation of $\seqspannum{i}$ tends to infinity for all $0 \leq i \leq \criticalsizesequential. $
\begin{proof}[Proof of \Cref{lem:expec_to_inf}]
  Let $i \in [\criticalsizesequential]\cup\{0\}$. We start by bounding $\expec{\seqspannum{i}}$ from below. By \eqref{eq:expec(seqspannum)} and \Cref{lem:seqspancount} it holds that 
    \begin{align}
        \expec{\seqspannum{i}} \overset{\eqref{eq:expec(seqspannum)}}{=}  |\seqset{i}| p^{i+1} \overset{\eqref{eq:Sell}}{=} \frac{n!}{(n-2i)!}  (k-1)^{2i} 2^{-i} k^{n+i^2-i} p^{i+1} \notag \\
        \geq \frac{n!}{(n-2i)!}  (k-1)^{2i} 2^{-i}  k^{n+i^2-i} \left(200 n^{-2} k^{-2\sqrt{n}+1}\right)^{i+1},\label{eq:prf:expec_to_inf:1}
    \end{align} where the last inequality uses $p \geq p^* := 200 n^{-2} k^{-2\sqrt{n}+1}$. 
    To further bound this expression we  will use the following estimates: 
    \begin{align*}
    \frac{n!}{(n-2i)!} &= n^{2i} \prod_{j=1}^i\left(1-\frac{j}{n}\right) \geq n^{2i} \exp\left(-\frac{2i^2}{n}\right) \geq n^{2i} e^{-i};\\
    (k-1)^{2i} &= k^{2i}\left(1-\frac{1}{k}\right)^{2i} \geq k^{2i} 2^{-2i}.
    \end{align*} 
    Plugging these into \eqref{eq:prf:expec_to_inf:1} and simplifying we obtain
    \begin{align*}
        \expec{\seqspannum{i}} \geq  n^{2i} e^{-i} k^{2i} 2^{-2i}  2^{-i} k^{n+i^2-i}  \left(200 n^{-2} k^{-2\sqrt{n}+1}\right)^{i+1} = 200n^{-2} \left( \frac{200}{8e}\right) k^{i^2-i(2 \sqrt{n}-2)+n-2\sqrt{n}+1}.
    \end{align*}
    It thus follows that 
     \begin{align}
     \frac{\expec{\seqspannum{i}}}{\Th{n^4}8^{i+1}}  &\geq \Theta(n^{-6}) \left(\frac{200}{64e}\right)^i k^{i^2-i(2 \sqrt{n}-2)+n-2\sqrt{n}+1} \notag \\ 
     &\geq \Th{n^{-6}}\left(\frac{10}{9}\right)^i k^{i^2-i( 2\sqrt{n}-2)+n-2\sqrt{n}+1}.\label{eq:Xiexpectation}
    \end{align}
   To derive a lower bound on \eqref{eq:Xiexpectation} we focus on the   the function $\hat g:\mathbb R^2 \to \mathbb R, x\mapsto x^2-x( 2\sqrt{n}-2)+n-2\sqrt{n}+1$ in the exponent. Observing that the function $\hat g$ defines a convex parabola centered at $\sqrt{n}- 1$, it takes its minimum value of $0$ at the minimiser $x= \sqrt{n}- 1$. 
   Thus, we obtain that for all $i \in [\criticalsizesequential]\cup\{0\}$ 
    \begin{align*}
        \frac{\expec{\seqspannum{i}}}{\Theta(n^4)8^{i+1}} \geq \Th{n^{-6}} \left(\frac{10}{9}\right)^{\sqrt{n}-1} = \smallomega{1},
    \end{align*} completing the proof.
\end{proof}

Finally we show that in the case when $n$ is odd, say $n =2m+1$ for  $m\in \mathbb N_0$, we can always find an infected vertex to extend an infected \feasible{} sequence of \seqsize{} $m$.
\begin{proof}[Proof of \Cref{lem:odd_case_distinction}]
    Let $n = 2m+1$ for $m \in \mathbb N_0$. For an arbitrary projection $\subcube{}$ of dimension $2m$, the number of vertices in $G$  that are not contained in $\subcube{}$ is clearly given as
    $(k-1)k^{2m}$. Thus, using $p\geq p^*:= 200 n^{-2} k^{-2\sqrt{n}+1}$, the probability that none of these vertices were initially infected, can be bounded from above by 
    \begin{align*}
        (1-p)^{(k-1)k^{2m}} \leq \exp \left( -pk^{2m} \right) \leq \exp \left( -k^{\frac{n}{2}}\right).
    \end{align*}
Now observe that there are $nk$ \projections{} of dimension $2m$ in $G$. Indeed, there are $n$ ways for choosing the single coordinate in which a \projection{} of dimension $2m$ is not complete and $k$ further possibilities for choosing a vertex of the base graph $K_k$ in this coordinate. Thus it follows by a union bound that the probability that there exists a \projection{} $\subcube{}$ of dimension $2m$ such that no vertex in $G \setminus \subcube{}$ was initially infected can be bounded from above by 
    \begin{align*}
        nk \exp \left( -k^{\frac{n}{2}}\right) = o(1).
    \end{align*}
\end{proof}

\section{Concluding Discussion}\label{sec:Discussion}

In \Cref{thm:main:threshold} we determined the asymptotic value of the critical probability up to  multiplicative constants:
\begin{align}
p_c\left(\square_{i=1}^n K_k, 2\right) = \Th{ n^{-2}k^{-2\sqrt{n}+1}}.
\label{eq:KMS}
\end{align}
The corresponding result for the case $k=2$ was proved in the earlier work of Balogh and Bollob\'{a}s\cite{BaBo2006} which was improved by Balogh, Bollob\'{a}s and Morris\cite{BaBoMo2010}: 
\begin{align}
 \frac{16\lambda}{n^2}\parens{1+\frac{\log n}{\sqrt{n}}}2^{-2\sqrt{n}} \leq p_c(\square_{i=1}^n K_2,2) \leq \frac{16\lambda}{n^2}\parens{1+\frac{ 5 \log^2 n}{\sqrt{n}}}2^{-2\sqrt{n}},\label{eq:BBMsharp}
\end{align} where $\lambda \approx 1.166$.
In order to improve our result, e.g., to obtain a sharp threshold, proof techniques in this paper might be useful. One of the key ingredients in our proof is the bound on the probability of a \projection{} of dimension $2\ell$ being {\em internally spanned} (see \Cref{lem:probintspanupperbound}): This bound is off precisely by factor $\lambda^\ell$, in comparison to the corresponding bound to obtain \eqref{eq:BBMsharp}. Getting more precise asymptotics on the probability of a \projection{} being internally spanned is the most crucial step to obtain a sharp threshold. 
\begin{problem}
    Determine a sharp threshold for $p_c\left(\square_{i=1}^n K_k, 2\right)$ for any $k \geq 3$. 
\end{problem}
The Hamming  graph is a very natural example of a Cartesian product graph and has quite useful properties, e.g., it is regular and vertex-transitive. We believe the proof techniques in this paper are also applicable to similar types of graphs and especially to Cartesian product graphs with different kinds of base graphs. An interesting problem in this context is to determine a set of minimal properties the base graphs have to satisfy in order to obtain a threshold in the flavour of \Cref{thm:main:threshold}.
\begin{problem}
    Determine $p_c\left(G, 2\right)$ for other types of Cartesian product graphs $G$. 
\end{problem}
Lastly, we note that throughout this paper we only considered the infection parameter $r=2$. For $r\geq 3$,  the problem of determining the critical probability  $p_c\left(\square_{i=1}^n K_k, r\right)$ is widely open even for the case $k=2$ and $r=3$. 
\begin{problem}
Determine $p_c\left(\square_{i=1}^n K_k, 3\right)$ for arbitrary $k\geq 2$.
\end{problem}
For the special case when $r= \roundup{\frac{n}{2}}$ and $k=2$, i.e., in the so-called  {\em majority bootstrap percolation} on the $n$-dimensional hypercube  $Q^n=[2]^n=\square_{i=1}^n K_2$,  a sharp threshold was established by Balogh, Bollob\'{a}s and Morris in \cite{BaBoMo2009majority} and their results have been extended to Cartesian products of bounded size regular graphs by Collares, Erde, Geisler and Kang \cite{CoErGeKa2024}. For arbitrary $r \geq 3$, Morrison and Noel answered  the extremal question of determining a smallest percolating set in $Q^n$  in \cite{MoNo2018}; regarding the critical probability,  Balogh, Bollob\'{a}s and Morris \cite[Conjecture 6.3]{BaBoMo2010} conjectured the following in 2010, but there is still no progress:
 \begin{align*}
    p_c\left(\square_{i=1}^n K_2,r\right) = \expp{-\Th{n^{1/2^{r-1}}}}.
\end{align*}   
The sharp threshold \eqref{eq:BBMsharp} suggests that the base of the exponential function in the above conjecture would be two and thus the conjecture could be reformulated as
 \begin{align*}
  -\log_2  p_c\left(\square_{i=1}^n K_2,r\right) =  \Th{n^{1/2^{r-1}}}.
\end{align*}  
In view of our threshold result \eqref{eq:KMS} we make an analogous conjecture:
\begin{conj} 
For arbitrary $r \geq 3$ and $k\geq 2$, 
 \begin{align*}
-\log_k  p_c\left(\square_{i=1}^n K_k,r\right) = \Th{n^{1/2^{r-1}}}.
\end{align*}   
\end{conj}

\section*{Acknowledgements}
The authors are supported by the Austrian Science Fund (FWF) [10.55776/\{W1230, DOC183, F1002\}]. For the purpose of open access, the authors have applied a CC BY public copyright licence to any Author Accepted Manuscript version arising from this submission.

\printbibliography
\appendix

\section{}\label{sec:appendix_bootstrap_basics}

Here we prove auxiliary results in Section \ref {subsec:determ.bootstrap} concerning deterministic bootstrap percolation on the Hamming graph. Throughout this section, in order to have a clear distinction between vertices of base graphs and vertices of the product graph, we will write vertices of the product graph as column vectors, i.e., every vertex in $G= \square_{i=1}^n G^{(i)} $ is of the form $\vertexvecn{v} = \transpose{(v_1,\ldots,v_n)}$, where $v_i \in V^{(i)}$ for all $i \in [n]$. We need the following simple facts about $2$-neighbour bootstrap percolation on $G$, which are a direct consequence of \Cref{def:closure}.
\begin{lem}\label{lem:general_bootstrap_facts}
  For any $U,W \subseteq V$ we have
    \begin{enumerate}[label = (\arabic*)]
        \item $\closure{U} = \closure{\closure{U}}$;
        \item if $U \subseteq W$, then $\closure{U} \subseteq \closure{W}$;
        \item if $W \subseteq W_1 \cup W_2$ with $W_i \subseteq W$ for $i \in \{1,2\}$, then $\closure{U \cup W} = \closure{\closure{U \cup W_1} \cup W_2}= \closure{\closure{U \cup W_2} \cup W_1}$. \label{lem:general_bootstrap_facts:third_point}
    \end{enumerate}
\end{lem}

We continue by showing that \projections{} of $G$ are closed.
\begin{proof}[Proof of Lemma \ref{lem:projclosed}]
   Let $\subcube{}$ be an arbitrary projection of $G$. Assume that $\subcube{}$ is not closed, i.e., $\closure{\subcube{}} \setminus \subcube{} \neq \emptyset$. Then there exists a vertex $\vertexvecn{v} = \transpose{(v_1,\ldots,v_n)} \notin \subcube{}$ which has at least two neighbours in $\subcube{}$. Pick any two of those neighbours and denote them by $\vertexvecn{w}_{1}= \transpose{(w_{1,1},\ldots,w_{1,n})}$ and $\vertexvecn{w}_{2}=\transpose{(w_{2,1},\ldots,w_{2,n})}$. As $\vertexvecn{v} \notin \subcube{}$ and $\{\vertexvecn{w}_{1},\vertexvecn{w}_{2} \} \subseteq \subcube{}$, there exists a coordinate $i \in [n]$ such that $v_i \neq w_{1,i} = w_{2,i}$. Hence it follows from $\dist{\vertexvecn{v}}{\vertexvecn{w}_{1}} = \dist{\vertexvecn{v}}{\vertexvecn{w}_{2}} = 1$ that $ w_{1,j} = v_j = w_{2,j}$ holds for all $j \in [n]\setminus i$. Thus, $w_{1,j} = w_{2,j}$ for all $j \in [n]$, implying $\vertexvecn{w}_{1} = \vertexvecn{w}_{2}$ and thus that $\vertexvecn{v}$ has only a single neighbour in $\subcube{}$ which is a contradiction.  
\end{proof}

To prove \Cref{lem:projections_span_projection} we first need to understand the simplest case, where one of the \projections{} is a single vertex. 
\begin{lem}\label{lem:proj_and_vert_span_proj}
    Let $\subcube{}$ be a \projection{} of $G$ and $\vertexvecn{v} \in V(G)$ be such that  $\dist{\vertexvecn{v}}{\subcube{}} \leq 2$. Then, $\closure{\subcube{} \cup \vertexvecn{v}}$ is a \projection{} of $G$. 
\end{lem}
\begin{proof}
Let $\vertexvecn{v} = \transpose{(v_1,\ldots,v_n)}$ and let $U = \square_{i=1}^n U^{(i)}$ be a \projection{} of $G$. Note that the case $\dist{\vertexvecn{v}}{U} =0$ is trivial, because if so, then $\vertexvecn{v} \in U$, which implies that $\closure{U \cup \vertexvecn{v}} = U$.

We continue with the case $\dist{\vertexvecn{v}}{U} =1$ and assume w.l.o.g. that $v_1 \notin U^{(1)}$ and $v_i \in U^{(i)}$ for $i \in [n]\setminus \{1\}$. Let us define for all $j \in [n]$,
\begin{align*}
   W_j := \square_{i=1}^n W_j^{(i)}, 
\end{align*} 
with $W_j^{(1)} := K_k $, $W_j^{(i)} := U^{(i)}$ for $2 \leq i \leq j $ and $ W_j^{(i)} := \{v_i\}$ for $i >j$. We will show by induction that for all $j \in [n]$
\begin{align}\label{eq:lem:proj_and_vert_span_proj_inductionhyp}
    W_j \subseteq \closure{U \cup \vertexvecn{v}}.
\end{align} 
To show the induction base for $j=1$, let $\vertexvecn{w} := \transpose{(w_1,v_2,v_3,\ldots,v_n)}$ be an arbitrary vertex in $W_1$. It clearly holds that $\dist{\vertexvecn{w}}{\vertexvecn{v}} = \dist{\vertexvecn{w}}{U} = 1$, which implies $\vertexvecn{w} \in \closure{U \cup \vertexvecn{v}}$ and thus $W_1 \subseteq \closure{U \cup \vertexvecn{v}}$. Proceeding with the induction hypothesis, suppose \eqref{eq:lem:proj_and_vert_span_proj_inductionhyp} holds for all $j \in [m-1]$ for some $m \in [n]$ and let $\vertexvecn{w} = \transpose{(w_1,\ldots,w_{m},v_{m+1},\ldots,v_n)} \in W_{m}$. Assume w.l.o.g. that $w_{m} \neq v_{m}$, as otherwise $\vertexvecn{w} \in W_{m-1} \subseteq \closure{U \cup \vertexvecn{v}}$. Let us consider the two vertices 
$$\vertexvecn{w}' = \transpose{(w_1,\ldots,w_{m-1},v_{m},v_{m+1},\ldots,v_n)}\ \text{ and } \ \vertexvecn{w}'' = \transpose{(v_1,w_2,w_3,\ldots,w_{m},v_{m+1},\ldots,v_n)}.$$ By construction it holds that $\vertexvecn{w}' \in W_{m-1}$ and $\vertexvecn{w}'' \in U$ and thus it follows from the induction hypothesis that 
$\{\vertexvecn{w}',\vertexvecn{w}''\} \subseteq \closure{U \cup \vertexvecn{v}}$. Furthermore, it holds that $\dist{\vertexvecn{w}}{\vertexvecn{w}'} = \dist{\vertexvecn{w}}{\vertexvecn{w}''}=1$, implying $\vertexvecn{w} \in \closure{\vertexvecn{w}' \cup \vertexvecn{w}''} \subseteq \closure{U \cup \vertexvecn{v}}$ and thus $W_m \subseteq  \closure{U \cup \vertexvecn{v}}$.

It follows that $W_n  \subseteq \closure{U \cup \vertexvecn{v}}$. Now note that 
$W_n = K_k \square \left( \square_{i=2}^n U^{(i)}\right)$ is a \projection{} of $G$ that contains both $U$ and $\vertexvecn{v}.$ Using \Cref{lem:projclosed} we obtain 
\begin{align*}
    \closure{U \cup \vertexvecn{v}} \subseteq \closure{W_n} = W_n,
\end{align*} as required. 

We move on to the case $\dist{\vertexvecn{v}}{U} = 2$. It follows that there exists a vertex $\vertexvecn{v}'$ satisfying $\dist{\vertexvecn{v}'}{\vertexvecn{v}} = \dist{\vertexvecn{v}'}{U} = 1$. It thus clearly holds that $\vertexvecn{v}' \in \closure{U \cup \vertexvecn{v}}$. By the previous case $\closure{U \cup \vertexvecn{v}'}$ is a \projection{} $\subcube{}$ of $G$. Furthermore, as $\dist{\vertexvecn{v}}{\subcube{}} \leq \dist{\vertexvecn{v}}{\vertexvecn{v}'} = 1$, it also follows from the previous case that $\closure{ \subcube{} \cup \vertexvecn{v} }$ is a \projection{} $\subcube{}'$ of $G$. Thus, using \crefitemintheorem{lem:general_bootstrap_facts}{lem:general_bootstrap_facts:third_point}, we obtain 
\begin{align*}
   \closure{U \cup \vertexvecn{v}} = \closure{U \cup \vertexvecn{v}' \cup \vertexvecn{v}} = \closure{\closure{U \cup \vertexvecn{v}'} \cup \vertexvecn{v}} = \closure{\subcube{} \cup \vertexvecn{v}} = \subcube{}',
\end{align*} completing the proof. 
\end{proof}

Having understood how a single vertex interacts with a \projection{}, we are ready to prove \Cref{lem:projections_span_projection}.
\begin{proof}[Proof of \Cref{lem:projections_span_projection}]
    Let $\subcube{}',\subcube{}''$ be two \projections{} of $G$ satisfying $d \coloneqq \dist{\subcube{}'}{\subcube{}''} \leq 2$. Let us set $m \coloneqq |\subcube{}''| $. 
    Clearly, there exists a vertex $\vertexvecn{v} \in \subcube{}''$ satisfying $\dist{\subcube{}'}{\vertexvecn{v}} \leq 2$. Furthermore, as \projections{} of $G$ are connected, we can find an order $(\vertexvecn{v}_1, \ldots ,\vertexvecn{v}_{m})$ of the vertices of $\subcube{}''$ such that $\vertexvecn{v}_1 = \vertexvecn{v}$ and $\dist{\{\vertexvecn{v}_1, \ldots ,\vertexvecn{v}_i\}}{\vertexvecn{v}_{i+1}} \leq 1$, for all $i \in [m-1]$. 
    
    Let us define $U_0$ := $\subcube{}'$ and for each $i \in [m]$,
    \begin{align*}
        U_i := \closure{U_{i-1} \cup \vertexvecn{v}_i}.
    \end{align*} We will show by induction that for all $i \in [m]$, $U_i$ is a \projection{} and that 
    \begin{align}\label{eq:lem:projections_span_projection:inductionhyp}
        \closure{\subcube{}' \cup \subcube{}''} = \closure{U_i \cup \subcube{}''}.
    \end{align}
    By \Cref{lem:proj_and_vert_span_proj}, $U_1 \coloneqq \closure{\subcube{}'\cup \vertexvecn{v}_1}$ is a \projection{}. Furthermore, by \crefitemintheorem{lem:general_bootstrap_facts}{lem:general_bootstrap_facts:third_point} it holds that $$\closure{\subcube{}' \cup \subcube{}''} = \closure{\closure{U_0 \cup \vertexvecn{v}_1}\cup \subcube{}''} = \closure{U_1 \cup \subcube{}''},$$ proving the induction base. 
    
    Proceeding with the induction hypothesis, assume now that \eqref{eq:lem:projections_span_projection:inductionhyp} holds and that $U_j$ is a \projection{}for all $j \in [i-1]$ for some $i \in [m]$. As $\dist{\{\vertexvecn{v}_1, \ldots ,\vertexvecn{v}_{i-1}\}}{\vertexvecn{v}_{i}} \leq 1$, it follows that also $\dist{U_{i-1}}{\vertexvecn{v}_{i}} \leq 1$ and thus by \Cref{lem:proj_and_vert_span_proj} that $U_{i} \coloneqq \closure{U_{i-1} \cup \vertexvecn{v}_i}$ is a \projection{}. Furthermore, using \eqref{eq:lem:projections_span_projection:inductionhyp} and \Cref{lem:general_bootstrap_facts} it follows that 
    $$\closure{\subcube{}' \cup \subcube{}''} = \closure{U_{i-1} \cup \subcube{}''} = \closure{\closure{U_{i-1} \cup \vertexvecn{v}_{i}}\cup \subcube{}''} = \closure{U_{i} \cup \subcube{}''},$$ completing the induction. Hence, using $\subcube{}'' \subseteq U_{m}$ and \eqref{eq:lem:projections_span_projection:inductionhyp} for $i = m$, there exists a \projection{} $\subcube{}$ of $G$, such that $\closure{\subcube{}' \cup \subcube{}''} = \closure{ U_{m} \cup \subcube{}''} =\closure{U_m} = U_{m} = \subcube{}$. 
    
    It remains to prove the claimed upper bound on $\dimension{\subcube{}}$. To that end, setting $d \coloneqq \dist{\subcube{}'}{\subcube{}''}$, we note that the graph $\subcube{}' \cup \subcube{}''$ can have at most $\dimension{\subcube{}'}+\dimension{\subcube{}''}+d$ non-trivial coordinates. Thus, $\subcube{}' \cup \subcube{}''$ is contained in a \projection{} $W$ satisfying $\dimension{W} \leq \dimension{\subcube{}'}+\dimension{\subcube{}''}+d$. However, this also implies that $\closure{\subcube{}' \cup \subcube{}'' } \subseteq \closure{W} = W$ and thus finally,
    $\dimension{\closure{\subcube{}' \cup \subcube{}'' }} \leq \dimension{\subcube{}'}+\dimension{\subcube{}''}+d$.
\end{proof}

\begin{proof}[Proof of Lemma \ref{lem:seqspanningspansprojection}]
   Set  $ u \coloneqq |U| -1$ and let $(\vertexvecn{v}_0,\ldots, \vertexvecn{v}_{u})$ be an ordering of the vertices of $U$ such that Property \eqref{Prop:sequentiallyspanning} is satisfied. The statement clearly holds in the case $u=2$, so let $u >2$ and assume by induction that the statement holds for all sets $U'$ satisfying $|U'| <|U|=u+1$. The set $U \setminus \{\vertexvecn{v}_u\}$ clearly also satisfies Property \eqref{Prop:sequentiallyspanning} and thus $\closure{U\setminus \{\vertexvecn{v}_u}\}$ spans a \projection{} $\subcube{}$ of dimension $2(u-1)$. It follows from $\dist{\subcube}{\vertexvecn{v}_u} =2$ that $\subcube{} \cup \vertexvecn{v}_u$ has to consist of at least $2u$ non-trivial coordinates. Using \Cref{lem:projections_span_projection} we know that $\subcube{} \cup \vertexvecn{v}_u$ has to span some \projection{} $\subcube{}'$. However, by the previous argument it follows that  $\dimension{\subcube{}'} \geq 2u$ and on the other hand, again by \Cref{lem:projections_span_projection} we know that $\dimension{\subcube{}'} \leq \dimension{\subcube{}} +\dimension{\vertexvecn{v}_u}+\dist{\subcube{}}{\vertexvecn{v}_u} = 2(u-1)+0+2 = 2u$, completing the proof. 
\end{proof}

\section{}\label{sec:appendix_bounds}
Here we give all the details of the proof of \Cref{lem:bound:auxfunc}.
\begin{proof}[Proof of \crefitemintheorem{lem:bound:auxfunc}{item:lem:bound:auxfunc:a}]
Observe first that by  construction,  for all $1 \leq j \leq \rounddown{\frac{n}{2}}-1$,  the functions $\constinlemma{j}$ satisfy
\begin{align}
    \frac{\constinlemma{j+1}}{\constinlemma{j}}\geq 3. \label{eq:ratio_constinlemma}
\end{align}
By case distinction, we will prove that  $\auxfunc{\cdot}{i}{s}$ is decreasing in the first coordinate: Let $(m,i,s) \in \N^3$. 
\begin{itemize}
    \item[\textbf{Case 1:}] $2(m-i+2) < s$. \\
    Plugging in the corresponding values for $\auxfunc{\cdot}{\cdot}{\cdot}$ we obtain together with \eqref{eq:ratio_constinlemma} that  
    \begin{align*}
    \frac{\auxfunc{m+1}{i}{s}}{\auxfunc{m}{i}{s}} &= \frac{ \prod_{j = m+2}^{2m-i+3} \left(\frac{2m-i+4-j}{\constinlemma{j}} \right)^2\prod_{j = 2m-i+4}^{s+i-1} \constinlemma{j}^{-1}}{ \prod_{j = m+1}^{2m-i+1} \left(\frac{2m-i+2-j}{\constinlemma{j}} \right)^2 \prod_{j = 2m-i+2}^{s+i-1} \constinlemma{j}^{-1}} \\
    &= \frac{ \prod_{j = m+2}^{2m-i+3} \left(2m-i+4-j \right)^2}{\prod_{j = m+1}^{2m-i+1} \left(2m-i+2-j \right)^2} \frac{\prod_{j = m+1}^{2m-i+1} \constinlemma{j}^2 \prod_{j = 2m-i+2}^{s+i-1} \constinlemma{j} }{\prod_{j = m+2}^{2m-i+3} \constinlemma{j}^2 \prod_{j = 2m-i+4}^{s+i-1} \constinlemma{j} } \\
    &= \frac{(m-i+2)^2\constinlemma{m+1}^2}{\constinlemma{2m-i+2}\constinlemma{2m-i+3}} 
    \\ \overset{\eqref{eq:ratio_constinlemma}}&{\leq}\frac{(m-i+2)^2}{3^{2m-2i+3}} \leq 1. 
\end{align*}
\item[\textbf{Case 2:}] $2(m-i+2)=  s.$ \\
In this case,  $\rounddown{\frac{s}{2}} = (m-i+2)$ and $s$ is even, yielding
\begin{align*}
    \frac{\auxfunc{m+1}{i}{s}}{\auxfunc{m}{i}{s}} &= \frac{\prod_{j = m+2}^{m+1+\rounddown{\frac{s}{2}} }\left(\frac{2m-i+4-j}{\constinlemma{j}} \right)^2 }{  \prod_{j = m+1}^{2m-i+1} \left(\frac{2m-i+2-j}{\constinlemma{j}} \right)^2 \prod_{j = 2m-i+2}^{s+i-1} \constinlemma{j}^{-1}} \\
    &= \frac{\prod_{j = m+2}^{2m-i+3}\left(\frac{2m-i+4-j}{\constinlemma{j}} \right)^2 }{  \prod_{j = m+1}^{2m-i+1} \left(\frac{2m-i+2-j}{\constinlemma{j}} \right)^2 \prod_{j = 2m-i+2}^{2m-i+3} \constinlemma{j}^{-1}}\\
    &= \frac{(m-i+2)^2\constinlemma{m+1}^2}{\constinlemma{2m-i+2}\constinlemma{2m-i+3}} \\ \overset{\eqref{eq:ratio_constinlemma}}&{\leq} \frac{(m-i+2)^2}{3^{2m-2i+3}}  \leq 1.
\end{align*}
\item[\textbf{Case 3:}] $2(m-i+2) -1=  s.$ \\
Note that in this case $\rounddown{\frac{s}{2}} = (m-i+1)$  and furthermore $s$ is odd. Thus, using  \eqref{eq:ratio_constinlemma} we obtain
\begin{align*}
    \frac{\auxfunc{m+1}{i}{s}}{\auxfunc{m}{i}{s}} 
    &= \frac{\prod_{j = m+2}^{m+1+\rounddown{\frac{s}{2}} }\left(\frac{2m-i+4-j}{\constinlemma{j}} \right)^2 \left(\frac{2(m-\rounddown{\frac{s}{2}}-i+2)+1}{\constinlemma{m +\rounddown{\frac{s}{2}}+2}}\right)}{  \prod_{j = m+1}^{2m-i+1} \left(\frac{2m-i+2-j}{\constinlemma{j}} \right)^2 \prod_{j = 2m-i+2}^{s+i-1} \constinlemma{j}^{-1}} \\ &= \frac{\prod_{j = m+2}^{2m-i+2}\left(\frac{2m-i+4-j}{\constinlemma{j}} \right)^2 \left(\frac{3}{\constinlemma{2m -i+3}}\right)}{  \prod_{j = m+1}^{2m-i+1} \left(\frac{2m-i+2-j}{\constinlemma{j}} \right)^2 \prod_{j = 2m-i+2}^{2m-i+2} \constinlemma{j}^{-1}}\\
    &= \frac{3(m-i+2)^2\constinlemma{m+1}^2}{\constinlemma{2m-i+2}\constinlemma{2m-i+3}} \\ \overset{\eqref{eq:ratio_constinlemma}}&{\leq} \frac{(m-i+2)^2}{3^{2m-2i+2}} \leq 1.
\end{align*}
\item[\textbf{Case 4:}] $2(m-i+1) \geq s$. \\
In this case we obtain
\begin{align*}
    \frac{\auxfunc{m+1}{i}{s}}{\auxfunc{m}{i}{s}} &\leq \frac{\left( \prod_{j = m+2}^{m+1+\rounddown{\frac{s}{2}}} \left(\frac{2m-i+4-j}{\constinlemma{j}} \right)^2\right) \left(\indicator{s \text{ even}}+ \indicator{s \text{ odd}} \frac{2(m-\rounddown{\frac{s}{2}}-i+2)+1}{\constinlemma{m +\rounddown{\frac{s}{2}}+2}}\right)}{\left( \prod_{j = m+1}^{m+\rounddown{\frac{s}{2}}} \left(\frac{2m-i+2-j}{\constinlemma{j}} \right)^2\right)\left(\indicator{s \text{ even}}+ \indicator{s \text{ odd}} \frac{2(m-\rounddown{\frac{s}{2}}-i+1)+1}{\constinlemma{m +\rounddown{\frac{s}{2}}+1}}\right) } \\
    &= \frac{(m-i+2)^2 \constinlemma{m+1}^2 }{\left(m-i+2-\rounddown{\frac{s}{2}}\right)^2 \constinlemma{m+1+\rounddown{\frac{s}{2}}}^2  } \frac{\left(\indicator{s \text{ even}}+ \indicator{s \text{ odd}} \frac{2(m-\rounddown{\frac{s}{2}}-i+2)+1}{\constinlemma{m +\rounddown{\frac{s}{2}}+2}}\right)}{\left(\indicator{s \text{ even}}+ \indicator{s \text{ odd}} \frac{2(m-\rounddown{\frac{s}{2}}-i+1)+1}{\constinlemma{m +\rounddown{\frac{s}{2}}+1}}\right)}\\
    \overset{\eqref{eq:ratio_constinlemma}}&{\leq} \frac{\left(\rounddown{\frac{s}{2}}+1\right)^2}{9^{\rounddown{\frac{s}{2}}}} \leq 1,
\end{align*} 
\end{itemize} 
concluding the case distinction and thus showing that $\auxfunc{\cdot}{i}{s}$ is increasing in the first coordinate. 
\end{proof}

\begin{proof}[Proof of \crefitemintheorem{lem:bound:auxfunc}{item:lem:bound:auxfunc:b}]
Let $m \geq i-1 $ and $s\geq 1$. We will split the sum into two parts and show that 
\begin{align}
    \sum_{j= m+1}^{\ell}  \frac{2(j+1-i) \auxfunc{j}{i+1}{s-1}}{\constinlemma{j}}  &\leq 6 \auxfunc{m}{i}{s} \quad \text{and} \label{eq1:lem:bound:auxfunc} \\
     \sum_{j= m+1}^{\ell}   \frac{(j-i)^2  \auxfunc{j}{i+2}{s-2}}{8{\constinlemma{j}}^2}\indicator{s\geq 2} &\leq 2 \auxfunc{m}{i}{s}.\label{eq2:lem:bound:auxfunc}
\end{align}
In order to show \eqref{eq1:lem:bound:auxfunc}, let  us consider ratios of consecutive terms. 
 It follows from \eqref{eq:ratio_constinlemma} and \crefitemintheorem{lem:bound:auxfunc}{item:lem:bound:auxfunc:a} for all $j \geq i$ that
\begin{align*}
    \frac{2(j+2-i) \auxfunc{j+1}{i+1}{s-1}\constinlemma{j}}{2(j+1-i) \auxfunc{j}{i+1}{s-1}\constinlemma{j+1}} \leq \frac{j+2-i}{3(j+1-i)} \leq \frac{2}{3}, 
\end{align*} implying, together with the fact that $m+1 \geq i$, that
\begin{align*}
    \sum_{\specialindex= m+1}^{\ell}  \frac{2(\specialindex+1-i) \auxfunc{\specialindex}{i+1}{s-1}}{\constinlemma{\specialindex}} &\leq \frac{2(m-i+2)\auxfunc{m+1}{i+1}{s-1}}{\constinlemma{m+1}} \sum_{\specialindex= 0}^{\ell-m-1} \left(\frac{2}{3}\right)^j \\
    &\leq \frac{6(m-i+2)\auxfunc{m+1}{i+1}{s-1}}{\constinlemma{m+1}}. 
\end{align*} 
Thus, in order to verify \eqref{eq1:lem:bound:auxfunc} it suffices to show that
\begin{align*}
     \frac{\auxfunc{m+1}{i+1}{s-1}}{\auxfunc{m}{i}{s}} \leq  \frac{\constinlemma{m+1}}{(m-i+2)}.
\end{align*}

We proceed again with a case distinction. 
\begin{itemize}
    \item[\textbf{Case 1:}] $2(m-i+1) < s-1$. \\
    It follows from the definition of $\Psi$ and \eqref{eq:ratio_constinlemma} that 
    \begin{align*}
        \frac{\auxfunc{m+1}{i+1}{s-1}}{\auxfunc{m}{i}{s}} = \frac{\prod_{j = m+2}^{2m-i+2} \left(\frac{2m-i+3-j}{\constinlemma{j}} \right)^2\prod_{j = 2m-i+3}^{s+i-1} \constinlemma{j}^{-1}}{\prod_{j = m+1}^{2m-i+1} \left(\frac{2m-i+2-j}{\constinlemma{j}} \right)^2\prod_{j = 2m-i+2}^{s+i-1} \constinlemma{j}^{-1}} = \frac{ \constinlemma{m+1}^2}{\constinlemma{2m-i+2}} \overset{\eqref{eq:ratio_constinlemma}}{\leq} \frac{\constinlemma{m+1}}{(m-i+2)}. 
    \end{align*}
    \item[\textbf{Case 2:}] $2(m-i+1) = s-1$. \\
    In this case $s-1$ is even and we thus obtain
    \begin{align*}
         \frac{\auxfunc{m+1}{i+1}{s-1}}{\auxfunc{m}{i}{s}} &= \frac{\prod_{j = m+2}^{m+1+\rounddown{\frac{s-1}{2}}} \left(\frac{2m-i+3-j}{\constinlemma{j}} \right)^2  }{\prod_{j = m+1}^{2m-i+1} \left(\frac{2m-i+2-j}{\constinlemma{j}} \right)^2\prod_{j = 2m-i+2}^{s+i-1} \constinlemma{j}^{-1}} \\ &= \frac{\prod_{j = m+2}^{2m-i+2} \left(\frac{2m-i+3-j}{\constinlemma{j}} \right)^2  }{\prod_{j = m+1}^{2m-i+1} \left(\frac{2m-i+2-j}{\constinlemma{j}} \right)^2\prod_{j = 2m-i+2}^{2m-i+2} \constinlemma{j}^{-1}} \\
         &= \frac{ \constinlemma{m+1}^2}{\constinlemma{2m-i+2}} \overset{\eqref{eq:ratio_constinlemma}}{\leq} \frac{\constinlemma{m+1}}{(m-i+2)}. 
    \end{align*}

    \item[\textbf{Case 3:}] $2(m-i+1) > s-1$ and $s$ is even. \\
    In this case we obtain
    \begin{align*}
         \frac{\auxfunc{m+1}{i+1}{s-1}}{\auxfunc{m}{i}{s}} &= \frac{\prod_{j = m+2}^{m+1+\rounddown{\frac{s-1}{2}}} \left(\frac{2m-i+3-j}{\constinlemma{j}} \right)^2 \left(\indicator{s-1 \text{ even}}+ \indicator{s-1 \text{ odd}} \frac{2(m-\rounddown{\frac{s-1}{2}}-i+1)+1}{\constinlemma{m +1+\rounddown{\frac{s-1}{2}}+1}}\right)}{\prod_{j = m+1}^{m+\rounddown{\frac{s}{2}}} \left(\frac{2m-i+2-j}{\constinlemma{j}} \right)^2 \left(\indicator{s \text{ even}}+ \indicator{s \text{ odd}} \frac{2(m-\rounddown{\frac{s}{2}}-i+1)+1}{\constinlemma{m +\rounddown{\frac{s}{2}}+1}}\right)} \\
         &= \frac{\prod_{j = m+2}^{m+\frac{s}{2}} \left(\frac{2m-i+3-j}{\constinlemma{j}} \right)^2 \left(\frac{2(m-\frac{s}{2}-i)+1}{\constinlemma{m +\frac{s}{2}+1}}\right)}{\prod_{j = m+1}^{m+\frac{s}{2}} \left(\frac{2m-i+2-j}{\constinlemma{j}} \right)^2 } \\
         &= \frac{\left(2(m-i-\frac{s}{2})+1 \right)\constinlemma{m+1}^2}{\left(m-i+2-\frac{s}{2}\right)^2\constinlemma{m+\frac{s}{2}+1}}\overset{\eqref{eq:ratio_constinlemma}}{\leq} \frac{\constinlemma{m+1}}{(m-i+2)}.
    \end{align*}

    \item[\textbf{Case 4:}] $2(m-i+1) > s-1$ and $s$ is odd. \\
    In this case we obtain
    \begin{align*}
         \frac{\auxfunc{m+1}{i+1}{s-1}}{\auxfunc{m}{i}{s}} &= \frac{\prod_{j = m+2}^{m+1+\rounddown{\frac{s-1}{2}}} \left(\frac{2m-i+3-j}{\constinlemma{j}} \right)^2 \left(\indicator{s-1 \text{ even}}+ \indicator{s-1 \text{ odd}} \frac{2(m-\rounddown{\frac{s-1}{2}}-i+1)+1}{\constinlemma{m +1+\rounddown{\frac{s-1}{2}}+1}}\right)}{\prod_{j = m+1}^{m+\rounddown{\frac{s}{2}}} \left(\frac{2m-i+2-j}{\constinlemma{j}} \right)^2 \left(\indicator{s \text{ even}}+ \indicator{s \text{ odd}} \frac{2(m-\rounddown{\frac{s}{2}}-i+1)+1}{\constinlemma{m +\rounddown{\frac{s}{2}}+1}}\right)} \\
         &=\frac{\prod_{j = m+2}^{m+1+\frac{s-1}{2}} \left(\frac{2m-i+3-j}{\constinlemma{j}} \right)^2 }{\prod_{j = m+1}^{m+\frac{s-1}{2}} \left(\frac{2m-i+2-j}{\constinlemma{j}} \right)^2 \left( \frac{2(m-\frac{s-1}{2}-i+1)+1}{\constinlemma{m+ \frac{s-1}{2}+1}}\right)} \\ &= \frac{\constinlemma{m+1}^2}{\left(2(m-i+2)-s\right)\constinlemma{m+1+\frac{s-1}{2}}} \\ \overset{\eqref{eq:ratio_constinlemma}}&{\leq} \frac{\constinlemma{m+1}}{(m-i+2)}.
    \end{align*} 
\end{itemize}
This completes the proof of \eqref{eq1:lem:bound:auxfunc}.

We will show \eqref{eq2:lem:bound:auxfunc} in a similar manner.
Observe first that by \eqref{eq:ratio_constinlemma} and the fact that $\auxfunc{\cdot}{i}{s}$ is decreasing in the first coordinate, it follows  that for all $j \geq i+1$
\begin{align*}
    \frac{8(j+1-i)^2  \auxfunc{j+1}{i+2}{s-2}{\constinlemma{j}}^2}{8 (j-i)^2\auxfunc{j}{i+2}{s-2} {\constinlemma{j+1}}^2} \leq \frac{(j+1-i)^2}{9(j-1)^2} \leq \frac{4}{9} \leq \frac{1}{2},
\end{align*} implying
\begin{align*}
     &\sum_{j= m+1}^{\ell}   \frac{(j-i)^2  \auxfunc{j}{i+2}{s-2}}{8{\constinlemma{j}}^2} \\
     &\leq 2\left( \indicator{m+1>i}\frac{(m-i+1)^2 \auxfunc{m+1}{i+2}{s-2}}{8 \constinlemma{m+1}^2} + \indicator{m+1 =i}\frac{\auxfunc{m+2}{i+2}{s-2}}{8 \constinlemma{m+2}^2}\right).
\end{align*}
Now, if $m+1=i$, it immediately follows that 
\begin{align*}
    \frac{\auxfunc{m+2}{i+2}{s-2}}{\auxfunc{m}{i}{s}} = \frac{\prod_{j = m+3}^{2m-i+3} \left(\frac{2m-i+4-j}{\constinlemma{j}} \right)^2\prod_{j = 2m-i+4}^{s+i-1} \constinlemma{j}^{-1}}{\prod_{j = m+1}^{2m-i+1} \left(\frac{2m-i+2-j}{\constinlemma{j}} \right)^2\prod_{j = 2m-i+2}^{s+i-1} \constinlemma{j}^{-1}} = \frac{\prod_{j = m+3}^{s+i-1} \constinlemma{j}^{-1}}{\prod_{j = m+1}^{s+i-1} \constinlemma{j}^{-1}} = \constinlemma{m+1}\constinlemma{m+2} \leq 8 \constinlemma{m+2}^2. 
\end{align*}

If $m+1>i$,  we will again show by a case distinction that 
\begin{align*}
     \frac{\auxfunc{m+1}{i+2}{s-2}}{\auxfunc{m}{i}{s}} \leq \frac{ \constinlemma{m+1}^2}{(m-i+1)^2},
\end{align*} which then implies \eqref{eq2:lem:bound:auxfunc}.
First, observe that 
$2(m-i+1) < s$ is equivalent to $2((m+1)-(i+2)+1) < s-2$. 
\begin{itemize}
    \item[\textbf{Case 1:}] $2(m-i+1) < s$. \\
     It follows from the definition of $\Psi$ that
\begin{align*}
    \frac{\auxfunc{m+1}{i+2}{s-2}}{\auxfunc{m}{i}{s}} = \frac{ \prod_{j = m+2}^{2m-i+1} \left(\frac{2m-i+2-j}{\constinlemma{j}} \right)^2\prod_{j = 2m-i+2}^{s+i-1} \constinlemma{j}^{-1}}{ \prod_{j = m+1}^{2m-i+1} \left(\frac{2m-i+2-j}{\constinlemma{j}} \right)^2\prod_{j = 2m-i+2}^{s+i-1} \constinlemma{j}^{-1}} = \frac{\constinlemma{m+1}^2}{(m-i+1)^2}.
\end{align*}

\item[\textbf{Case 2:}] $2(m-i+1) \geq s$. \\
     Similarly to the previous case, we obtain
\begin{align*}
    \frac{\auxfunc{m+1}{i+2}{s-2}}{\auxfunc{m}{i}{s}} &= \frac{  \prod_{j = m+2}^{m+1+\rounddown{\frac{s-2}{2}}} \left(\frac{2m-i+2-j}{\constinlemma{j}} \right)^2 \left(\indicator{s -2 \text{ even}}+ \indicator{s-2 \text{ odd}} \frac{2(m-\rounddown{\frac{s-2}{2}}-i)+1}{\constinlemma{m +\rounddown{\frac{s-2}{2}}+2}}\right)}{   \prod_{j = m+1}^{m+\rounddown{\frac{s}{2}}} \left(\frac{2m-i+2-j}{\constinlemma{j}} \right)^2 \left(\indicator{s \text{ even}}+ \indicator{s \text{ odd}} \frac{2(m-\rounddown{\frac{s}{2}}-i+1)+1}{\constinlemma{m +\rounddown{\frac{s}{2}}+1}}\right)} \\
    &= \frac{\constinlemma{m+1}^2 }{(m-i+1)^2} \left(\indicator{s \text{ even}} +\indicator{s \text{ odd}}\frac{\left(2(m-\rounddown{\frac{s-2}{2}}-i)+1\right) \constinlemma{m +\rounddown{\frac{s}{2}}+1}}{\left(2(m-\rounddown{\frac{s}{2}}-i+1)+1\right)\constinlemma{m +\rounddown{\frac{s-2}{2}}+2}}\right) \\
    &\leq  \frac{\constinlemma{m+1}^2 }{(m-i+1)^2}.
\end{align*} 
\end{itemize}
This completes the proof of \eqref{eq2:lem:bound:auxfunc}.

It follows from \eqref{eq1:lem:bound:auxfunc} and \eqref{eq2:lem:bound:auxfunc} that
\begin{align*}
    \sum_{j= m+1}^{\ell}  \frac{2(j+1-i) \auxfunc{j}{i+1}{s-1}}{\constinlemma{j}} + \frac{(j-i)^2  \auxfunc{j}{i+2}{s-2}}{8{\constinlemma{j}}^2} \leq 8 \auxfunc{m}{i}{s},
\end{align*} as claimed.  
\end{proof}
\end{document}